\documentclass[12pt]{amsart}
\usepackage{graphics, ulem}
\usepackage{graphicx,amssymb}
\usepackage{a4,pstricks,pst-node,amsmath,amssymb}
\usepackage[bookmarks=true]{hyperref}

\numberwithin{equation}{section}
\newcommand{\margincom}[1]{}

\newcommand{\bd}{\begin{displaymath}}
\newcommand{\be}{\begin{equation}}
\newcommand{\ba}{\begin{array}}
\newcommand{\ed}{\end{displaymath}}
\newcommand{\ee}{\end{equation}}
\newcommand{\ea}{\end{array}}

\newcommand{\Prob}{{\rm I\hspace{-0.8mm}P}}
\newcommand{\Exp}{{\rm I\hspace{-0.8mm}E}}
\newcommand{\indicator}[1]{{\mbox{\large\bf$1$}}_{#1}}

\def\N{\mathbb{N}}
\def\Z{\mathbb{Z}}
\def\R{\mathbb{R}}

\newcommand{\hh}{\mbox{{\bf (H)}}}
\newcommand{\h}{{\bf (H)}}
\newtheorem{theorem}{Theorem}[section]

\newtheorem{proposition}[theorem]{Proposition} 
\newtheorem{definition}[theorem]{Definition} 
\newtheorem{lemma}[theorem]{Lemma}
\newtheorem{corollary}[theorem]{Corollary} 
 
\newtheorem{remark}[theorem]{Remark}

\begin{document}
\title[Flux of TASEP with site disorder]{Quantitative estimates for the flux of TASEP with dilute site disorder}
\author{C. Bahadoran$^1$, T. Bodineau$^2$}
\thanks{We would like to thank A. Sznitman for very useful discussions. We acknowledge the supports of grants ANR-15-CE40-0020-02 and ANR-14-CE25-0011}
\maketitle
$$ \ba{l}
^1\,\mbox{\small Laboratoire de Math\'ematiques Blaise Pascal, Universit\'e Clermont Auvergne, 63178 Aubi\`ere, France} \\
\quad \mbox{\small e-mail:
bahadora@math.univ-bpclermont.fr}\\
^2\,\mbox{\small CMAP, Ecole Polytechnique, CNRS, Universit\'e Paris-Saclay, 91128 Palaiseau, France} \\
\quad \mbox{\small e-mail:
thierry.bodineau@polytechnique.edu}
\ea
$$
\begin{abstract}
We prove that the flux function of the totally asymmetric simple exclusion process (TASEP) with site disorder exhibits a flat segment for sufficiently dilute disorder. 
For high dilution, we obtain an accurate description of the flux.
The result is established under
a decay assumption of the maximum current in finite boxes, which is implied in particular by a sufficiently slow power tail assumption on the disorder distribution near its minimum.
To circumvent the absence of explicit invariant measures, we use an original renormalization procedure and some ideas inspired by homogenization.
\end{abstract}


\section{Introduction}
%
The flux function, also called current-density relation in traffic-flow physics \cite{css}, is the most fundamental object to describe the macroscopic behavior of driven lattice gases.
The paradigmatic model in this class is the totally asymmetric simple exclusion process (TASEP), where particles on the one-dimensional integer lattice hop to the right at unit rate and obey an exclusion rule.
Density $\rho\in(0,1]$ is the only conserved quantity and is associated locally with a 
flux (or current) that is defined as the amount of particles crossing a given site per unit time in a system with homogeneous density $\rho$. For TASEP, the flux function is explicitly given by
\be
\label{f0}
{\color{black} f_0} (\rho)=\rho(1-\rho).
\ee
In the hyperbolic scaling limit \cite{rez}, the empirical particle density field is governed by entropy solutions of the scalar conservation law
\be
\label{burgers_intro}
\partial_t\rho(t,x)+\partial_x  \, {\color{black} f_0}  \big( \rho(t,x) \big)=0,
\ee
with ${\color{black} f_0}$ given by \eqref{f0}. 
%
This kind of result can be extended to a variety of asymmetric models \cite{sep, rez2,bgrs4}, 
but  when the invariant measures are not explicit, little can be said about the flux function. 
Nevertheless, convexity or concavity can be obtained as a byproduct of the variational approach set up in \cite{sep0}, which applies to  totally asymmetric models with state-independent jump rates, like TASEP. However strict convexity or strict concavity, which are related to the absence of a phase transition, require new mathematical ideas to be derived for general driven dynamics.


\medskip

In disordered systems, a phase transition has been first proved  for nearest-neighbor asymmetric site-disordered zero-range processes (ZRP) and its signature is a constant flux on a density interval $[\rho_c,+\infty)$, where $\rho_c$ is the density of the maximal invariant measure. 
The invariant measures  of  the disordered ZRP are explicit so that the flux can be exactly computed and the phase transition precisely located.
The necessary and sufficient condition for the occurrence of a phase transition is a slow enough tail of  the jump rate distribution near its minimum value $r$. Microscopically,  phase transition takes the form of Bose-Einstein condensation \cite{ev}. 
In an infinite system with mean drift to the right, the excess mass is captured by the asymptotically slowest sites at $-\infty$. This was 
proven rigorously in \cite{afgl} for the totally asymmetric ZRP with constant jump rate with respect to the number of particles which is equivalent to TASEP with {\it particlewise} disorder \cite{ks}. This holds also  for nearest-neighbor ZRP with more general  jump rates \cite{bmrs} (see also \cite{fs} for partial results in higher dimension).
The TASEP picture can be interpreted as a traffic-flow model with slow and fast vehicles. The phase transition then occurs on a density interval $[0,\rho_c]$, where the flux is linear with a slope equal to the constant mean velocity of the system. This velocity is imposed by the slowest vehicles at $+\infty$. As one moves ahead, slower an slower vehicles are encountered, followed by a platoon of faster vehicles, and preceded by a gap before the next platoon \cite{ks}. 

\medskip

In this paper, we consider TASEP with i.i.d. site disorder such that the jump rate at each site has a random value
whose distribution is supported in an interval $[r,1]$, with $r\in(0,1)$. 
A flat piece in the flux was observed numerically by physicists  \cite{kru, tb, HS} and interpreted as
the occurence of a phase transition by several heuristic arguments.
Contrary to the disordered ZRP, the invariant measures are no longer explicit in the site-disordered TASEP, which makes the analysis of the flux more challenging.
Before commenting on the  flat segment in the flux, let us mention that 
 the existence of a hydrodynamic limit of the form \eqref{burgers_intro} for TASEP with i.i.d. site disorder was established in \cite{sep}, using last passage percolation (LPP) and variational coupling.
Consequently, the flux function was shown to be concave. More generally, the existence of a limit of the type \eqref{burgers_intro} was obtained in \cite{bgrs4}
for asymmetric attractive systems in ergodic environment, based on the study of invariant measures.
We refer also to \cite{CL, sch1, sch2, sz} for further rigorous results in a different class of disordered SEP.

 \medskip

Recently, Sly gave in \cite{sly}  a short and very elegant proof of the existence of a flat segment 
in the flux for TASEP with general rate distribution. 
The proof in \cite{sly} relies on a clever coupling implemented in the LPP formulation of the TASEP.
In this paper, we develop a different approach, announced in \cite{bb}, based on a renormalization 
method to obtain a precise information on the flux function and on the flat segment  at the price  of 
additional assumptions on the disorder distribution.
We focus on the case of dilute disorder which plays a key role in the physical literature  \cite{kru}
{\color{black} as a sharp transition occurs for any arbitrarily small amount of disorder}.
The jump rate at each site is chosen randomly, according to some dilution parameter $\varepsilon\in[0,1]$, 
so that a site is ``fast'' with probability $1-\varepsilon$,
in which case it has rate $1$, or ``slow'' with probability $\varepsilon$, in which case its rate has some distribution $Q$ with support $(r,1]$ for some $r\in(0,1)$. Under some assumption on the distribution $Q$, and 
for sufficiently diluted disorder, i.e. $\varepsilon$ small enough, we prove 
\textcolor{black}{(Theorem \ref{th_plateau})} 
the existence of a flat segment 
and determine the limiting size of this segment when 
$\varepsilon$ vanishes. Moreover, we prove 
\textcolor{black}{(Theorem \ref{th_dilute_limit})} 
the convergence of the whole flux function to  an explicit function, \textcolor{black}{which exhibits a sharp transition at $\varepsilon=0$.}
We stress the fact that  Sly's argument \cite{sly} does not require any assumption on $Q$ nor on the dilution of the disorder, however the control on the flux in \cite{sly} is less precise for small $\varepsilon$ \textcolor{black}{than the one we provide in Theorems \ref{th_plateau} and \ref{th_dilute_limit}. 
The reason for this is that our renormalization approach is a perturbation of $\varepsilon=0$, while Sly's approach involves a comparison with the homogeneous TASEP corresponding to $\varepsilon=1$. It follows that 
Sly's estimate on the size of the flat segment, is (unlike ours) not optimal for small $\varepsilon$.
}

\medskip

The physical interpretation of the flat segment in the flux \cite{kru} is 
the emergence at different scales of atypical disorder slowing down the particles and leading to traffic jams.
As one moves ahead along the disorder, slowest and slowest regions are encountered, with larger and larger stretches of sites with the minimal rate $r$ (or near this minimal rate). Locally a slow stretch of environment inside a typical region is expected to create a
picture similar to the {\it slow bond} TASEP introduced in \cite{jl}. It is known that a slow bond with an arbitrarily small blockage \cite{bss} restricts the local current. On the hydrodynamic scale \cite{sepslow}, this creates a traffic jam
with a high density of queuing vehicles to the left and a low density to the right, that is an antishock for Burgers' equation.
Renormalization turns the problem into a hierarchy of slow-bond like pictures, where at each scale, the difference between 
the ``typically fast'' and ``atypically slow'' region becomes smaller and smaller.
Slower jams will gradually absorb faster ones so that one expects to see a succession of mesoscopically growing shocks and antishocks.
Some results in this direction were obtained in \cite{GKS} in the case of particle disorder.
Even though, a single slow bond induces a phase transition, it is not clear if the transition will remain in presence of disorder or if the randomness rounds it off  as in equilibrium systems \cite{AW}.

\medskip

Renormalization is often key to analyze multi-scale phenomena in disordered systems; we refer to \cite{Sznitman, vares} for a general overview.
Our renormalization scheme  controls rigorously the multi-scale   slow bond picture described in the  paragraph above. 
A major difficulty compared to the single slow bond is that as one moves to larger scales, the typical maximum current 
associated with a given scale and the maximum current associated with the rare slow regions occurring 
at the same scale converge 
 to the same value $r/4$ (with $r$ the minimal value of the jump rates).
Thus a delicate issue is to show that this small current  difference exceeds the typical order of fluctuations at each scale, so that the slow-bond picture remains  valid at all scales.
To quantify this difference, we rely on an assumption on the decay of the maximum current in a finite box \eqref{condition_h_2}.
This assumption is satisfied  under a condition on the tail of the rate distribution $Q$ near its minimum $r$
 (see Lemma \ref{lemma_assumption}). 
Heuristics suggest that this assumption should be always valid although we have not been able
so far to prove this conjecture.

\medskip

We achieve our renormalization scheme by formulating the problem in wedge LPP framework with columnar disorder and exponential random variables.
In the LPP framework, the phase transition takes the form of a {\it pinning} transition for the optimal path \cite{khh}:
the path gets a better reward  from vertical portions along slow parts of the disorder.
The core of this approach is to obtain a recursion between mean passage times at two successive scales.
Like many shape theorems \cite{mar}, our results partially extends to LPP with more general distributions. 

\medskip

\textcolor{black}{
Another interpretation of our renormalization scheme (see \cite[Section 3.3.2]{bb}) is that it consists in a hierarchy of homogenization problems for scalar conservation laws which approximate the particle system in  blocks of mesoscopic size.
As explained in \cite{bb}, the homogenization of a one-dimensional scalar conservation law with a ``fast'' flux and a ``slow'' flux is easily seen to produce a flat segment as long as the fluxes in each block are bell-shaped (but not necessarily concave).  With this bell-shape assumption on the flux function, the emergence of antishocks, previously established in \cite{sepslow} for TASEP with a slow bond, was shown (see \cite{ba}) to hold in more general asymmetric models with {\it single} localized blockage, {\it even} in the absence of mapping on a percolation problem. Therefore, our renormalization picture {\it suggests} that the emergence of the flat segment should be true for such exclusion-like models with bell-shaped flux. However, we are currently far from being able to implement these ideas microscopically in the absence of a LPP representation. The main reason is that the latter yields fluctuation estimates on the current, which makes the slow-bond (or homogenization) picture effective even when the difference between slow and fast cells tends to zero. 
{\color{black}
Even though the implementation of the renormalization is model dependent, we stress the fact that our renormalization strategy should be useful to study other dynamics in random media. In particular, it was 
implemented in \cite{bt} to control the velocity of interfaces moving in a disordered environnement. 
}
}

\medskip

The paper is organized as follows. In Section \ref{sec:results}, we set up the notation and state our main result.
In Section \ref{sec:lpp}, we formulate the problem in the last passage percolation framework and introduce the reference flux and the passage time functions. In Section \ref{sec:renorm}, we introduce the renormalization procedure and describe the main steps of the proof.
In Section \ref{subsec:outline}, we prove a recurrence which links the passage time bounds of two successive scales. This is the heart of the renormalization argument.
In Section \ref{sec:other_proofs}, we study this recurrence in detail and show that it propagates the bounds we need from one scale to another.
In Section \ref{sec: prop passage time}, we establish an important fluctuation estimate needed in Section \ref{subsec:outline}. Finally, the proofs of our main theorems are completed in Section \ref{sec:completion}.

\section{Notation and results}\label{sec:results}

\subsection{TASEP with site disorder}
Let $\N:=\{0,1,\ldots\}$ (resp. $\N^*:=\{1,2,\ldots\}$) be the set of nonnegative (resp. positive) integers.
The disorder is modeled by  $\alpha=(\alpha(x):\,x\in\Z)\in{\bf A}:=[0,1]^\Z$, \textcolor{black}{an i.i.d.} sequence of positive bounded random variables. The precise distribution of $\alpha$ will be defined in 
Section \ref{sec: results}.
For a given realization of $\alpha$, we consider the TASEP on $\Z$ with site disorder $\alpha$. 
The dynamics is defined as follows.
A site $x$ is occupied by at most one particle which may jump with rate $\alpha(x)$ to site $x+1$ if it is empty.
A particle configuration on $\Z$ is of the form $\eta=(\eta(x):\,x\in\Z)$, where for $x\in\Z$, 
$\eta(x)\in\{0,1\}$ is the number of particles at $x$. The state space is ${\bf X}:=\{0,1\}^\Z$. The generator of the process is given by
\be\label{generator}
L^\alpha\textcolor{black}{\varphi}(\eta)=\sum_{x\in\Z}\alpha(x)\eta(x)[1-\eta(x+1)]\left[
\textcolor{black}{\varphi}\left(\eta^{x,x+1}\right)-\textcolor{black}{\varphi}(\eta)
\right],
\ee
\textcolor{black}{for any function $\varphi$ on $\bf X$ depending on finitely many sites
(the set of such functions, called cylinder functions, is a core for the generator $L^\alpha$),}
where $\eta^{x,x+1}=\eta-\delta_x+\delta_{x+1}$ denotes the new configuration after a particle has jumped from $x$ to $x+1$, \textcolor{black}{and $\delta_x$ is the configuration that is empty outside site $x$ and has a particle at $x$.}
\medskip

\noindent
\textbf{Current and flux function.} 
%
%
%
%
%
%
The macroscopic flux function $f$ can be defined as follows.
\textcolor{black}{For $\eta\in\bf X$}, we denote by $J^\alpha_x(t,\textcolor{black}{\eta})$ the rightward current across site $x$ up to time $t$, that is
the number of jumps from $x$ to $x+1$ up to time $t$,  in the TASEP $(\eta_t^\alpha)_{t\geq 0}$ \textcolor{black}{with generator \eqref{generator}} starting from initial state $\textcolor{black}{\eta}$.
%
%
For $\rho\in[0,1]$, let $\eta^\rho$ be an initial particle configuration with  asymptotic particle density $\rho$ in the following sense:
\be
\label{uniform_profile}
\lim_{n\to\infty} \frac{1}{n} \sum_{x=0}^n \eta^\rho(x)
=
\rho
=\lim_{n\to\infty} \frac{1}{n} \sum_{x=-n}^0 \eta^\rho(x).
\ee
%
%
We then set
\be
\label{flux_current}
f(\rho) := \lim_{t\to\infty}\frac{1}{t}J^\alpha_x(t,\eta^\rho),
\ee
where the limit is understood in probability with respect to the law of the quenched process.
\textcolor{black}{It is indeed shown in \cite{sep} that the function $f$ in \eqref{flux_current} exists for almost every realization of the disorder $\alpha$, and does not depend on the latter, nor on the choice of  initial configurations $\eta^\rho$ satisfying \eqref{uniform_profile}}.
Other definitions of the flux and the proof of their equivalence with the above definition can be found in \cite{bb}.
%
%
%
%

It is shown in \cite{sep} that $f$ is a concave function, see \eqref{flux_legendre} below. 
It was conjectured in \cite{tb} that for i.i.d. disorder, the flux function $f$ exhibits a flat segment, that is an interval $[\rho_c,1-\rho_c]$ (with $0\leq\rho_c<1/2$) on which $f$ is constant (see Figure \ref{fig: flux truncated}). The proof of \cite{sly} uses a comparison with a homogeneous rate $r$ TASEP.
We introduce a different approach, based on renormalization and homogenization ideas, viewing the disordered model as a perturbation of a homogenous rate $1$
TASEP. This yields \textcolor{black}{(see Theorems \ref{th_plateau} and \ref{th_dilute_limit} below)} not only an independent proof of the existence of a flat segment, but also optimal estimates when the density of defects is small enough. 
%
%
%
%
%
\subsection{ The flux and flat segment for rare defects}
\label{sec: results}
From now on, we consider i.i.d. disorder such that the support of the distribution of $\alpha(x)$ is contained in $[r,1]$, where {\color{black} $r \in ]0,1[$} is  the infimum of this support.
Then, \textcolor{black}{as stated in the following proposition}, the flux is bounded from above by $r/4$.
%
%
%
\begin{proposition}
\label{prop_max}{black}
The maximum value of the flux function is given by 
$$\max_{\rho\in[0,1]}f(\rho) = r/4.$$
\end{proposition}
This result comes from the fact that the current of the disordered system is limited by atypical large 
stretches  with jump rates close to $r$. On these atypical regions, the system behaves as a homogeneous rate $r$ TASEP
which has maximum current $r/4$.
A detailed proof can be found in Appendix \ref{appendix_max}.
%
%
For our main results, we formulate additional assumptions on the distribution of the environment. 
We assume that the disorder is a perturbation of the homogeneous case with rate 1.
Let 
$Q$ be a probability measure on {\color{black}  $[r,1]$}, such that $r$ is the infimum of the support of $Q$. 
Given $\varepsilon\in(0,1)$ a ``small'' parameter, 
we define the distribution of $\alpha(x)$ by
\be\label{law_disorder}
Q_\varepsilon = (1-\varepsilon) \delta_1 + \varepsilon Q .
\ee
The law of $\alpha = (\alpha(x), x \in \Z)$ is the product measure with marginal $Q_\varepsilon$ at each site
$$
\mathcal P_\varepsilon(d\alpha):=\bigotimes_{x\in\Z}Q_\varepsilon[d\alpha(x)].
$$
Expectation with respect to $\mathcal P_\varepsilon$ is denoted by $\mathcal E_\varepsilon$. 
\textcolor{black}{The decomposition \eqref{law_disorder} has a natural interpretation if the support of $Q$ is bounded away from $1$, which, however, need not be assumed. Namely,}
each site is chosen independently at random to be, 
with probability $1-\varepsilon$, a ``fast'' site with normal rate $1$, \textcolor{black}{and} with probability $\varepsilon$ to be a ``defect'' with rate distribution $Q$. 
Thus $\varepsilon$ is the mean density of defects. 
For example if $Q = \delta_r$, then the defects are slow bonds with rate $r < 1$.\\ \\
Let us denote by $f_\varepsilon$ the flux function \eqref{flux_current} for the disorder distribution
{\color{black} $Q_\varepsilon$}.
%
%
\noindent
\textcolor{black}{
We 
%
then define {\color{black} the edge of the flat segment as}
\be
\label{def_critical_density}
\rho_c(\varepsilon):=\inf\left\{\rho\in\left[0,\frac{1}{2}\right]: 
\qquad
f_\varepsilon  \equiv  \frac{r}{4}\mbox{ on }[\rho,1-\rho]\right\}.
\ee
It follows from Proposition \ref{prop_max} that $\rho_c(\varepsilon)\leq 1/2$. 
It is also known (see \cite{sep}) that $f_\varepsilon$ is symmetric with respect to $\rho=1/2$, i.e.
\be
\label{symmetry_f}
\forall \rho\in[0,1],\quad f_\varepsilon (1-\rho)= f_\varepsilon (\rho).
\ee
Therefore, \eqref{def_critical_density} is equivalent to saying that the flat segment of $f_\varepsilon$ is the interval 
$[\rho_c(\varepsilon),1-\rho_c(\varepsilon)]$. 
\textcolor{black}{
The following monotonicity properties with respect to $\varepsilon$ can be established (see Appendix \ref{app:std}) by standard coupling arguments. 
\begin{proposition}
\label{prop_monotone}
{\color{black} The macroscopic parameters are monotonous with respect to the dilution:}
\begin{itemize}
\item[(i)] The function $\varepsilon\mapsto f_\varepsilon(\rho)$ is nonincreasing; 
\item[(ii)] the function $\varepsilon\mapsto \rho_c(\varepsilon)$ is nondecreasing.
\end{itemize}
\end{proposition}
}
%
%
%
%
%
Our main results (Theorems \ref{th_plateau} and \ref{th_dilute_limit} below)
hold under a general assumption \hh\mbox{ on} the disorder distribution $Q$ which will 
be stated and explained in the next subsection.
Concretely,  as will be shown there, assumption \hh\mbox{ is} easily implied by the following simple tail assumption:
\begin{lemma}\label{lemma_assumption}
Assumption \hh\mbox{ }holds if the following condition is satisfied:
\be
\label{tail_assumption}
\mbox{for some }\kappa>1, \qquad Q \big( [r,r+u) \big)=O(u^\kappa)\mbox{ as }u\to 0^+.
\ee
\end{lemma}
}
\begin{theorem}
\label{th_plateau}
%
%
Under assumption \hh, 
there exists $\varepsilon_0>0$ such that $\rho_c(\varepsilon)<\frac{1}{2}$ for every $\varepsilon<\varepsilon_0$. Furthermore, the size of the flat segment is explicit when $\varepsilon$ vanishes:
\be\label{limit_plateau}
\lim_{\varepsilon\to 0}\rho_c(\varepsilon)= \rho_c(0),
\ee
with
\be
\label{def_limit_plateau}
\rho_c(0):=\frac{1}{2}\left(1-\sqrt{1-r}\right).
\ee
\end{theorem}

\begin{remark}
It follows from \eqref{limit_plateau} that the limiting value of the length $1-2\rho_c(\varepsilon)$ of the flat segment is $\sqrt{1-r}$.
The result of \cite{sly} is that $\rho_c(\varepsilon)<1/2$ for any $\varepsilon\in(0,1)$, without requiring 
{\color{black} assumption \hh or} \eqref{tail_assumption}.
The proof of \cite{sly}  yields 
\textcolor{black}{the upper bound 
\be\label{upper_bound_sly}
\rho_c(\varepsilon)\leq\frac{1}{2}-\frac{1}{4}\mu r,\quad\mbox{where}\quad
\mu:=\frac{1}{r}-\Exp\left[\frac{1}{\alpha(0)}\right]
\ee
When the expectation in \eqref{upper_bound_sly} is computed from the disorder distribution \eqref{law_disorder}, it yields (we now denote it with an index $\varepsilon$ to emphasize its dependence on $\varepsilon$)
$$
\Exp_\varepsilon\left[\frac{1}{\alpha(0)}\right]=1-\varepsilon+\varepsilon\int\frac{1}{\alpha}Q(d\alpha)
$$
which converges to $1$ as $\varepsilon\to 0$ to $0$. Thus, in the dilute limit, the quantity $\mu=\mu_\varepsilon$ in \eqref{upper_bound_sly} converges to $(1/r)-1$,  and the upper bound on $\rho_c(\varepsilon)$  converges to $(1+r)/4$, which is strictly bigger than $\rho_c(0)$ in \eqref{limit_plateau}. Correspondingly, in the dilute limit, \eqref{upper_bound_sly} gives a lower bound $(1-r)/2$ on the length of the flat segment, which is strictly smaller that $\sqrt{1-r}$.
}
%
%
%
%
%
\end{remark}
\textcolor{black}{
The next theorem characterizes the dilute limit \cite{kru} of the whole} flux function.
Let
\be
\label{def_d0}
\forall \rho\in[0,1], 
\qquad
\textcolor{black}{f^{0}}(\rho) := \min\left[\rho(1-\rho),\frac{r}{4}\right].
\ee
\begin{theorem}
\label{th_dilute_limit}
%
%
%
Under assumption \hh, uniformly over $\rho\in[0,1]$, one has 
\be
\label{dilute_limit}
\lim_{\varepsilon\to 0}f_\varepsilon(\rho)= \textcolor{black}{f^{0}}(\rho).
\ee
\end{theorem}
\textcolor{black}{
Remark that $f^0\neq f_0$, the latter flux function corresponding to the case $\varepsilon=0$, that is a homogeneous rate $1$ TASEP: \textcolor{black}{recall that} {\color{black} $f_0(\rho):=\rho(1-\rho)$, \textcolor{black}{as} defined in \eqref{f0}.}
The fact that the limit in \eqref{dilute_limit} is $f^0$ and not $f_0$ is the sharp transition announced in the introduction. It can be understood as follows: between highly dilute defects, the system is a homogeneous
rate $1$ TASEP. However, the memory of the defects persists (only) through the maximum flux value $r/4$ instead of 
{\color{black} $1/4 = \max_\rho f_0(\rho)$.}
}\\ \\
It is important to note that, although $\rho_c(0)$ is the lower bound of the flat segment of $f_0$, 
{\color{black} the convergence \eqref{limit_plateau} is not a direct consequence
from \eqref{dilute_limit}.}
Theorem \ref{th_dilute_limit} does not imply the existence of the flat segment for given $\varepsilon$ either. However, the proofs of \eqref{limit_plateau} and \eqref{dilute_limit} are closely intertwined and both follow from our renormalization approach.
%
%
%
%
\subsection{A general assumption}
Let us now state  {\color{black} assumption {\hh} which is used in Theorems \ref{th_plateau} and \ref{th_dilute_limit}. 
For this we first define the maximal current in a finite domain.}


\medskip

Let \textcolor{black}{$B = [x_1,x_2]\cap\Z$ be a nonempty interval in $\Z$, where 
{\color{black} $x_1,x_2\in \Z \cup \{ \pm \infty\}$} with $x_1\leq x_2$.}  In the 
following, $\alpha_B:=(\alpha(x):\ x\in B)$ denotes the environment restricted to $B$.
Consider the TASEP in $B$  with the following boundary dynamics: a particle enters at site $x_1$ \textcolor{black}{(if $x_1>-\infty$)} with rate 
\textcolor{black}{$\alpha(x_1-1)$} if this site is empty;
a particle leaves from site $x_2$ \textcolor{black}{(if $x_2<+\infty$)} with rate $\alpha(x_2)$ if this site is occupied. 
{\color{black}
Note that this process
depends 
on the disorder in the larger box}
\be
\label{def_bdiese}
B^\#:=[x_1-1,x_2]\cap\Z .
\ee
From now, 
\textcolor{black}{we} index all related objects by $B^\#$ (the domain of the relevant disorder variables) rather
than $B$ (the domain where particles evolve).
The generator of this process is given by 
\begin{align}
L^\alpha_{B^\#} \textcolor{black}{\varphi}(\eta) & :=  \sum_{x=x_1}^{x_2-1}\alpha(x)\eta(x)[1-\eta(x+1)]\left[
\textcolor{black}{\varphi}\left(\eta^{x,x+1}\right)-\textcolor{black}{\varphi}(\eta)
\right]\nonumber\\
& +  \textcolor{black}{{\bf 1}_{\Z}(x_1)}\alpha(x_1-1)[1-\eta(x_1)]\left[
\textcolor{black}{\varphi}\left(\eta+\delta_{x_1}\right)-\textcolor{black}{\varphi}(\eta)
\right]\nonumber\\
& +\textcolor{black}{{\bf 1}_{\Z}(x_2)}\alpha(x_2)\eta(x_2)\left[
\textcolor{black}{\varphi}\left(\eta-\delta_{x_2}\right)-\textcolor{black}{\varphi}(\eta)
\right],
\label{generator_open}
\end{align}
where $\eta\pm\delta_x$ denotes the creation/annihilation of a particle at $x$.
{\color{black} If $x_1 = -\infty$ or $x_2 = \infty$, the corresponding boundary term does not exists in \eqref{generator_open}.}
\textcolor{black}{When $B$ is finite, this process has a unique invariant measure. This allows the following
definition of the maximal current for the disordered TASEP restricted to $B$. }
\begin{definition}
\label{H_particle}
\textcolor{black}{Assume $B$ is finite.}
The maximal current $j_{\infty,B^\#}(\alpha_{B^\#})$ is the stationary current in the open system defined above, i.e. (independently of $x=x_1,\ldots,x_2-1$)
\begin{eqnarray}
\label{eq: max flux def}
j_{\infty,B^\#}(\alpha_{B^\#}) & = & \int\alpha(x)\eta(x)[1-\eta(x+1)]d\nu^\alpha_{B^\#}(\eta)\\
& = & 
\int 
\textcolor{black}{ \alpha(x_1-1)}
 [1-\eta(x_1)]d\nu^\alpha_{B\#}(\eta)=\int\alpha(x_2)\eta(x_2)d\nu^\alpha_{B^\#}(\eta) , 
\nonumber
\end{eqnarray}
where $\nu^\alpha_{B^\#}$ is the unique invariant measure for the process on $B$ with generator $L^\alpha_{B^\#}$.
\end{definition}
\begin{remark}
One can see that
the right-hand side of \eqref{eq: max flux def} is independent of $x$ 
by writing that the expectation under $\nu^\alpha_{B^\#}$ of $L^\alpha_{ B^\#}\eta(x)$ for $x\in[x_1,x_2]\cap\Z$ (which yields the difference of two consecutive integrals in \eqref{eq: max flux def}) is zero.
\end{remark}
To simplify notation, we shall at times omit the dependence on $\alpha_{B^\#}$ and write $j_{\infty,B^\#}$.
It is well-known \cite{der} that in the homogeneous case, i.e. when $\alpha(x) = r$ for all $x$
in $[0,N]$  (with $r$ a positive constant), $j_{\infty,[0,N]}$ is no longer a random variable and
\be
\label{eq: courant homogene}
\lim_{N \to \infty} j_{\infty,[0,N]} = \inf_N j_{\infty,[0,N]}  = \frac{r}{4} \, . 
\ee
%
%
%
In fact, explicit computations \cite{der} show that, for some constant $C>0$,
\be\label{decay_r}
j_{\infty,[0,N]}\geq \frac{r}{4}+\frac{C}{N}.
\ee
The quantity $j_{\infty,[0,N]}(\alpha_{[0,N]})$ is a function of the environment which measures the speed of decay 
of the maximum current in a box to $r/4$ as the size of the box increases. 
{\color{black} Assumption {\hh}, stated below, requires that with high probability on the disorder the decay of the maximal current towards $r/4$ is slightly slower than \eqref{decay_r}.}

\medskip

\noindent
{\bf Assumption \hh}.
{\it There exists $b\in(0,2)$, $a>0$, $c>0$ and $\beta>0$ such that, for $\varepsilon$ small enough, the following holds for any $N$}: 
\be
\label{condition_h_2}
\mathcal P_\varepsilon
\left(
j_{\infty,[0,N]} ( \alpha_{[0,N]}  )  
\leq\frac{r}{4}+\frac{a}{N^{b/2}} \right) 
\leq \frac{c}{N^\beta} \, .
\ee
Note that if assumption \hh\, is satisfied for some $b\in(0,1)$, it is satisfied {\it a fortiori} for $b=1$.
Thus, from now on, without loss of generality, we will assume that $b\in[1,2)$.
{\color{black}  We stress the fact that the condition $b<2$ is borderline as 
a simple comparison with the homogeneous case \eqref{decay_r} 
leads to a control of the decay for  $b=2$.}

\medskip

We have not been able to prove that assumption {\hh} is satisfied for Bernoulli disorder $Q=\delta_r$, although we believe this is true. However, as stated in  Lemma \ref{lemma_assumption}, the tail assumption \eqref{tail_assumption} implies {\hh}. 
%
%
%
%
%
\begin{proof}[Proof of Lemma \ref{lemma_assumption}]
Let $\alpha^\star:=\min_{x\in [0,N]}\alpha(x)$. 
It follows from a standard coupling argument 
\textcolor{black}{(see (ii) of Lemma \ref{lemma_j_alpha})} that the flux is monotone with respect to the jump rates:
%
\be
\label{simple_coupling}
j_{\infty,[0,N]}(\alpha_{[0,N]}) \geq                                                            
j_{\infty,[0,N]} \big( \alpha^\star,\ldots,\alpha^\star \big),
\ee
where 
$j_{\infty,[0,N]} \big( \alpha^\star,\ldots,\alpha^\star \big)$ 
stands for the current of a homogeneous TASEP {in $[1,N]$ with bulk, exit and entrance rates
$\alpha^\star$.
By \eqref{decay_r}, it is larger than $\alpha^\star/4$, so that $j_{\infty,[0,N]}(\alpha_{[0,N]})  \geq\alpha^\star/4$.
Thus assumption \mbox{\hh} will be implied by controlling $\alpha^\star$. 
Using  the tail of the distribution $Q$ \eqref{tail_assumption}, we get }
$$
\mathcal P_\varepsilon\left(
\min_{x \in[0,N]}\alpha(x)\leq r+\frac{a'}{N^{b/2}}
\right) \leq N Q \left( [r,  r+\frac{a'}{N^{b/2} } ) \right)     
\leq \frac{c'}{N^\beta},
$$
for some well chosen parameters $a'>0$ , $c'>0$, $b\in(0,2)$, $\beta>0$.
This follows from elementary computations.
\end{proof}
%
%
%
%
%
\section{last passage percolation approach}\label{sec:lpp}
The derivation of Theorems \ref{th_plateau}  and \ref{th_dilute_limit} relies on a reformulation of the problem in terms of last passage percolation.
\subsection{Wedge last passage percolation} 
Let 
%
%
%
%
$Y=(Y_{i,j}:\,(i,j)\in\Z\times\N)$ 
be an i.i.d. family of exponential random variables with parameter 1 independent of the 
environment $(\alpha(i):\ i\in\Z)$. In the following, these variables will sometimes be called {\it service times}, in reference
to the queuing interpretation of TASEP.
The distribution of $Y$ is denoted by $\Prob$ and {\color{black} the} expectation with respect to this distribution by $\Exp$.
Let 
$$
\mathcal W:=\{(i,j)\in\Z^2:\,j\geq 0,\,i+j\geq 0\} \, .
$$ 
Index $i$ represents a site and index $j$ a particle.
Given two points $(x,y)$ and $(x',y')$ in $\Z\times\N$, we denote by  $\Gamma((x,y),(x',y'))$ the set of paths
$\gamma=(x_k,y_k)_{k=0,\ldots,n}$ such that $(x_0,y_0)=(x,y)$, $(x_n,y_n)=(x',y')$, and $(x_{k+1}-x_k,y_{k+1}-y_k)\in\{(1,0),(-1,1)\}$ for every $k=0,\ldots,n-1$.
Note that  $\Gamma((x,y),(x',y'))=\emptyset$ if $(x'-x,y'-y)\not\in\mathcal W$. Given a path $\gamma\textcolor{black}{=(x_k,y_k)_{k=0,\ldots,n}}\in\Gamma((x,y),(x',y'))$, its passage time is defined by
\be\label{def_passage}
T^\alpha(\gamma):=\sum_{k=0}^n 
\frac{Y_{x_k,y_k}}{\alpha(x_k)}.
%
%
\ee
The last passage time between $(x,y)$ and $(x',y')$ is defined by 
\be\label{def_last}
T^\alpha((x,y),(x',y')):=\max\{T^\alpha (\gamma):\,\gamma\in\Gamma((x,y),(x',y'))\}.
\ee
We shall simply write $T^\alpha(x,y)$ for $T^\alpha((0,0),(x,y))$. This quantity has the following particle interpretation.
%
%
For $(t,x)\in[0,+\infty)\times\Z$, let
\begin{eqnarray*}
H^\alpha(t,x) =  \min\{y\in\N: \quad T^\alpha (x,y)>t\}
\quad \text{and} \quad 
\eta_t^\alpha(x)  =  H^\alpha(t,x-1)-H^\alpha(t,x).
\end{eqnarray*}
Then $(\eta_t^\alpha)_{t\geq 0}$ is a TASEP with generator \eqref{generator} and initial configuration 
\textcolor{black}{$\eta^*=\indicator{\Z\cap(-\infty,0]}$}, and 
$H^\alpha$ is its height process.
Besides, if we label particles initially so that  the particle at $x\leq 0$ has label $-x$, 
then for $(x,y)\in\mathcal W$, $T^\alpha(x,y)$ is the time at which particle $y$ reaches site  $x+1$.
Let us recall the following result from \cite{sep}.
%
%
\begin{theorem}
\label{th_sep}
Let $\mathcal W':=\{(x,y)\in\R^2:\,y \geq 0,\,\,x+y\geq 0\}$.
For {\color{black} $\mathcal P$-a.s.} realization of the disorder $\alpha$, the function\margincom{\textcolor{black}{j'aurais mis a.e. ?}}
\be
\label{def_limtime}
(x,y)\in\mathcal W'
\mapsto
\tau(x,y):=\lim_{N\to\infty}  \frac{1}{N}  T^\alpha([Nx],[Ny])
\ee
is well-defined in the sense of a.s. 
convergence with respect to the distribution of $Y$. 
It is finite, positively $1$-homogeneous and superadditive (thus concave). The function
\be
\label{def_limheight}
(t,x)\in[0,+\infty)\times\R\mapsto h(t,x):=\lim_{N\to\infty}\frac{1}{N}H^\alpha([Nt],[Nx])
\ee
is well-defined in the sense of a.s. convergence with respect to the distribution of $Y$. It is finite, positively $1$-homogeneous and subadditive (thus convex). These functions do not depend on $\alpha$ and are related through
\begin{eqnarray}
\label{corr_3} h(t,x) & = & \inf\{y  \in [0,+\infty): \quad  \tau(x,y)>t\},\\
\label{corr_4} \tau(x,y) & = & \inf\{t  \in[0,+\infty): \quad h(t,x)\geq y\}.
\end{eqnarray}
\end{theorem}
\textcolor{black}{
\begin{remark}\label{remark_L1}
It is shown in \cite[Lemma 5.2]{sep} that $rT^\alpha(x,y)$ is stochastically dominated  by the sum of $x+2y$
i.i.d. $\mathcal E(1)$ random variables. Tail estimates for this sum imply that the
sequence of random variables $N^{-1}T^\alpha([Nx],[Ny])$ is uniformly integrable. Thus the limit \eqref{def_limtime} also holds in ${\mathbb L}^1$.
\end{remark}
}
By homogeneity, the function $h$ in \eqref{def_limheight} is of the form
\be\label{def_limheight_k}
h(t,x)=t k \left( \frac{x}{t} \right) 
\ee
for some convex function $k:\R\to\R^+$. 
It is known that for homogeneous TASEP (that is $\alpha(x)=1$ for all $x$), we have
\begin{eqnarray}
\label{eq: homogeneous LPP}
\tau(x,y)=(\sqrt{x+y}+\sqrt{y})^2 ,\quad 
k(v)=\frac{(1-v)^2}{4}\indicator{[-1,1]}(v)-v\indicator{(-\infty,-1)}(v) .
\end{eqnarray}

\subsection{Reformulation of Theorems \ref{th_plateau}   and \ref{th_dilute_limit}}

In this section, we are going to rewrite the flux in the last passage framework and show that Theorem \ref{th_plateau} can be deduced from a statement on the passage time.
It is shown in \cite{sep} that the macroscopic flux function $f$ is related to $k$ (defined in \eqref{def_limheight_k}) by the convex duality relation
\be
\label{flux_legendre}
f(\rho):=\inf_{v\in\R}[k(v)+v\rho],\quad\rho\in[0,1]
\ee
which implies concavity of $f$. 
We now
introduce a family of ``reference'' macroscopic flux functions and associated macroscopic passage time and height functions. Let $0\textcolor{black}{<}\rho_c\leq 1/2$ and $J \geq 0$.
For 
$\rho\in[0,1]$, we define (see Figure \ref{fig: flux truncated})
%
%
%
\be
\label{ref_flux}
f^{\rho_c,J}(\rho)  : = 
%
%
J\min\left(
\frac{\rho}{\rho_c},\frac{1-\rho}{\rho_c},1
\right).
\ee

\begin{figure}[htpb]
  \centering
\includegraphics[width=.5\columnwidth]{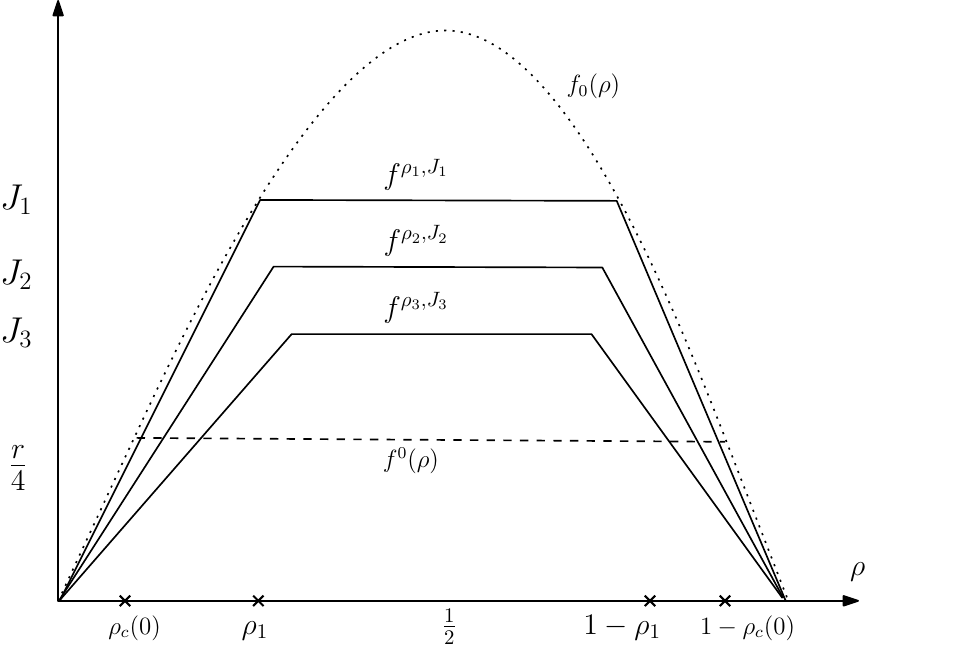}
\caption{\tiny 
{\color{black}
The homogeneous TASEP flux $f_0(\rho) = \rho (1-\rho)$ is represented in dotted line and 3 graphs of modified fluxes $f^{\rho_n,J_n}$ are depicted in plain line.
The renormalization strategy   amounts to bound from below the flux at the scale $n$ by $f^{\rho_n,J_n}$ and to use this information to control the lower bound on the flux at the scale $n+1$. As depicted in the figure, the sequence of fluxes $J_n$ decays to  $r/4$ when the scale $n$ grows. The width of the flat segment $[\rho_n, 1- \rho_n]$ shrinks
also at each step but remains controlled. 
When the dilution $\varepsilon$ tends to 0, the limiting flux $f^0$ defined in \eqref{def_d0} is the flux $f_0(\rho)$ truncated in 
$[\rho_c(0), 1- \rho_c(0)]$ at the level $r/4$ (dashed line). 
For $\varepsilon$ small enough, $J_1$ can be chosen very close to $r/4$ and $\rho_1$ close to $\rho_c(0)$. Furthermore for small $\varepsilon$, 
the flat segment $[\rho_n, 1- \rho_n]$ is almost unchanged at each scale and
this leads to the convergence in Theorem \ref{th_plateau}.}
 } 
  \label{fig: flux truncated} 
\end{figure}

Given Proposition \ref{prop_max}, 
{\color{black} the occurence of a flat segment in Theorem \ref{th_plateau} boils down to proving 
the existence of $\varepsilon_0>0$ and  $\rho \in[0,1/2)$ such 
that the flux remains above $r/4$  for densities in  $[\rho, 1- \rho]$:}
\be
\label{lower_bound_0}
\forall \varepsilon < \varepsilon_0, \qquad 
f_\varepsilon  \geq f^{{\rho},r/4} .
%
\ee
{\color{black}
As $f_\varepsilon \leq r/4$ by Proposition \ref{prop_max}, lower bound \eqref{lower_bound_0} implies that $f_\varepsilon$ equals $r/4$ on $[\rho, 1- \rho]$.
Since $f_\varepsilon$ is concave and symmetric \eqref{symmetry_f}, $\rho_c(\varepsilon)$ in \eqref{def_critical_density} is characterized by:}
%
\be
\label{charac_critical}
\rho_c(\varepsilon)=\inf\left\{
\rho\in[0,1/2]: \quad f\geq f^{\rho,r/4}
\right\}.
\ee
%
%
%
The convex conjugate of  $f^{\rho_c,J}$ through Legendre duality \eqref{flux_legendre} is defined for $x\in\R$ by
\begin{eqnarray}
\label{ref_height_2}
k^{\rho_c,J}(x) & := & (-x)\indicator{(-\infty,-J/\rho_c)}(x)\\
\nonumber & + & [J-(1-\rho_c)x]\indicator{[-J/\rho_c,0)}(x)
+[J-\rho_c x]\indicator{[0,J/\rho_c)}(x).
\end{eqnarray}
Finally, one can associate to $k^{\rho_c,J}$ a  passage time function and a height function, related by \eqref{corr_3}--\eqref{corr_4},  and defined for $x\in\R$ and $y\geq x^-$
by
\begin{eqnarray}
\label{ref_time} 
\tau^{\rho_c,J}(x,y) := \frac{\rho_c x^+ - (1-\rho_c)x^-+y}{J} 
%
\quad \text{and} \quad 
h^{\rho_c,J}(t,x)   :=  t k^{\rho_c,J}(x/t) ,
\end{eqnarray}
where $x^+ = \max \{ x,0 \}$ and $x^- = -\min \{ x,0 \}$.
It follows from \eqref{corr_3}, \eqref{def_limheight_k} and \eqref{flux_legendre} that
\be
\label{inverse_order}
f\geq f^{\rho,J} \quad \Leftrightarrow  \quad \tau\leq\tau^{\rho,J}.
\ee
Hence, the quantity $\rho_c(\varepsilon)$ in \eqref{def_critical_density} can be defined equivalently as follows:
\be
\label{critical_density_2}
{\rho}_c(\varepsilon)=\inf\{\rho\in[0,1/2]:
\quad \tau_\varepsilon  \leq  \tau^{\rho,r/4}\}.
\ee
Thus the lower bound \eqref{lower_bound_0} on the flux  can be rephrased in terms of an upper bound on the last passage time.
%
Theorems \ref{th_plateau}--\ref{th_dilute_limit} are consequences of the following theorems, which will be proved in the next sections.
\begin{theorem}
\label{thm: passage upper}
Let $\tau_\varepsilon$ be the limiting passage time defined by \eqref{def_limtime} when the environment has distribution $\mathcal P_\varepsilon$. Then, under assumption \hh, 
there exist $\varepsilon_0>0$ and $\rho< 1/2$ such that 
\be
\label{lower_bounds}
\textcolor{black}{\forall \varepsilon < \varepsilon_0},
\qquad 
\tau_\varepsilon \leq \tau^{\rho,r/4},
\ee
with $ \tau^{\rho,r/4}$ defined in \eqref{ref_time}.
In particular, $\tau_\varepsilon(.,y)$ has a cusp at $x=0$ 
\textcolor{black}{
and
the optimal value $ {\rho}_c(\varepsilon)$ introduced in \eqref{critical_density_2} converges in the dilute limit:}
\be
\label{limit_rho'}
\lim_{\varepsilon\to 0} {\rho}_c(\varepsilon)=\frac{1}{2}(1-\sqrt{1-r}).
\ee
\end{theorem}
\begin{theorem}
\label{th:passage_dilute}
\textcolor{black}{The  passage time function $\tau_\varepsilon$ converges in the dilute limit:}
\textcolor{black}{
\be
\label{def_tau0}
\lim_{\varepsilon\to 0}\tau_\varepsilon(x,y)=\tau_0(x,y):=\left\{
\ba{lll}
%
%
(\sqrt{x+y}+\sqrt{y})^2 &  \mbox{if} & y\leq x^+ y_1^1(0)-x^-y_1^{-1}(0)\\ \\
%
\tau^{\rho_c(0),r/4}(x,y) & \mbox{if} & y> x^+ y_1^1(0)-x^-y_1^{-1}(0)
%
%
\ea
\right.
\ee
}
\textcolor{black}{where $\tau_0$ is the counterpart of the flux function $f_0$
defined in \eqref{dilute_limit} and}
\margincom{{\color{black} Il y a 2 valeurs de $y_1(0)$ selon $\sigma = \pm 1$ ?\\
et aussi proposition 4.2 ?
}\textcolor{black}{en effet, j'ai note $y_1^\sigma$, cf. preuve prop 4.5}}
\be
\label{def_y1}
y_1^{\textcolor{black}{1}}(0):=\frac{\rho_c(0)^2}{1-2\rho_c(0)}\in[0,+\infty],\quad
\textcolor{black}{
y_1^{-1}(0):=\frac{[1-\rho_c(0)]^2}{1-2\rho_c(0)}\in[0,+\infty]
}
\ee
where $\rho_c(0)$ was introduced in \eqref{def_limit_plateau}.
\end{theorem}

Theorem \ref{thm: passage upper} can be partially extended to LPP with general service-time distribution and heavier tails.
In this case the particle interpretation is less standard, though the process can be viewed as a non-markovian TASEP (see e.g. \cite{khg}  and \cite{sep}).
Our approach (and the extension just explained) also applies to other LPP models with columnar disorder (in the wedge picture) or diagonal disorder (in the square picture), like for instance the $K$-exclusion process \cite{sep}.

\subsection{Last passage reformulation of assumption \hh}\label{subsec:reform}

We  will  reformulate condition \hh\mbox{ in }the last passage setting. To this end, we define restricted passage times. Let $B=[x_1,x_2]\cap\Z$ (where $x_1,x_2\in\Z$) be a finite interval of $\Z$. If $(x,y)$ and $(x',y')$ are such that $x$ and $x'$ lie in $B$, we define $\Gamma_B((x,y),(x',y'))$ as the subset of $\Gamma((x,y),(x',y'))$ consisting of paths
$\gamma$ that lie entirely inside $B$ in the sense that $x_k\in B$ for every $k=0,\ldots,n$. We then define 
\be\label{restricted_time}
T^\alpha_B((x,y),(x',y')):=\max \Big\{T^\alpha (\gamma):\,\gamma\in\Gamma_B((x,y),(x',y')) \Big\}.
\ee
%
The counterpart of Definition \ref{H_particle} is
\begin{lemma}
\label{lemma_reform}
\textcolor{black}{ $B=[x_1,x_2]\cap \Z$ be a nonempty interval of $\Z$, with $x_1,x_2\in   \Z \cup\{\pm \infty\}$ such that
$x_1\leq x_2$.}
The limit
\be\label{def_mintime}
T_{\infty,B}(\alpha_B):=\lim_{m\to\infty} 
%
%
{\frac{1}{m}}  T^\alpha_{B}((x_0,0),(x_0,m))
=\sup_{m\in\N^*}   
%
%
\Exp\left[{\frac{1}{m}}  T^\alpha_{B}((x_0,0),(x_0,m))\right]
\ee
exists $\Prob$-a.s. for $x_0\in B$, does not depend on the choice of $x_0$, and defines a random
variable depending only on the disorder restricted to $B$.
Besides, \textcolor{black}{if $B$ is finite,} we have
\be\label{max_current_lpp}
T_{\infty,B}(\alpha_B)= \frac{1}{j_{\infty,{B}}(\alpha_{B})}  ,
\ee
where 
$j_{\infty,B}(\alpha_{B})$ is the stationary current \eqref{eq: max flux def} in the open system restricted to 
\be
\label{def_bprime}
B':=[x_1+1,x_2]\cap\Z.
\ee
%
%
\end{lemma}
\textcolor{black}{
\begin{remark}
Since $T^\alpha_B((x_0,0),(x_0,m))\leq T^\alpha((x_0,0),(x_0,m))$, by Remark \ref{remark_L1} above, 
$T^\alpha_B((x_0,0),(x_0,m))$ is dominated stochastically by the sum of $2m$ i.i.d. $\mathcal E(1)$ random variables. Thus, as in Theorem \ref{th_sep}, the limit in \eqref{def_mintime} also holds in $L^1$.
\end{remark}
}
\noindent
\textcolor{black}{
Note that in \eqref{eq: max flux def},  $j_{\infty,B^\#}(\alpha_{B^\#})$ was defined
as the maximum current for the TASEP in $B$.  
By \eqref{def_bdiese} and \eqref{def_bprime},
$(B')^\#=B$ so that  the above lemma is  consistent with \eqref{eq: max flux def}.
}
The proof of  Lemma \ref{lemma_reform} is postponed to Appendix \ref{appendix_lpp}.
%
%
%
%
%
%

\medskip

To simplify notation, we shall at times omit $\alpha_B$ and write $T_{\infty,B}$,  $j_{\infty,B}$.
We can now restate  condition \hh\, {\color{black} in terms of last passage time}:

\smallskip

\noindent
{\bf Assumption \h}. {\it There exists $b\in(0,2)$, $a>0$, $c>0$ and $\beta>0$ such that, for $\varepsilon$ small enough, one has for any $N\in\N^*$:}
\be
\label{condition_h}
\mathcal P_\varepsilon\left(
T_{\infty,[0,N]}  ( \alpha_{[0,N]} ) 
\geq\frac{4}{r}-\frac{a}{N^{b/2}}
\right)\leq \frac{c}{N^\beta} \, .
\ee
The constants $a,c$ in \eqref{condition_h_2} are different from those in \eqref{condition_h}, but $b$ and $\beta$ are the same.

\section{Renormalization scheme}\label{sec:renorm}

From now, we are going to focus on the last passage percolation model in order to prove Theorems  \ref{thm: passage upper} and \ref{th:passage_dilute}.
We first describe a renormalization procedure to show that a bound of the form \eqref{lower_bounds}
holds with high probability at every scale (see Proposition \ref{prop_block}  below).

%

\subsection{Definition of blocks}\label{sec:def_blocks}
Let $n\in\N^*$ be the renormalization ``level'' and $K_n=K_n(\varepsilon)$ the size of a renormalized block of level $n$ (by block we mean a finite subinterval of $\Z$). For $n=1$, we initialize $K_1=K_1(\varepsilon)$ and define a block $B$ of order $1$ to be good if it contains no defect, i.e. $\alpha(x)=1$ for every $x\in B$. Otherwise, the block is said to be bad.

\medskip

For $n\geq 1$, we set 
\be\label{def_l_gamma} 
{\color{black} K_{n+1}=l_n K_n, \qquad \text{with} \quad } 
l_n= \lfloor K_n^\gamma \rfloor\textcolor{black}{=l_n(\varepsilon)}\ee
with $\gamma\in(0,1)$. For $n\geq 1$, a block \textcolor{black}{$B$} of order $n+1$ has size $K_{n+1}$ and is partitioned  into $l_n$ disjoint blocks of level $n$ \textcolor{black}{whose size is $K_n$}. This block is called ``good''
if it contains at most one bad subblock of level $n$, and if condition \eqref{condition_good} below  on the maximum current in the block holds:
\be
\label{condition_good} 
j_{\infty,B_{n+1}} 
\geq j_{n+1}
\quad \text{with} \quad 
j_{n+1} : =  \frac{r}{4}+\frac{a}{K_{n+1}^{b/2}} ,
\ee
where the constants $a,b$ were defined in 
\textcolor{black}{\eqref{condition_h_2}}.
Otherwise $B_{n+1}$ is said to be bad. We stress the fact that the status (good or bad) of $B_{n+1}$ depends only on the disorder variables $\alpha_{B_{n+1}}$ in $B_{n+1}$ and not on the exponential times $Y_{i,j}$.

\medskip

The renormalization is built such that large blocks are good with high probability.
Let $\textcolor{black}{q_n=}q_n(\varepsilon)$ denote the probability under $\mathcal P_\varepsilon$ that 
the block $[0,K_{n}-1]\cap\Z$, at level $n$, is bad. \textcolor{black}{Since the distribution of $\alpha$ 
is invariant with respect to space shifts, $q_n(\varepsilon)$ is also, for any $x\in\Z$, the probability that 
$[x,x+K_{n}-1]\cap\Z$ is a bad block. In the rest of this paper, quantities $K_n$, $l_n$, $q_n$ will be written with or without explicit dependence on $\varepsilon$, depending on necessity.}
\begin{lemma}
\label{lemma_good}
Suppose that assumption \hh\,  holds and set 
$$
K^*(\varepsilon):=\left( \frac{2c}{\varepsilon} \right)^{\frac{1}{\beta+1}}, 
\quad
K_*:= 2+ (4c)^{ \frac{1}{\beta-\gamma(\beta+2)} },
$$
\begin{equation}
\label{eq: parameters}
\gamma_0:=\frac{\beta}{\beta+2}, 
\quad
\varepsilon_0:=\min\left\{1,2^{-\beta}c,
(2c)\left[3+(4c)^{-\frac{1}{\beta-\gamma(\beta+2)}} \right]^{\beta+1}
\right\},
\end{equation}
with the constants $c,\beta$ appearing in  \eqref{condition_h_2} and \eqref{condition_h}.
Then 
for all  $\gamma\in(0,\gamma_0)$ and $\varepsilon \leq \varepsilon_0$,  
{\color{black} there is an integer $K_1 (\varepsilon)$ in}
the interval $[K_*, K^*(\varepsilon)]$ such that 
%
\be
\label{conclusion_lemma_good}
\forall \varepsilon < \varepsilon_0, \qquad \lim_{n\to\infty}q_n(\varepsilon)=0 
\quad \text{and furthermore} \quad 
%
\lim_{\varepsilon\to 0}K_1(\varepsilon)=+\infty.
\ee 
\end{lemma}
\begin{proof} 
For $n\geq 1$, let $\zeta_{n}=\frac{c}{K_{n}^{\beta}}$ be the upper bound in \eqref{condition_h}. 
Then, by definition of good blocks and independence of the environment, 
one obtains the recursive inequality
\begin{eqnarray*}
q_1 & \leq & K_1\varepsilon , \\
q_{n+1} & \leq & (l_n q_n)^2 + \zeta_{n+1},
\quad n\geq 1.
\end{eqnarray*}
\textcolor{black}{
Indeed, the first inequality is a union bound over the $K_1$ sites of the block for the probability $\varepsilon$ of each site
being a defect. The second inequality is a union bound over the $l_n(l_n-1)/2\leq l_n^2$ pairs of subblocks 
for the probability $q_n^2$ of both subblocks being bad at order $n$.
{\color{black} The last term $\zeta_{n+1}$ estimates from above the probability that the maximum current in the whole block does not satisfies  \eqref{condition_good} (see \eqref{condition_h_2} and \eqref{condition_h}).}
}

Note that if for some $n$, we have $q_n\leq 2 \zeta_n$ and $\zeta_{n+1}\geq 4 l_n^2 \zeta_n^2$, then $q_{n+1}\leq 2 \zeta_{n+1}$.
Thus if we have $q_1\leq 2\zeta_1$ and $\zeta_{n+1}\geq 4 l_n^2 \zeta_n^2$ 
for all $n\geq 1$, then $q_n\leq 2 \zeta_n$ 
for all $n\geq 1$, implying $q_n(\varepsilon)\to 0$ as $n\to\infty$  provided $K_1\geq 2$.

On the one hand, $q_1\leq 2 \zeta_1$ follows from {\color{black} $K_1\leq K^*(\varepsilon)$.}
%
%
On the other hand, 
$\zeta_{n+1}\geq 4 l_n^2 \zeta_n^2$ is equivalent to
$$
K_n^{\beta-(\beta+2)\gamma}\geq 4 c,\quad\forall n\geq 1.
$$
Assuming $0\leq\gamma<\frac{\beta}{\beta+2}$, 
since $K_n$ is increasing in $n$,  the above inequality holds for all $n\geq 1$ if it holds for $n=1$, which is equivalent to 
\be
\label{condition_2}
K_1\geq K':= (4c)^{\frac{1}{\beta-\gamma(\beta+2)}}.
\ee
Finally,  setting $K_*:=2+K'$, we have $1+K_* \leq K^*(\varepsilon)$  if  $\varepsilon \leq \varepsilon_0$.
{\color{black}
Thus \eqref{conclusion_lemma_good} is satisfied by choosing the sequence $K_1(\varepsilon):=\lfloor K^*(\varepsilon)\rfloor$. }
%
%
%
%
\end{proof}

\subsection{Mean passage time in a block}

The strategy to prove Theorem \ref{th_plateau} is now as follows. 
To each block 
$B=[x_0, x_1 := x_0+K_n-1]\cap\Z$ of  level $n$, we associate finite-size macroscopic restricted passage time functions (in the left and right directions) taking as origin either extremity of the block:
\begin{eqnarray}
\label{eq: path direction}
\begin{cases}
\tau^\alpha_{n,B}(1,y)  =  \Exp \left( \frac{1}{K_n} T_{B}^\alpha((x_0,0),(x_1, \lfloor K_n y\rfloor)) \right)
,\quad y\geq 0, \\
\tau^\alpha_{n,B}(-1,y) = \Exp \left( \frac{1}{K_n} T_{B}^\alpha((x_1,0),(x_0,\lfloor K_n y\rfloor)) \right)
,\quad y\geq 1,
\end{cases}
\end{eqnarray}
which depend only on the disorder $\alpha_{\textcolor{black}{B}}$.
To keep compact notation, we will write both functions in the form 
$\tau^\alpha_{\textcolor{black}{B}}(\sigma,y)$ with $\sigma = \pm 1$ and $y \geq \sigma^- = - \min\{\sigma,0\}$.

\medskip

The main step towards Theorem \ref{thm: passage upper}, stated in Proposition \ref{prop_block} below, is to prove that the mean passage time at each level $n$ remains bounded by 
the reference function \eqref{ref_time} with parameters $\rho_n,J_n$ appropriately controlled to ensure that the flat segment is preserved at each order.
\begin{proposition}
\label{prop_block}
For small enough $\varepsilon$, there exist sequences 
$(\rho_n=\rho_n(\varepsilon))_{n \geq 1}\in[0,1]^{ \N^*}$ 
and $(J_n=J_n(\varepsilon))_{n \geq 1}\in [0,+\infty)^{ \N^*}$
such that: 

\smallskip

\noindent
(i) {\color{black} Uniformly over 
good blocks $B$ at level $n$ and for every $\sigma\in\{-1,1\}$}\margincom{\textcolor{black}{$\alpha$ est fixe et B varie, donc plutot over good block ?}}
%
\be
\label{bound_on_good_blocks}
\forall y\geq \sigma^-, \qquad 
{\color{black} \sup_{\text{good} \ B}} 
\tau^\alpha_{n,B}(\sigma,y) \leq 
\tau^{\rho_n,J_n}(\sigma,y),
\ee
(ii) $\lim_{n\to\infty} J_n=r/4$ and $J_n>r/4$ for all $n\in\N^*$, \\
(iii) $\limsup_{n\to\infty}\rho_n<1/2$ \textcolor{black}{
and with the definition \eqref{def_limit_plateau} of $\rho_c(0)$
$$
\limsup_{\varepsilon\to 0}\limsup_{n\to+\infty}
\rho_n(\varepsilon)\leq \rho_c(0)= \frac{1}{2}(1-\sqrt{1-r}).
$$
}
%
%
%
\end{proposition}
Once Proposition \ref{prop_block} is established, completing the proof of Theorem \ref{thm: passage upper} (and thus Theorem \ref{th_plateau}) is a relatively simple task,
which boils down to obtain a similar bound on {\it unrestricted} passage times (see  Section \ref{sec:completion}). 
\textcolor{black}{The upper bound $\tau^{\rho_n,J_n}$ is the counterpart, in the last passage percolation setting, of the modified flux $f^{\rho_n,J_n}$  depicted Figure \ref{fig: flux truncated}.}

{\color{black} 
The  derivation of Theorem \ref{th_dilute_limit} relies on  a refined version of \eqref{bound_on_good_blocks}
in the dilute limit.}
%
\begin{proposition}\label{prop_dilute_limit}
For every $\sigma\in\{-1,1\}$ and \textcolor{black}{$\sigma^-\leq y<\sigma^+y_1^1(0)-\sigma^-y_1^{-1}(0)$}
\be\label{dilute_g}
\limsup_{\varepsilon\to 0}\limsup_{n\to+\infty}
{\color{black} \sup_{\text{good} \ B }} \;
\tau^\alpha_{n,B}(\sigma,y) 
\leq(\sqrt{\sigma+y}+\sqrt{y})^2.
\ee
\margincom{\textcolor{black}{idem}} where  $y_1(0)$ was defined in \eqref{def_y1}.
\end{proposition}

\subsection{Coarse-graining and recursion}
\label{subsec:coarse}


The strategy of the proof of Propositions \ref{prop_block} and \ref{prop_dilute_limit} is based on a coarse-graining procedure.
\textcolor{black}{
We will first show a general estimate for any good block $B$ at level $n \geq 1$ and 
$\sigma\in\{-1,1\}\textcolor{black}{:}$\margincom{\textcolor{black}{idem}}\textcolor{black}{
\be
\label{nonparbound}
\forall y\geq\sigma^-, \qquad 
\sup_{\text{good} \ B}\;
 \tau^\alpha_{n,B}(\sigma,y)\leq g_n(\sigma,y),
\ee
}
where $g_n(\sigma, \cdot)$ is a sequence of concave functions defined recursively in Proposition \ref{prop:induction_gen}. Then Propositions \ref{prop_block} and \ref{prop_dilute_limit} will be deduced from \eqref{nonparbound}.
}
\begin{proposition}
\label{prop:induction_gen}
{\color{black} Fix $C$ a large enough constant and set}
\be
\label{eq: delta n}
j_{n+1}:=\frac{r}{4}+\frac{a}{K_{n+1}^{b/2}}
,\quad l_n := \lfloor K_n^\gamma \rfloor \quad 
\mbox{ and }\quad
\delta_n:= C\frac{(\log K_{n+1})^{3/2}}{\sqrt{K_n}}\textcolor{black}{.}
\ee
%
%
\textcolor{black}{Then,} the sequence $(g_n)_{n \geq 1}$ defined on $[\sigma^-,+\infty)$ by 
\margincom{CH 31/10 nouveau $g_1$ de la limite dilu\'ee}
\be\label{level_1}
g_1(\sigma,y):=(\sqrt{\sigma+y}+\sqrt{y})^2
\ee
and 
\begin{align}
g_{n+1}(\sigma,y)  :=   \sup_{ \sigma^- \leq\bar{y}\leq \frac{l_n}{(l_n-1)} y}
\left\{
\left(1-\frac{1}{l_n}\right)\left[
g_n(\sigma,\bar{y})- \frac{\bar{y}}{j_{n+1}}
\right] \right\} 
+ \frac{y}{j_{n+1}}+  \frac{1+\sigma}{2l_nj_{n+1}}\label{recursion_gen}\\
\textcolor{black}{+\delta_n \varphi(y),}\nonumber
%
\end{align}
where
\begin{equation}
\label{eq: phi}
\varphi(y):=\sqrt{\frac{\sigma}{2}+y}\left[2 +\log (1+y)\right]^{3/2},
\end{equation}
satisfies the bound  \eqref{nonparbound}
for any good block $B$ and $n \geq 1$.
\end{proposition}
The proof of this proposition is postponed to Section \ref{subsec:outline}.
The recursion  \eqref{recursion_gen} between $g_n$ and $g_{n+1}$ is obtained by decomposing
  a path at level $n+1$  
into subbpaths contained in  subblocks of size $K_n$. We then express the total passage time as a maximum of a sum of the partial passage times in each subblock, where the maximum is over all possible intermediate heights of the path at the interfaces. 
The \textcolor{black}{term $\delta_n \varphi$ on the last line of \eqref{recursion_gen}}
is a fluctuation estimate (see Proposition \ref{prop:fluct} below) on the difference between the expectation of the maximum of partial times and the maximum of the expectations. The  \textcolor{black}{first line of the r.h.s. of \eqref{recursion_gen}} comes from approximating each partial passage time with its mean and using the induction hypothesis \eqref{nonparbound}.

\bigskip

{\color{black}
\textcolor{black}{After proving Proposition \ref{prop:induction_gen}, to pursue the proof of Proposition \ref{prop_block},
we will} bound the functions $g_n$ in terms of the reference function $\tau^{\rho_n,J_n}$\textcolor{black}{:}
\begin{equation}
\label{eq: parametric form}
g_n(\sigma, y) \leq \frac{\rho_n^\sigma - \sigma^- +y}{J_n}=\tau^{\rho_n^\sigma,J_n}(\sigma,y)\leq \tau^{\rho_n,J_n}(\sigma,y),
\end{equation}
where
\be
\label{better_J} 
J_n :=j_{n+1},
\quad 
\rho_n^\sigma :=\sup_{y\geq \sigma^-} \Big\{
j_{n+1} \,g_n(\sigma, y)-y
\Big\} + \sigma^-,\quad
\rho_n:=\max(\rho_n^1,\rho_n^{-1}).
\ee
\textcolor{black}{(note that \eqref{eq: parametric form} is a trivial consequence of \eqref{better_J}).}
We must then show that $\rho_n$ and $J_n$ satisfy (ii) and (iii) of Proposition \ref{prop_block}.
To this end, in Section \ref{sec:other_proofs}, we will prove Propositions \ref{prop:estimation_gen} and \ref{prop: estimation suites} below
(recall that $l_n$, $\Delta_n$ and $\rho_n$ actually depend on $\varepsilon$):

\begin{proposition}
\label{prop:estimation_gen}
Assume $\hh$ with $b\in[1,2)$.
Then for $\varepsilon$ small enough,
%
%
the sequence $(\rho_n^\sigma)_{n\in\N^*}$ defined in \eqref{better_J} satisfies
\be
\label{better_recursion}
\rho_{n+1}^\sigma \leq \frac{j_{n+2}}{j_{n+1}}\left[
\left( 1-\frac{1}{l_n}  \right)
\rho_n^\sigma +\frac{1}{l_n}+\Delta_n
\right],
\ee
where $\Delta_n=\Delta_n(\varepsilon)$ has the following property: there exist $\varepsilon_1>0$ and $C>0$ such that for every $0<\varepsilon\leq\varepsilon_1$ and $n\geq 1$
\be
\label{def_Delta_n}
\Delta_n \leq C \, j_{n+1}\frac{\delta_n^2}{2(j_{n+2}^{-1}-j_{n+1}^{-1})}
\left[\log\left(
\frac{\delta_n}{j_{n+2}^{-1}-j_{n+1}^{-1}}
\right)\right]^{3},
\ee
with $\delta_n$ as in \eqref{eq: delta n}.
\end{proposition}
 Assumption \hh\, ensures that the decay of $j_n$ to $r/4$ is slow enough so that the additional fluctuations of order $\Delta_n$ do not hinder property (ii) of Proposition \ref{prop_block}. 
\textcolor{black}{
Implicit in the statement of Proposition \ref{prop:estimation_gen} is the parameter $\gamma\in(0,1)$ defined by \eqref{def_l_gamma}, which regulates the speed of growth 
of our renormalization blocks. Recall that \eqref{eq: parameters} requires $\gamma$ small enough for good blocks to be typical.
}
For \textcolor{black}{Proposition \ref{prop:estimation_gen}} to be useful, $\gamma$ has to be chosen \textcolor{black}{possibly even closer} to 0, so that the upper bound \eqref{def_Delta_n}
vanishes in the limit $n \to \infty$. \textcolor{black}{This is the content of the next proposition,
which} will imply Proposition \ref{prop_block}.}
\begin{proposition}
\label{prop: estimation suites}
Assume $\hh$ with $b\in[1,2)$ and $\gamma <\inf \{ \gamma_0, \frac{2}{b} -1 \}$ where $\gamma_0$ is 
introduced in \eqref{eq: parameters}. 
Then  for small enough $\varepsilon$,  
the sequences $(\rho_n=\rho_n(\varepsilon))_{n\geq 1}$  and $(J_n=J_n(\varepsilon))_{n\geq 1}$ defined in \eqref{better_J}
satisfy the statements of Proposition \ref{prop_block}.
%
%
\end{proposition}
%
%
\textcolor{black}{
The following result established in Section \ref{sec:other_proofs} will
lead to Proposition \ref{prop_dilute_limit}.}
\begin{proposition}
\label{prop:passage_dilute}
{\color{black}
The dilute limit \eqref{dilute_g} holds as the sequence $g_n$ (defined in Proposition \ref{prop:induction_gen}) satisfies
\margincom{\color{black} Il faudrait un $y_1^\sigma$ ?\textcolor{black}{OK}}
\be
\label{dilute_g bis}
\textcolor{black}{\forall y \in [\sigma^- , \sigma^+y_1^1(0)-\sigma^-y_1^{-1}(0)[}, 
\qquad 
\limsup_{\varepsilon\to 0}\limsup_{n\to+\infty} g_n (\sigma,y) 
\leq(\sqrt{\sigma+y}+\sqrt{y})^2.
\ee
Recall that $g_n$ depends on $\varepsilon$ through the coarse graining scale.
}
\end{proposition}

\section{Proof of Proposition \ref{prop:induction_gen}}
\label{subsec:outline}

In this section, we prove the recursion in Proposition \ref{prop:induction_gen}.
To this end, we decompose a path of length $K_{n+1}$ according to its traces on the interfaces between the subblocks of size $K_n$ (see Figure \ref{fig: coarse grained}).
The set of such traces will hereafter be called the ``skeleton'' of the path.
The idea is to use  \eqref{nonparbound} as an induction hypothesis for the subpaths in each block of size $K_n$. If we neglect the fluctuations of these subpaths,
the ```mean'' computation reduces to  optimizing  the positions of the traces so as to maximize the total passage times of subpaths of  level $n$. 
%
This ``mean'' induction relation is altered by an error term (see Proposition \ref{prop:fluct} below) arising from fluctuations of the subpaths as well as the entropy induced by the many possible skeletons. 
%

\subsection{Skeleton decomposition}

\begin{figure}[h]
  \centering
\includegraphics[width=7cm, angle=270]{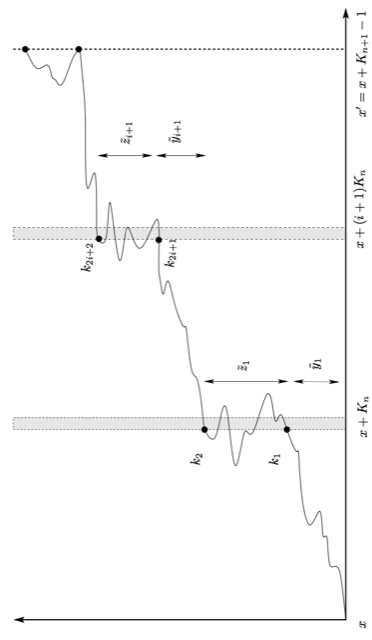}
\caption{\small A block of level $n+1$ is partitioned into blocks of length $K_n$ (only 3 blocks are depicted). The grey regions represent the boundaries between the blocks at  level $n$ which are separated by a microscopic length 1.
A coarse grained path is depicted and the black dots denote the renewal points $k_i$.} 
  \label{fig: coarse grained} 
\end{figure}

We consider $B=[x,x'=x+K_{n+1}-1]\cap\Z$ a block of order $n+1$, where $x\in\Z$. 
Let $\gamma= ( (x_k,y_k) )_{k=0,\ldots,m-1}$
be a path restricted to $B$ connecting $(x,0)=(x_0,y_0=0)$ to $(x',y'=[K_{n+1}y])=(x_{m-1},y_{m-1})$.
We define the {\it skeleton} $s(\gamma)=\tilde{\gamma}$ of $\gamma$ as follows
(see figure \ref{fig: coarse grained}). 
Let $k_0=-1$ and $y_{-1}=0$. For $i\in\textcolor{black}{\{0,\ldots,l_n-1\}}$, we set
\begin{eqnarray*}
k_{2i+1} & : = & \min\{k>k_{2i}:\,x_k=x+(i+1)K_n-1\},\\
k_{2i+2} & : = & \max\{k\geq k_{2i+1}:\,x_k=x+(i+1)K_n-1\}.
\end{eqnarray*}
\textcolor{black}{
Note that $k_{2i+1}$ and $k_{2i+2}$ are finite because our path connects $(x,0)$ to $(x',y')$ with horizontal increments $\pm 1$.
}
Because $x_{k+1}-x_k\leq 1$, we necessarily have $x_{1+k_{2i+2}}=x+(i+1)K_n$ and $y_{1+k_{2i+2}}=y_{k_{2i+2}}$. Note that $x_{k_{2l_n-1}}=x_{k_{2l_n}}=x'$ and $y_{k_{2l_n}}=y'$.
Recall that the block $B$ is made of $l_n = \lfloor K_n^\gamma \rfloor$ boxes of length $K_n$.
The skeleton $s(\gamma)$ of $\gamma$ is then the sequence $\tilde{\gamma}=(\tilde{y}_i,\tilde{z}_i)_{i=1,\ldots,l_n}\in(\N^2)^{l_n}$ given by
\begin{eqnarray}
\tilde{y}_i & := & y_{k_{2i-1}}-y_{k_{2i-2}}\label{skeleton_y},\\
\tilde{z}_i & : = & y_{k_{2i}}-y_{k_{2i-1}}\label{skeleton_z}.
\end{eqnarray}
By definition, we have 
\be\label{total_height}
\sum_{i=1}^{l_n}\left(\tilde{y}_i+\tilde{z}_i\right)=y'=  \lfloor K_{n+1} y\rfloor.
\ee
In a similar way for the paths going from right to left, if $B=[x'=x-K_{n+1}+1,x]\cap\Z$, we may define the skeleton of a path connecting $(x,0)=(x_0,y_0=0)$ to $(x',y'=[K_{n+1}y])=(x_{m-1},y_{m-1})$. 
Let $k_0=-1$ and $y_{-1}=-1$. For $i\in\N$, let
\begin{eqnarray*}
k_{2i+1} & : = & \min\{k>k_{2i}:\,x_k=x-(i+1)K_n+1\} , \\
k_{2i+2} & : = & \max\{k\geq k_{2i+1}:\,x_k=x-(i+1)K_n+1\}.
\end{eqnarray*}
Because $x_{k+1}-x_k\geq -1$, we necessarily have $x_{1+k_{2i+2}}=x-(i+1)K_n$ and $y_{1+k_{2i+2}}=1+y_{k_{2i+2}}$. Note that $x_{k_{2l_n-1}}=x_{k_{2l_n}}=x'$ and $y_{k_{2l_n}}=y'$.
The skeleton $s(\gamma)$ of $\gamma$ is then the sequence $\tilde{\gamma}=(\tilde{y}_i,\tilde{z}_i)_{i=1,\ldots,l_n}$ given by
\eqref{skeleton_y}--\eqref{skeleton_z}. Since allowed path increments are $(1,0)$ and $(-1,1)$, this sequence must now satisfy the constraint
$\tilde{y}_i\geq K_n$ for $i\geq 1$. 
\\ \\
Let $\tilde{\Gamma}_n((x,0),(x',y'))$ denote the set of skeletons of all paths $\gamma$ restricted to $B$ connecting $(x,0)$ and $(x',y')$, that is
the set of sequences $\tilde{\gamma} = (\tilde{y}_i,\tilde{z}_i)_{i=1,\ldots,l_n}\in(\N^2)^{\{1,\ldots,l_n \} }$ satisfying \eqref{total_height}, with the
constraint $\tilde{y}_i\geq K_n$ in the case $x'<x$. We will simply write $\tilde{\Gamma}_n$ when the endpoints are obvious from the context.
\subsection{\textcolor{black}{P}assage time decomposition}
Let $\sigma = \pm 1$ denote as in \eqref{eq: path direction} the direction of the paths.
To encompass both cases $\sigma=\pm 1$, we will use the following simplifying  convention: an interval can be written $[a,b]$ even if $a>b$, in which case it actually means $[b,a]$.
From now on, for notational simplicity, we consider the block $B = [ 0, \sigma(K_{n+1}-1)]$, instead of a block with arbitrary position $x\in\Z$.
For $l\in\{1,\ldots,l_n\}$, we denote by
\textcolor{black}{$B_l:=[\sigma(l-1)K_n,\sigma(lK_n-1)]\cap\Z$} the $l$-th subblock of level $n$ in the decomposition of $B$.
For a path skeleton $\tilde{\gamma}=(\tilde{y}_l,\tilde{z}_l)_{l=1,\ldots,l_n}\in\tilde{\Gamma}_n$, define 
$$
\tilde{h}_i:=\sum_{j=1}^{i-1}\left[\tilde{y}_j+\tilde{z}_j \right]
$$
if $i\geq 2$ and $\tilde{h}_1=0$. The quantity $\tilde{h}_i$ represents the height at which a path with skeleton $\tilde{\gamma}$ enters block $i$ without ever returning to block $i-1$.
\begin{figure}[htpb]
\includegraphics[width=7cm, angle = 270]{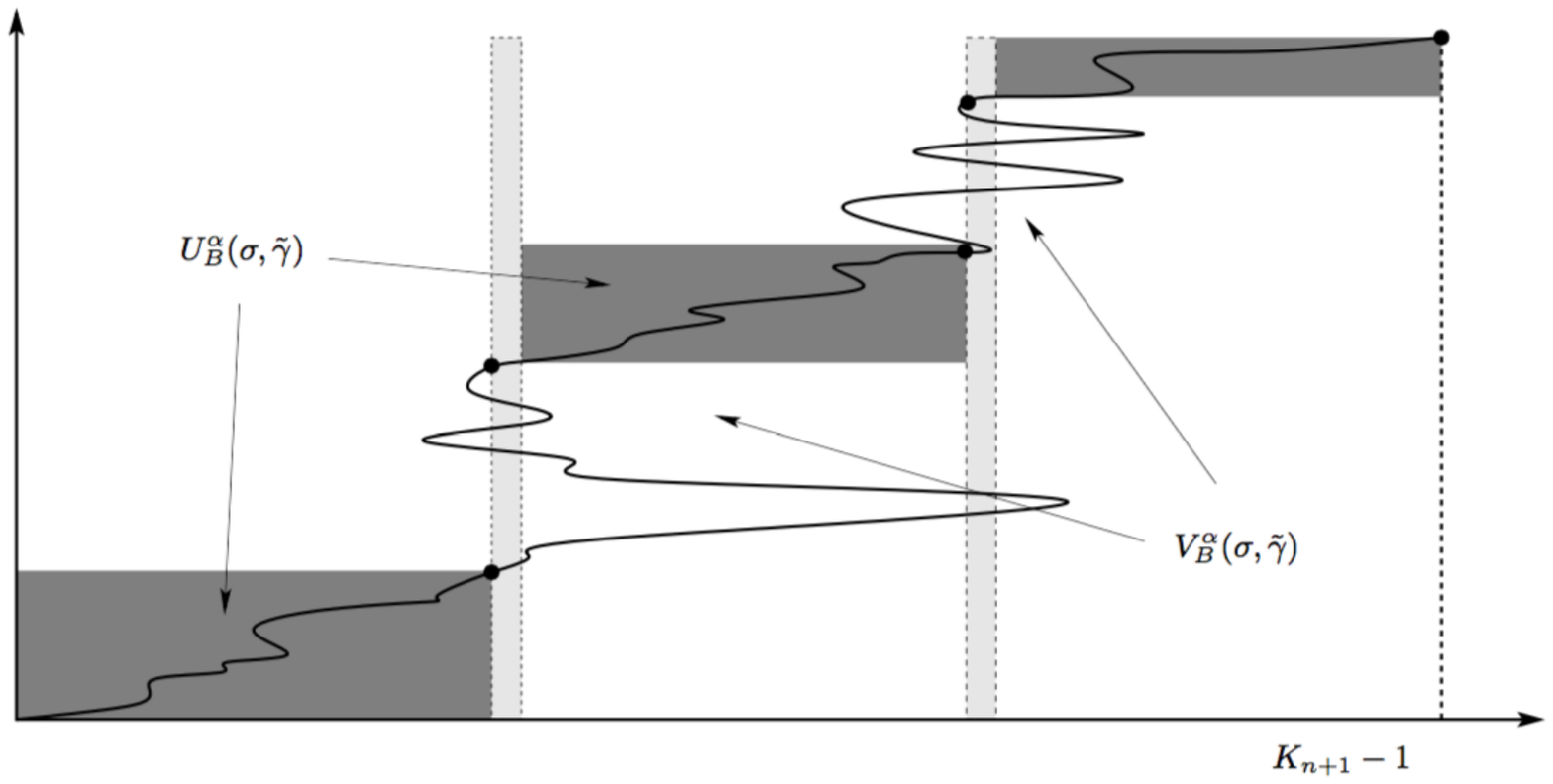}
\caption{\small A coarse grained path is depicted in a block of order $n+1$. 
The horizontal crossings through each block $B_l$ are restricted to the dark grey regions. The passage time $U^\alpha_B(\sigma,\tilde \gamma)$ depends only on the variables $\{ Y_{i,j} \}$ inside the grey regions which are disjoint from the regions used by the vertical paths contributing to  
$V^\alpha_B(\sigma,\tilde \gamma)$.} 
  \label{fig: independence} 
\end{figure}
%
%
For a path $\gamma\in\Gamma_B((0,0),(\sigma(K_{n+1}-1),y'))$  with skeleton $\tilde{\gamma}$, we have that 
\be
T^\alpha_B(\gamma)\leq U^\alpha_{B}(\sigma, \tilde{\gamma})+V^\alpha_{B}(\sigma, \tilde{\gamma}) \leq T^\alpha_B((0,0),( \sigma (K_{n+1}-1),y' ) \big),
\label{decomp_block}
\ee
where
$$
U^\alpha_{B}(\sigma,\tilde{\gamma})  := 
\sum_{l=1}^{l_n} U^\alpha_{B,l}(\sigma,\tilde{\gamma}),\quad
%
%
V^\alpha_{B}(\sigma,\tilde{\gamma}) :=  
\sum_{l=1}^{l_n} V^\alpha_{B,l}(\sigma,\tilde{\gamma}) ,
%
$$
with 
\begin{eqnarray}
U^\alpha_{B,l}(\sigma,\tilde{\gamma}) & := &
T^\alpha_{B_l} \big( \sigma (l-1)K_n,\tilde{h}_l),(\sigma (lK_n-2),\tilde{h}_l+\tilde{y}_l+\sigma-1) \big),
\nonumber \\
V^\alpha_{B,l}(\sigma,\tilde{\gamma})  &  := &  T^\alpha_{B} \big( (\sigma(lK_n-1),\tilde{h}_l+\tilde{y}_l+\frac{\sigma-1}{2}),(\sigma (lK_n-1),\tilde{h}_l+\tilde{y}_l+\frac{\sigma-1}{2}+\tilde{z}_l) \big),
\nonumber\\
& & \label{eq: def U V}
\end{eqnarray}
where $U^\alpha_{B}(\tilde{\gamma})$ is the contribution of the horizontal crossings in the 
blocks $B_l$ and  $V^\alpha_{B}(\tilde{\gamma})$ the contribution of the vertical paths at the 
junction of the blocks $B_l$ (see \textcolor{black}{F}igure \ref{fig: independence}). 
Noticing that the second inequality in \eqref{decomp_block} is an equality if and only if $\tilde{\gamma}$ is the skeleton
of the optimal path, we get
\be
%
T^\alpha_B((0,0),(\sigma(K_{n+1}-1),y'))  =  \max_{
\tilde{\gamma}\in\tilde{\Gamma}_n((0,0),(\sigma(K_{n+1}-1),y'))
}\left\{
U^\alpha_{B}(\sigma,\tilde{\gamma})+V^\alpha_{B}(\sigma,\tilde{\gamma})
\right\}.
\label{double_max}
\ee
To derive Proposition  \ref{prop:induction_gen}, we have to estimate 
\begin{equation}
\label{eq: objectif}
\tau^\alpha_{n+1,B}(\sigma,y)
:= \frac{1}{K_{n+1}} \Exp \Big( T^\alpha_B((0,0),(\sigma(K_{n+1}-1),y')) \Big)
\end{equation}
with  $y' = \lfloor K_{n+1} y\rfloor$ and  we decompose this expectation into the sum of two components:
\be\label{mean_opti}
\max_{
\tilde{\gamma}\in\tilde{\Gamma}_n((0,0),(\sigma(K_{n+1}-1),y'))
}
\left\{
\Exp \left( \frac{U^\alpha_{B}(\sigma,\tilde{\gamma})}{K_{n+1}} \right)
+\Exp \left( \frac{V^\alpha_{B}(\sigma,\tilde{\gamma}) }{K_{n+1}} \right)
\right\},
\ee
that is the ``mean optimization problem" and a  ``fluctuation part'' defined for  
$y' = \lfloor K_{n+1} y\rfloor$ as
\begin{eqnarray}
& {\mathcal F}_n (y) =  \Exp \left( \frac{1}{K_{n+1}}\max_{
\tilde{\gamma}\in\tilde{\Gamma}_n((0,0),(\sigma(K_{n+1}-1),y'))
}\Big\{
U^\alpha_{B}(\sigma,\tilde{\gamma})+V^\alpha_{B}(\sigma,\tilde{\gamma})
\Big\}  \right) 
\nonumber\\
& \qquad \qquad \qquad \qquad 
-  \max_{ \tilde{\gamma}\in\tilde{\Gamma}_n((0,0),(\sigma(K_{n+1}-1),y'))
}
\left\{
\Exp \left( \frac{U^\alpha_{B}(\sigma,\tilde{\gamma})}{K_{n+1}} \right)
+\Exp \left( \frac{V^\alpha_{B}(\sigma,\tilde{\gamma}) }{K_{n+1}} \right)
\right\}.
\label{fluct_part}
\end{eqnarray}

The term \eqref{mean_opti}, which involves known information from subblocks, will give the main recursion structure, 
while  \eqref{fluct_part} will be an error term. The latter will be controlled by fluctuations and entropy of paths.
The precise result that will be established in Section \ref{sec: prop passage time} is the following:%
\begin{proposition}
\label{prop:fluct}
With the notation \eqref{fluct_part}, one has uniformly in $y$
\begin{eqnarray}
{\mathcal F}_n (y)
\leq
\delta_n\sqrt{\frac{\sigma}{2}+y} \,\big( 1 +\log(1+y) \big)^{3/2},
\label{eq:fluct}
\end{eqnarray}
with $\delta_n$ defined in \eqref{eq: delta n}.
%
%
\end{proposition}

We will in fact replace the upper bound in \eqref{eq:fluct} 
by a slightly worse one
for the sole purpose of making it a concave function of $y$, which is important for us. We therefore observe that
\be
\label{concave_bound}
\mathcal F_n(y)\leq\mathcal G_n(y):=\delta_n\varphi(y),
\ee
where $\varphi$ is the function defined by \eqref{eq: phi}.
{\color{black} The concavity of $\varphi$ can be checked by a straightforward, but tedious computation.}

\subsection{The {\bf main recursion} \eqref{recursion_gen}}
Using the skeleton decomposition, we are now going to derive Proposition \ref{prop:induction_gen}.
Let us explain the choice \eqref{level_1} of $g_1$ in Proposition \ref{prop:induction_gen}.
To initiate the induction relation, we need a bound at level $1$ for $\rho_1$ and $J_1$. 
For $n=1$, a good block {\color{black} at level 1} contains only rates $\alpha(x) =1$.
Since the restricted passage times are smaller than the unrestricted ones, 
and the latter are superadditive, the asymptotic shape \eqref{def_limheight_k}--\eqref{eq: homogeneous LPP} of the homogeneous last passage percolation 
yields the exact upper bound
$$
\tau^\alpha_{1,B}(\sigma,y)\leq\left(\sqrt{\sigma+y}+\sqrt{y}\right)^2 =:g_1(\sigma,y).
$$
By  definition, $g_1$ is concave. 
Note that, if $g_n(\sigma,.)$ is   concave, then $g_{n+1}(\sigma,.)$ defined by \eqref{recursion_gen} inherits this property.

\bigskip

Suppose now that the inequality \eqref{nonparbound} 
$$
\tau^\alpha_{n,B}(\sigma,y)\leq g_n(\sigma,y)
$$  
holds at step $n$ and that $g_n$ is concave. We will show that the recursion is valid at step $n+1$ with $g_{n+1}$ defined as in \eqref{recursion_gen}.

We first focus on the mean optimization problem \eqref{mean_opti} and consider a good block $B = [0
, \sigma (K_{n+1} - 1) ]$
at level $n+1$.
For a fixed disorder $\alpha$, by \textcolor{black}{\eqref{def_mintime} and \eqref{max_current_lpp}},
%
\be
\label{vertical_bound}
\Exp\left[V^\alpha_{B,l}(\sigma,\tilde{\gamma})\right]
\leq 
\frac{1}{j_{\infty,B}} \; \tilde{z}_l.
\ee
Since $B$ is a good block, $j_ {\infty,B}$ satisfies \eqref{condition_good}. 
Thus 
\be
\label{eq: def j n+1}
j_{\infty,B}\geq\frac{r}{4}+\frac{a}{K_{n+1}^{b/2}}=:j_{n+1},
\ee
(recall that $j_{n+1}$ was introduced in \eqref{condition_good} as one of the conditions defining a good block).
%
%
As $B$ is a good block, the subblocks $B_l$ are good for all values of $l=1,\ldots,l_n$ except for 
possibly one bad subblock with index $i_0$.
The recurrence hypothesis \eqref{nonparbound} at level $n$
implies that the mean passage time on a  good subblock $B_l$ is bounded by 
\be
\label{induction}
\Exp\left[U^\alpha_{B,l}(\sigma,\tilde{\gamma})\right]= K_n\tau^\alpha_{n,B_l}\left(
\sigma,\frac{\tilde{y}_l}{K_n} \right)
\leq 
K_n g_n\left(
\sigma,\frac{\tilde{y}_l}{K_n} \right).
%
%
%
%
\ee
For the possibly remaining value $i_0$ such that $B_{i_0}$ is a bad block, we use  a crude upper bound 
by artificially extending the path in order to compare its cost to the one of a vertical connection:
$$
U^\alpha_{B,i_0}(\sigma,\tilde{\gamma})\leq 
T^\alpha_{B_{i_0}}\left(\left(\sigma(i_0-1)K_n, \tilde{h}_{i_0} \right),\left(\sigma(i_0-1)K_n,
\tilde{h}_{i_0} +\tilde{y}_{i_0}+\frac{1+\sigma}{2}K_n\right)\right),
$$
%
which yields, as in \eqref{vertical_bound},
\be\label{crude_vertical_bound}
\Exp\left[U^\alpha_{B,i_0 }(\sigma,\tilde{\gamma})\right]
\leq \frac{1}{j_{\infty,B_{i_0}}} \; \left(\tilde{y}_{i_0}+\frac{1+\sigma}{2}K_n\right)
\leq \frac{1}{j_{n+1}} \left(\tilde{y}_{i_0} +\frac{1+\sigma}{2}K_n\right).
\ee
{\color{black} 
\textcolor{black}{
In \eqref{crude_vertical_bound} we used the inequality $j_{\infty,B_{i_0}}\geq j_{\infty,B}$, which follows from $B_{i_0}\subset B$, Lemma \ref{lemma_reform} and definition \eqref{restricted_time}. We combined this inequality
with the bound \eqref{eq: def j n+1} for $j_{\infty,B}$.
}
Note that  if there is no bad subblock, we will still apply 
\eqref{crude_vertical_bound} to an arbitrarily chosen subblock to avoid distinguishing this seemingly better case,
which ultimately would not improve our result.}
Combining the above expectation bounds, we obtain
\be
\label{mean_time_bound}
\Exp\left[
U^\alpha_{B}(\sigma,\tilde{\gamma})
+V^\alpha_{B}(\sigma,\tilde{\gamma})
\right]\leq K_{n+1}  \; g_{n+1}^{(1)}(\sigma,y,\tilde{\gamma}),
\ee
where 
\begin{equation}
\label{eq: onethird}
g_{n+1}^{(1)}(\sigma,y,\tilde{\gamma}):=
\frac{1}{l_n}\left\{
\sum_{l=1,\,l\neq i_0}^{l_n}g_n(\sigma,\bar{y}_l)+\frac{1}{j_{n+1}}\left[
\frac{1+\sigma}{2}+\bar{y}_{i_0}+\sum_{l=1}^{l_n} \bar{z}_l
\right]
\right\},
\end{equation}
where $(\bar{y}_l,\bar{z}_l)_{l=1,\ldots,l_n}\in[0,+\infty)^{2l_n}$ is the rescaled skeleton defined by
$\bar{y}_l=K_n^{-1}\tilde{y}_l$ and $\bar{z}_l=K_n^{-1}\tilde{z}_l$, which satisfies the constraint \eqref{total_height}, whence
\be\label{continuum_skeleton}
\sum_{l=1}^{l_n}\left(
\bar{y}_l+\bar{z}_l
\right)
\leq l_ny
\quad \text{with} \quad 
\bar{y}_l\geq\sigma^- .
\ee
Define
\begin{equation}
\label{eq: bar y}
\sigma^-\leq\, \bar{y}:=\frac{1}{l_n-1} \sum_{l=1,\ldots,l_n:\,l\neq i_0}\bar{y}_l
\leq  \frac{l_n}{l_n-1} y,
\end{equation}
so that from \eqref{continuum_skeleton}, we have
\be
\label{total_height_macro}
\bar{y}_{i_0}+\sum_{l=1}^{l_n}\bar{z}_l
\leq l_n y-(l_n-1)\bar{y} .
\ee
By concavity of $g_n$, \eqref{mean_time_bound}--\eqref{eq: onethird} and \eqref{total_height_macro}, we obtain an upper bound for \eqref{mean_opti}:
\begin{align}
\label{mean_time_bound_gen}
& \max_{
\tilde{\gamma}\in\tilde{\Gamma}_n((0,0),(\sigma(K_{n+1}-1),[K_{n+1}y]))
}
\left\{
\Exp \left( \frac{U^\alpha_{B}(\sigma,\tilde{\gamma})}{K_{n+1}} \right)
+\Exp \left( \frac{V^\alpha_{B}(\sigma,\tilde{\gamma}) }{K_{n+1}} \right)
\right\}\\
& \qquad \qquad\qquad \leq  \sup_{\sigma^- \leq \bar{y}\leq \frac{l_n}{l_n -1} y}
\left\{ 
\left(1-\frac{1}{l_n}\right)
\left[
g_n(\sigma,\bar{y}) - \frac{\bar{y}}{j_{n+1}} 
\right] \right\}
+ \frac{y}{j_{n+1}} + \frac{1+\sigma}{2l_nj_{n+1}} ,
\nonumber
\end{align}
where the value of $\bar y$ in \eqref{eq: bar y} has been replaced by a supremum.
To bound from above 
$\tau^\alpha_{n+1,B}(\sigma,y)$ (see  \eqref{eq: objectif}), it is enough to combine 
\eqref{mean_time_bound_gen} and Proposition \ref{prop:fluct}.
This completes the proof of Proposition \ref{prop:induction_gen}.

\section{Consequences of the main recursion}
\label{sec:other_proofs}
In this section, we prove Propositions \ref{prop:estimation_gen}, \ref{prop: estimation suites} 
and  \ref{prop:passage_dilute}.

\subsection{Proof of Proposition \ref{prop:estimation_gen}}\label{subsec:conseq_1}
%
As $g_1(\sigma,.)$ is concave, 
the recursion \eqref{recursion_gen} implies that $g_{n}(\sigma,.)$ is a concave function for all $n$. 
For notational simplicity, we shall write details of the proof for $\sigma=1$.  In this case, we simply write $g_n(.)$ for $g_n(\sigma,.)$ and $\rho_n$ for $\rho_n^\sigma$.
We will only briefly indicate what changes are involved for $\sigma=-1$.
We consider the sequence $(g_n)_{n \geq 1}$ given by the  recursion \eqref{recursion_gen} and set
\be
\label{def_yn}
y_n:=\inf\left\{
y\geq 0:\quad 
g'_n(y) \leq  \frac{1}{j_{n+1}} 
\right\},
\ee
where $g'_n$ stands for the right derivative of the concave function. Thus, if $y\geq(1-l_n^{-1})y_n$
\begin{equation}
 g_{n+1}(y) =  \left(
1-\frac{1}{l_n}
\right)\left[
g_n\left(y_n\right)- \frac{y_n}{j_{n+1}}
\right]
 +  \frac{y}{j_{n+1}} + \frac{1}{l_nj_{n+1}} +\delta_n \varphi(y),
\label{gn_1} 
\end{equation}
and if $y\leq(1-l_n^{-1})y_n$
\begin{equation}
 g_{n+1}(y) =  \left(
1-\frac{1}{l_n}
\right)\left[
g_n \left( \frac{l_n}{l_n-1} y \right)-\frac{l_n}{l_n-1} \frac{y}{j_{n+1}}
\right] +
 \frac{y}{j_{n+1}} + \frac{1}{l_nj_{n+1}} +\delta_n \varphi(y).
\label{gn_2}
\end{equation}
\begin{lemma}
\label{lemma_yn}
Assume $\hh$ with $b\in[1,2)$. Then for $\varepsilon$ small enough, the sequence $(y_n)_{n\geq 1}$ satisfies
\be
\label{value_yn}
\forall n \geq 2, \qquad 
 y_{n-1} \leq y_n  < \infty
\quad \text{and} \quad
{\varphi'(y_n)}=\frac{
j_{n+1}^{-1}-j_{n}^{-1}
}{\delta_{n-1}},
\ee
with $\varphi$ as in \eqref{eq: phi}.
\end{lemma}
\begin{proof}
The proof of  \eqref{value_yn} is split in 3 steps.

\smallskip

\noindent
{\it Preliminary computations.} For $n\geq 1$, we set
\begin{equation}
\label{eq: tn}
t_{n+1}:=
\frac{j_{n+2}^{-1}-j_{n+1}^{-1}}{\delta_{n}}=\psi_3(K_n),
\end{equation}
with
\begin{eqnarray*}
\psi_3(K) & := & (1+\gamma)^{-3/2} \frac{K^{1/2}}{
(\log K)^{3/2}}\left\{
\left(
\frac{r}{4}+\frac{a}{K^{\frac{b}{2}(1+\gamma)^2}}
\right)^{-1}-\left(
\frac{r}{4}+\frac{a}{K^{\frac{b}{2}(1+\gamma)}}
\right)^{-1}
\right\}\\
& \stackrel{K\to+\infty}{\sim} & (1+\gamma)^{-3/2}\frac{16a}{r^2
}K^{\frac{1}{2}-\frac{b}{2}(1+\gamma)}
\stackrel{K\to+\infty}{\longrightarrow} 0.
\end{eqnarray*}
Since $b\geq 1$ 
{\color{black} and $K_1(\varepsilon)$ diverges in the dilute limit  \eqref{conclusion_lemma_good}}, we conclude that 
\be\label{conclude_tn}
\lim_{\varepsilon\to 0}\sup_{n\geq 1}t_{n+1}(\varepsilon)=0.
\ee

\noindent
{\it 
Case $n=2$.}
Since $y_1>(1-l_1^{-1})y_1$, $g'_2(1,y_1)$ is obtained by differentiating \eqref{gn_1}:
$$
g'_2(1,y_1)-j_{3}^{-1}={j_2}^{-1}-j_3^{-1}+\delta_1\varphi'(y_1)=\psi_4(K_1),
$$
where $\psi_4(K_1)>0$, for large $K_1$, because as $K_1\to+\infty$, we have
$$
j_2^{-1}-j_3^{-1}\sim -K_1^{-b(1+\gamma)/2}
\quad \text{with} \quad 
\frac{b(1+\gamma)}{2} > \frac{1}{2}
\quad \text{and} \quad 
\delta_1=C\frac{(\log K_{1})^{3/2}}{K_1^{1/2}},
$$
{\color{black}
(recall  $b\geq 1$ and $\gamma>0$). Thus $y_2>y_1>(1-l_1^{-1})y_1$ as $g'_2(1,y_1)>j_{3}^{-1}$ for $\varepsilon$ small enough.}
Hence, for $y$ in the neighborhood of $y_2$, $g'_2(1,y)$ is also obtained by differentiating the expression \eqref{gn_1}.
It follows that
$$
y_2=\inf\left\{
y\geq 0:\qquad \varphi'(y)\leq \frac{j_3^{-1}-j_2^{-1}}{\delta_1}
\right\}.
$$
Since $\varphi$ is strictly concave and $\lim_{y\to +\infty}\varphi'(y)=0$,  \eqref{conclude_tn} implies that for $\varepsilon$ small enough, $y_2$ is the unique solution of $\varphi'(y_2)=t_2$.
Thus identity \eqref{value_yn} holds for $n=2$.

\medskip

\noindent 
{\it Case $n>2$.} We are going to prove the  claim  by induction.
Suppose that \eqref{value_yn} is valid up to rank $n$. 
To show $y_{n+1} \geq y_n$, it is enough to check that
$$
g'_{n+1}(y_n)>j_{n+2}^{-1}.
$$
Since $y_n>(1-l_n^{-1})y_n$, the above derivative is computed from the expression \eqref{gn_1}. Thus, using the induction hypothesis \eqref{value_yn}, we get for $n\geq 2$
\begin{eqnarray}\nonumber
g'_{n+1}(y_n)-j_{n+2}^{-1} & = &  j_{n+1}^{-1}-j_{n+2}^{-1}+{\delta_n}\varphi'(y_n)
= j_{n+1}^{-1}-j_{n+2}^{-1}+\frac{\delta_n}{\delta_{n-1}}\left(j_{n+1}^{-1}-j_n^{-1}\right)\\
& = & \psi(K_{n-1}) ,
\label{induction_gn}
\end{eqnarray}
 where, since $K_{n}=K_{n-1}^{1+\gamma}$,
\begin{eqnarray}
\nonumber\psi(K) & = & \left(
\frac{r}{4}+\frac{a}{K^{\frac{b}{2}(1+\gamma)^2}}
\right)^{-1}-\left(
\frac{r}{4}+\frac{a}{K^{\frac{b}{2}(1+\gamma)^3}}
\right)^{-1}\\
& + & (1+\gamma)^{3/2}K^{-\gamma/2}\left[
\left(
\frac{r}{4}+\frac{a}{K^{\frac{b}{2}(1+\gamma)^2}}
\right)^{-1}-\left(
\frac{r}{4}+\frac{a}{K^{\frac{b}{2}(1+\gamma)}}
\right)^{-1}
\right].
\label{def_psi}
%
\end{eqnarray}
Let us respectively denote by $\psi_1(K)$ and $\psi_2(K)$ the first and second line on the r.h.s. of \eqref{def_psi}.
Then as $K\to+\infty$, 
$$\psi_1(K)\sim-16ar^{-2}K^{-b(1+\gamma)^2/2},\quad
\psi_2(K)\sim16a(1+\gamma)^{3/2} \, r^{-2}K^{-b(1+\gamma)/2-\gamma/2}.
$$
Since for $b\geq 1$ and $\gamma>0$ we have
$$
\frac{b}{2}(1+\gamma)+\frac{\gamma}{2}<\frac{b}{2}(1+\gamma)^2.
$$
It follows that $\psi(K)>0$ for $K$ large enough. 
As $K_1(\varepsilon)$ diverges when  $\varepsilon$ tends to 0 (see  \eqref{conclusion_lemma_good}), 
we have that for small enough $\varepsilon$,
$y_{n+1} \geq y_n\geq(1-l_n^{-1})y_n$ holds for all $n\geq 2$.\\ \\
As $g'_{n+1}(y_{n+1})$ is given by the derivative of \eqref{gn_1} and $\varphi$ is strictly concave, we have to solve
\begin{equation}
\label{eq: solve yn+1}
g'_{n+1}(y_{n+1})=j_{n+1}^{-1}+{\delta_n}\varphi'(y_{n+1})=j_{n+2}^{-1}
\quad \Rightarrow \quad
\varphi'(y_{n+1})= \frac{j_{n+2}^{-1} - j_{n+1}^{-1}}{\delta_n} = t_{n+1}.
\end{equation}
%
%
As above, \eqref{conclude_tn} implies that, for $\varepsilon$ small enough, a solution of \eqref{eq: solve yn+1} exists for all $n\geq 2$. This proves the second part of the claim \eqref{value_yn}.  The proof is similar for $\sigma=-1$.
%
%
%
%
\end{proof}
%
%
%
Using Lemma \ref{lemma_yn}, 
we can now  complete the proof of Proposition 
\ref{prop:estimation_gen}.
We must show that  inequality \eqref{better_recursion} holds for the sequence $(\rho_n)_{n\in\N^*}$ with $\Delta_n$ satisfying \eqref{def_Delta_n}.
%
%
By definition \eqref{def_yn} of $y_n$, 
{\color{black} the supremum in \eqref{better_J} is reached at $y_n$ so that} 
$$
\rho_n=j_{n+1}g_n(y_n)-y_n.
$$
We are going to obtain a recursion for $\rho_n$. To this end, consider
$$
\rho_{n+1}=j_{n+2}g_{n+1}(y_{n+1})-y_{n+1}.
$$
By Lemma \ref{lemma_yn}, $y_{n+1} \geq y_n>(1-l_n^{-1})y_n$, so $g_{n+1}(y_{n+1})$ is obtained from \eqref{gn_1}. Thus
\begin{eqnarray}
\rho_{n+1} & =  j_{n+2}\left(1-\frac{1}{l_n}\right)\left[
g_n(y_n)- \frac{y_n}{j_{n+1}}
\right]+ \frac{j_{n+2}}{l_n \, j_{n+1}}+ \left( \frac{j_{n+2}}{j_{n+1}} - 1 \right)  y_{n+1}
+ j_{n+2}\delta_n\varphi(y_{n+1})\nonumber\\
& \leq  \frac{j_{n+2}}{j_{n+1}} \left( \left(1-\frac{1}{l_n}\right)\left[
j_{n+1}g_n(y_n)- {y_n}
\right]+ \frac{1}{l_n} +j_{n+1}\delta_n\varphi(y_{n+1})
\right),
\label{sn_part_one}
\end{eqnarray}
where on the second line we have used $j_{n+2}\leq j_{n+1}$. Setting ${\Delta}_n:=j_{n+1}\delta_n\varphi(y_{n+1})$, we recovered the inequality  \eqref{better_recursion}, and it remains to verify \eqref{def_Delta_n}.
Starting from 
$$
\varphi'(y)\stackrel{y\to+\infty}{\sim}
\frac{1}{2\sqrt{y}}
(\log y)^{3/2},
$$
we see that
\be
\label{phi_prime}
\varphi[\varphi^{'-1}(t)]\stackrel{t\to 0}{\sim}\frac{1}{2t}
\left(\log\frac{1}{4t^2}\right)^3.
\ee
Recall that 
by \eqref{eq: solve yn+1}, $y_{n+1}=\varphi^{'-1}(t_{n+1})$, where
$t_n$ is defined by \eqref{eq: tn} and satisfies \eqref{conclude_tn}.
Thus,
there exist $C',C''>0$ and
$\varepsilon_2>0$ such that, for every $0<\varepsilon\leq\varepsilon_2$ and $n \geq 1 $
$$
\varphi(y_{n+1})=\varphi[\varphi^{'-1}(t_{n+1})]\leq C''\frac{1}{t_{n+1}} \, \big|\log t_{n+1} \big|^{3}
\leq C' \frac{\delta_n}{2(j_{n+2}^{-1}-j_{n+1}^{-1})}
\left[
\log\left(
\frac{\delta_n}{j_{n+2}^{-1}-j_{n+1}^{-1}}
\right)
\right]^3.
$$
This implies \eqref{def_Delta_n} with  ${\Delta}_n=j_{n+1}\delta_n\varphi(y_{n+1})$.

\medskip

For $\sigma = -1$, 
still writing $\rho_n$ for $\rho_n^{\sigma}$, we have
$$\rho_n - 1= \sup_{y \geq 1}
\big\{ j_{n+1} g_n(-1,y) - y \big\}$$
and we get a  recursion similar to \eqref{sn_part_one}   
\begin{eqnarray*}
\rho_{n+1} - 1 \leq  \frac{j_{n+2}}{j_{n+1}} \left( \left(1-\frac{1}{l_n}\right)\left[
\rho_n - 1
\right]  + j_{n+1}\delta_n\varphi(y_{n+1})
\right),
\end{eqnarray*}
which can be rewritten
\begin{eqnarray*}
\rho_{n+1} \leq  \frac{j_{n+2}}{j_{n+1}} \left( \left(1-\frac{1}{l_n}\right) \rho_n 
+ \frac{1}{l_n} +
j_{n+1}\delta_n\varphi(y_{n+1})
+ \frac{j_{n+1}}{j_{n+2}} -1
\right).
\end{eqnarray*}
For $b \geq 1$, the remainder $\frac{j_{n+1}}{j_{n+2}} -1$ can be bounded by $\Delta_n$ so that the same type of inequality is also valid for $\sigma = -1$.

\subsection{Proof of Proposition \ref{prop: estimation suites}}
\label{proof_seq}
%
Let $a_n=1-\frac{1}{l_n}$. Then one can see by induction that 
 \eqref{better_recursion} implies
\begin{eqnarray}
\rho_n^\sigma & \leq & \rho_1^\sigma \prod_{i=1}^{n-1}a_i+\left(
1-\prod_{i=1}^{n-1}a_i
\right)
+\sum_{i=1}^{n-1}\Delta_i\prod_{j=i+1}^{n-1}a_j\nonumber\\
& \leq & \rho_1^\sigma \prod_{i=1}^{n-1}a_i+\left(
1-\prod_{i=1}^{n-1}a_i
\right)
+\sum_{i=1}^{n-1}\Delta_i ,
\label{induction_rho}
\end{eqnarray}
where we used that $\frac{j_{i+1}}{j_i}\leq 1$ for any $i \geq 1$.
Remember that the quantities $\rho_n$, $j_n$, $a_n$, $\Delta_n$ actually depend on $\varepsilon$.
Since $g_1$ is given by \eqref{level_1}, a simple computation shows that
\begin{equation}\label{eq:rho1}
\rho_1^1 :=\sup_{y\geq 0} \Big\{ j_2 \, g_1(1,y)-y \Big\} 
\end{equation}
is the smaller root $\rho$ of the equation
\begin{equation}\label{eq:j2}
\rho(1-\rho)=j_2,
\end{equation}
and that the supremum in (\ref{eq:rho1}) is achieved at 
{\color{black} 
$y_1^1:=\frac{\rho_1^2}{1-2\rho_1}.$
For $\sigma = -1$, we have
}
\begin{equation}
\label{eq:rhoprime1}
\rho_1^{-1}:=\sup_{y\geq 1} \Big\{ j_2 \, g_1(-1,y)-y \Big\}+1=\rho_1^1
\end{equation}
and the supremum achieved for {\color{black} $y_1^{ -1}:=\frac{(1-\rho_1)^2}{1-2\rho_1}
$.
}
In particular, \textcolor{black}{since  the divergence  \eqref{conclusion_lemma_good}} of $K_1$ 
  implies $\lim_{\varepsilon\to 0}j_2(\varepsilon)=r/4$, we also have
\begin{equation}
\label{rho1_dilute}
\lim_{\varepsilon\to 0}\rho_1^\sigma (\varepsilon)=\frac{1}{2}\left(1-\sqrt{1-r}\right)=\rho_c(0)
\end{equation}
that is the lower solution of (\ref{eq:j2}) with $r/4$ instead of $j_2$.
This says that  the approximation   after one step of  renormalization is close to the dilute limit.
{\color{black} Lemma \ref{lemma_induction}, stated below, shows that 
 $\rho_n(\varepsilon)$ remains close to $\rho_c(0)$ for $\varepsilon$ small.}
By (\ref{rho1_dilute}), \eqref{induction_rho} and Lemma \ref{lemma_induction} below, we have
\begin{equation}
\label{wehave}
\limsup_{\varepsilon\to 0}\limsup_{n\to+\infty}\rho_n(\varepsilon)
\leq\rho_c(0).
\end{equation}
{\color{black} This completes the proof of Proposition \ref{prop: estimation suites}.}
\begin{lemma}
\label{lemma_induction} 
Assume $\hh$ with $b\in[1,2)$ and
{\color{black}  $K_1$ satisfies \eqref{conclusion_lemma_good}.}
With the notation of Lemma \ref{lemma_good}, we fix 
$\gamma < \min \big\{ \gamma_0, \frac{2}{b} -1 \big\}$.
Then 
%
\begin{enumerate}
\item[(1)] $\lim_{\varepsilon\to 0}\prod_{n=1}^{+\infty} a_n(\varepsilon)=1$,
\item[(2)] $\lim_{\varepsilon\to 0}\sum_{n=1}^{+\infty}\Delta_n(\varepsilon)=0$.
\end{enumerate}
\end{lemma}
%
%
\begin{proof}
\ \\
\noindent
{\it Proof of (1).} We have to show that
\be\label{limit_1}
\lim_{\varepsilon\to 0}\sum_{n=1}^{+\infty} \log \left(1-\frac{1}{l_n(\varepsilon)}\right)=0 .
\ee
Since 
$$
l_n(\varepsilon)=\exp\left[
\log \big(K_1(\varepsilon) \big) \, \gamma(1+\gamma)^{n-1}
\right]\geq \exp\left[
\gamma(1+\gamma)^{n-1}
\right],
$$
we have, for $n\geq 2$,  $l_n(\varepsilon)^{-1}\leq C(\gamma):=e^{-\gamma(1+\gamma)}<1$. Hence, for $n\geq 2$,
$$
0\leq -
\log \left(1-\frac{1}{l_n(\varepsilon)}\right)\leq \frac{1}{l_n(\varepsilon)}+\frac{C'(\gamma)}{l_n(\varepsilon)^2}
\leq (1+C'(\gamma))\exp\left[
-\gamma(1+\gamma)^{n-1}
\right].
$$
The limit \eqref{limit_1} then follows from dominated convergence, and $\lim_{\varepsilon\to 0}K_1(\varepsilon)=+\infty$, which implies $\lim_{ \varepsilon\to 0}l_n(\varepsilon)=\textcolor{black}{+\infty}$ for any $n\geq 1$.\\ \\
{\it Proof of (2).} Here we  can write  $\Delta_n \textcolor{black}{\stackrel{n\to+\infty}{\sim }}\psi_0(K_n)$, where
\begin{eqnarray}
& {\psi}_0(K)  := \frac{(\log K)^3}{K}\left[
\left(\frac{r}{4}+
\frac{a}{K^{\frac{b}{2} (1+\gamma)^2}}
\right)^{-1}
- \left(\frac{r}{4}+
\frac{a}{K^{\frac{b}{2}  (1+\gamma)}}
\right)^{-1}
\right]^{-1} \nonumber \\
& \qquad \qquad \qquad \qquad  \times  
\left(
\log\left\{
\frac{(\log K)^{3/2}}{\sqrt{K}}
\left[
\left(\frac{r}{4}+
\frac{a}{K^{\frac{b}{2} (1+\gamma)^2}}
\right)^{-1}
- \left(\frac{r}{4}+
\frac{a}{K^{\frac{b}{2}  (1+\gamma)}}
\right)^{-1}
\right]^{-1}
\right\}
\right)^3
\label{eq: b<2}\\
& \qquad \qquad 
\stackrel{K\to+\infty}{\sim} 
C'' (\log K)^{6} K^{\frac{b}{2}(1+\gamma)-1}.
\nonumber
\end{eqnarray}
%
%
%
for some constant $C''>0$. The assumption on $b$ and the choice of $\gamma$ imply that
$c:=1-\frac{b}{2}(1+\gamma)>0$
(Equation \eqref{eq: b<2} is the main reason for restricting to the case $b<2$). 
By \eqref{conclusion_lemma_good}, there exists $\varepsilon_1>0$ such that $K_n(\varepsilon)\geq 2$ for  every $n\geq 1$ and $\varepsilon\in[0,\varepsilon_1]$.
Thus, by \eqref{eq: b<2}, there exists a constant $D>0$ such that, for such $n$ and $\varepsilon$,
$$
\Delta_n(\varepsilon)\leq \frac{D}{K_n(\varepsilon)^{c}}\leq\frac{D}{K_1(\varepsilon)^{c(1+\gamma)^n}}\leq \frac{D}{2^{c(1+\gamma)^n}} .
$$
Since $\lim_{\varepsilon\to 0}K_1(\varepsilon)=+\infty$, the result follows again from dominated convergence.
\end{proof}

\subsection{Proof of Proposition \ref{prop:passage_dilute}}

Note that 
\begin{equation}
\label{dilute_product}
\lim_{\varepsilon\to 0}
\prod_{n=1}^{+\infty}\frac{l_n(\varepsilon)}{l_n(\varepsilon)-1}=1.
\end{equation}
{\color{black} 
Thus for any \textcolor{black}{$\upsilon>0$}, there exists $\varepsilon_*>0$ such that, for $\varepsilon \leq \varepsilon_*$, the following holds
\begin{equation}
\label{eq:epstar}
\prod_{n=1}^{+\infty}\frac{l_n(\varepsilon)}{l_n(\varepsilon)-1}<1+\textcolor{black}{\upsilon} ,
\end{equation}
and $(y_n)_{n\geq 0}$ is an increasing sequence thanks to Lemma \ref{lemma_yn}.
We fix $y < \frac{y_1}{1+\textcolor{black}{\upsilon}}$ and $\varepsilon \leq \varepsilon_*$.
For any $N \in \mathbb{N}^*$, we define the sequence 
\begin{equation}
\label{finiteprod_dilute_k}
y_{N,N}:=y 
\quad \text{and} \quad 
\forall n \in \{ 1, N-1 \}, \qquad y_{n,N}:=
\prod_{k= n + 1}^N \frac{l_k}{l_k-1} y  \leq y_1 \leq y_n.
\end{equation}
As $\frac{l_n}{l_n-1}y_{n,N}=y_{n-1,N}\leq y_n$, then 
 $g_{n+1}(1,y_{n,N})$ is determined by \eqref{gn_2} so that 
\begin{eqnarray}
\label{recursion}
g_{n+1}(1,y_{n,N}) = \left(
1-\frac{1}{l_n}
\right)g_n\left(
1,
y_{{n-1,N}}
\right)+\frac{1}{l_n}{j_{n+1}}+\delta_n\varphi\left(y_{{{n,N}}}\right).
\end{eqnarray}
Starting from $y_{N,N}:=y $ and proceeding recursively, we deduce that 
\begin{align} 
g_{N+1}(1,y) \leq  \prod_{n=1}^N\left(1-\frac{1}{l_n}\right)g_{\textcolor{black}{1}}\left[
%
%
1,\prod_{n=1}^N\frac{l_n}{l_n-1}y
\right]
+\frac{4}{r}\sum_{n=1}^N 
\frac{1}{l_n}
+\sum_{n=1}^N\delta_n\varphi \left(   
\prod_{r =n+1}^N\frac{l_r}{l_r-1}y 
\right) .
\label{eventually}
\end{align}
From \eqref{rho1_dilute}, we know that $y_1 = y_1 (\varepsilon)$ converges to $y_1(0)
=\frac{\rho_c(0)^2}{1-2\rho_c(0)}$. Furthermore  $\lim_{\varepsilon\to 0}l_n(\varepsilon)=+\infty$ and $\lim_{\varepsilon\to 0}\delta_n(\varepsilon)=0$. Thus  it follows from (\ref{eventually}) that 
\begin{equation}
\label{lim_gN}
\forall y< \frac{y_1(0)}{1+\textcolor{black}{\upsilon}}, \qquad 
\limsup_{\varepsilon\to 0} \limsup_{N\to+\infty}g_N(1,y){\leq}g_1(y) .
\end{equation}
In the dilute limit,  $\textcolor{black}{\upsilon}$ can be  arbitrarily small so that the inequality above holds more generally for $y<y_1 (0)$.
A similar result holds for $\sigma = -1$.}

\section{Fluctuation bounds : Proof of Proposition \ref{prop:fluct}} 
\label{sec: prop passage time}

{\color{black}
Proposition \ref{prop:fluct} is proved in this section. 
Preliminary estimates are stated in Subsection \ref{subsec:conc} and then applied in Subsection \ref{fluct_ent}, which is the body of the proof.
}

\subsection{Concentration estimates}
\label{subsec:conc}
We shall need  a classical gaussian concentration inequality for last passage times. In the following lemma, it is assumed that the service times $Y_{i,j}$  involved in the definition \eqref{def_passage}--\eqref{def_last} of last passage times
are i.i.d. random variables  {\it bounded} by $M$ instead of being exponentially distributed. To avoid confusion with the previous notation, the corresponding probability $\Prob_M$ and expectation $\Exp_M$ are denoted below by an index $M$.
\begin{lemma}\cite[Lemma 3.1]{mar}
\label{lemma_concentration}
Assume that $Y=(Y_{i,j}:\,(i,j)\in\Z\times\N)$ is a vector of non negative independent random variables bounded from above by $rM$.
Let $(x_1,y_1)$ and $(x_2,y_2)$ in $\Z\times\N$ be such that $(x_2-x_1,y_2-y_1) \in \mathcal W$. Then 
\begin{eqnarray}\nonumber
T^\alpha \big( (x_1,y_1),(x_2,y_2) \big)  =  \Exp_M\left[
T^\alpha \big( (x_1,y_1),(x_2,y_2) \big)
\right]
+  8M\sqrt{L((x_1,y_1),(x_2,y_2))}Z,
\label{concentration_inequality}
%
\end{eqnarray}
where $L((x_1,y_1),(x_2,y_2)):=(x_2-x_1)+2(y_2-y_1)$ is the length of any path connecting $(x_1,y_1)$ to $(x_2,y_2)$, and $Z$ is a random variable with subgaussian tail
$$
\forall t \geq 0, \qquad \Prob_M(|Z|\geq t)\leq \exp(-t^2) \, .
$$
\end{lemma}
We stress the fact that Gaussian bounds on last passage times are by no means optimal in the case of exponential service times, for which more refined (but also more specific)  gaussian-exponential estimates are available (see e.g. \cite{tal}). However, for our purpose, they have the advantage of being both simple and sufficient,
while also extending to service distributions with heavier tails, as a result of the cutoff procedure introduced in Subsection \ref{fluct_ent}. 

\medskip

The above concentration inequality will be combined with the following result,  established in Appendix \ref{app:proof_cor_max}.
\begin{lemma}\label{cor_max}
Let $\mathcal A$ and $\mathcal I$ be finite sets. Assume that for each $a\in\mathcal A$, we have a 
family 
$({\mathcal Y}_{a,i})_{i\in \mathcal I}$ of independent random variables such that, for every $i\in\mathcal I$,
\be
\label{hyp_normal}
{\mathcal Y}_{a,i}  \, = \,
\Exp \big( {\mathcal Y}_{a,i} \big) + \sqrt{V_{a,i}} Z_{a,i},
\ee
where $V_{a,i}>0$, and $Z_{a,i}$ is a random variable such that 
\be
\label{normal_tail}
\Prob(Z_{a,i}\geq t)\leq e^{-t^2},
\ee
for every $t\geq 0$. Then 
\begin{eqnarray}
\Exp \left( \max_{a\in\mathcal A}\sum_{i\in\mathcal I} {\mathcal Y}_{a,i} \right) 
& \leq & 
\max_{a\in\mathcal A}\sum_{i\in\mathcal I} 
\Exp \big( {\mathcal Y}_{a,i}\big)  \nonumber\\
& &  +
\left(
\max_{a\in\mathcal A}\sum_{i\in\mathcal I}V_{a,i}
\right)^{\frac{1}{2}}
\left(
\sqrt{\pi}\sqrt{|\mathcal I|}+\sqrt{\pi}\sqrt{A}+\sqrt{A}\sqrt{\log|\mathcal A|}
\right),
\label{exp_max}
\end{eqnarray}
where $|.|$ denotes  the cardinality, and $A$ is \textcolor{black}{a universal constant.}
\end{lemma}

\subsection{Path renormalization: fluctuation and entropy}
\label{fluct_ent}
We now proceed in three steps. 
In step one, we define a cutoff procedure for the service times $Y_{i,j}$, by conditioning on their maximum,  in order to replace them with bounded variables, to which the  results of Subsection \ref{subsec:conc} apply.
In step two, we apply Lemma \ref{cor_max} to passage times in subblocks. This yields for the cutoff service times a result similar to the statement of proposition \ref{prop:fluct},
but without the whole logarithmic correction.
Finally, in step three, we remove the cutoff and use a bound on the expectation of the maximum of exponential variables, to obtain a quasi-gaussian estimate with a logarithmic correction.\\ \\ 
{\bf Step 1.} {\it Notation and conditional measure.}
Pick $\gamma$ such that 
\be\label{choose_gamma}
0<\gamma<\min \big\{ \gamma_0,(2/b)-1 \big\} ,
\ee
with $\gamma_0$ introduced in Lemma \ref{lemma_good}, and $b$ in  \eqref{condition_h_2} and
\eqref{condition_h}.
Let $B=\Z\cap[0,\sigma(K_{n+1}-1)]$ be a block of order $n+1$
and partition $B$ into  subblocks of  level $n$ denoted by
$B_l=[\sigma (l-1)K_n,\sigma (lK_n-1)]\cap\Z$, where $l=1,\ldots,l_n$.

\medskip

Set $y'= \lfloor K_{n+1}y\rfloor$, 
$\tilde{\Gamma}_n=\tilde{\Gamma}_n((0,0),(\sigma(K_{n+1}-1),y'))$
and define 
%
$$
%
M_B(y):=
\max \big \{ Y_{i,j}:\,i\in B,\,j=0,\ldots,y'= \lfloor K_{n+1}y\rfloor  \big  \} .
$$
Given $M>0$, denote by $\Prob_{B,M,y'}$ the distribution of $(X_{i,j}:\,\,i\in B,j=0,\ldots,y')$, where $X_{i,j}$ are i.i.d. and have the same distribution as  $Y_{i,j}$ conditioned on $Y_{i,j}\leq rM$.
(Note that after conditioning by $rM$, the percolation paths have weights $\frac{Y_{i,j}}{\alpha} \leq M$ as $\alpha \geq r$).
Denote by $\Prob'_{B,M,y'}$ the distribution of $(Y_{i,j}:\,i\in B,\,j=0,\ldots,y')$ conditioned on $M_B(y)=rM$. 
%
\textcolor{black}{
The reason why we introduce the two different distributions $\Prob_{B,M,y'}$ and $\Prob'_{B,M,y'}$
is that we can apply Lemma \ref{lemma_concentration} to the former, while the latter is obtained by conditioning the actual joint law of the r.v.'s $Y_{i,j}$ on their maximum. A useful relation between these two distributions is the following.
%
\begin{lemma}\label{lemma_useful}
The distribution of $T^\alpha_B((0,0),(\sigma(K_{n+1}-1),y'))$ under $\Prob'_{B,M,y'}$ is stochastically dominated by the distribution of $T^\alpha_B((0,0),(\sigma(K_{n+1}-1),y'))+M$ under $\Prob_{B,M,y'}$.
\end{lemma}
\begin{proof}
Let $X=(X_{i,j}:\,\,i\in B,j=0,\ldots,y')$ be a family of i.i.d. random variables whose distribution is the distribution of  $Y_{i,j}$ conditioned on $Y_{i,j}\leq rM$. Pick a uniformly distributed $(i_0,j_0)$ in $B\times\{0,\ldots,y'\}$; then give value $rM$ to $Y'_{i_0,j_0}$, and let the other $Y'_{i,j}$ for $(i,j)\neq (i_0,j_0)$ be independent with the same distribution as the above $X_{i,j}$. Then the family $Y=(Y'_{i,j}:\,(i,j)\in B\times\{0,\ldots,y'\})$ has distribution $\Prob'_{B,M,y'}$. It suffices to now to show that
\be\label{suffices}
T^\alpha_B((0,0),(\sigma(K_{n+1}-1),y'))[Y']\leq T^\alpha_B((0,0),(\sigma(K_{n+1}-1),y'))[X]+M,
\ee
where the notation $T^\alpha_B((0,0),(\sigma(K_{n+1}-1),y'))[X]$ denotes the passage time as a function of $X$, that is \eqref{def_passage}--\eqref{def_last} with the r.v.'s $Y_{i,j}$ replaced by $X_{i,j}$.
Let $\gamma^*=(x_k,y_k)_{k=0,\ldots,n}$ denote the optimal path that achieves $T^\alpha_B((0,0),(\sigma(K_{n+1}-1),y'))[Y']$, i.e. such that
\be
\label{optimal_path}
T^\alpha_B((0,0),(\sigma(K_{n+1}-1),y'))[Y']=T^\alpha_B(\gamma^*)[Y']:=\sum_{k=0}^n 
\frac{Y'_{x_k,y_k}}{\alpha(x_k)}.
\ee
Then
\be\label{compare_times_2}
T^\alpha_B(\gamma^*)[Y']\leq T^\alpha_B(\gamma^*)[X]+M\leq T^\alpha_B((0,0),(\sigma(K_{n+1}-1),y'))[X]+M.
\ee
%
%
Indeed, if $(i_0,j_0)$ does not lie on $\gamma^*$, the first inequality in \eqref{compare_times_2} is an equality; otherwise, the inequality holds because 
$$\frac{Y'_{i_0,j_0}}{\alpha(i_0)}\leq M\leq M+\frac{X_{i_0,j_0}}{\alpha(i_0)}.$$
This establishes \eqref{suffices}.
\end{proof}
}
%
%
%

\bigskip

\noindent
{\bf Step 2.} {\it Fluctuation and entropy bounds.}  
Given $\alpha$, the random variables $\{ U^\alpha_{B,l'}(\sigma,\tilde{\gamma}),$ $V^\alpha_{B,l}(\sigma,\tilde{\gamma}) \}_{l,l'}$ (defined in \eqref{eq: def U V}) are independent under $\Prob_{B,M,y'}$, because they depend on disjoint subvectors of $Y$ (see \textcolor{black}{F}igure \ref{fig: independence}). 
On the other hand, by Lemma \ref{lemma_concentration}, we get
\begin{eqnarray}
\label{concentration_h}
\begin{cases}
U^\alpha_{B,l}(\sigma,\tilde{\gamma}) 
=
\Exp_M\left[
U^\alpha_{B,l}(\sigma,\tilde{\gamma})
\right]+8M\sqrt{\sigma K_n+2\tilde{y}_l} \, Z_l^{(1)},
\\
V^\alpha_{B,l}(\sigma,\tilde{\gamma}) 
=
\Exp_M\left[ V^\alpha_{B,l}(\sigma,\tilde{\gamma})\right]+8M\sqrt{2\tilde{z}_l} \, Z_l^{(2)},
\end{cases}
\end{eqnarray}
where $(Z_l^{(i)})_{l=1,\ldots,l_n;i=1,2}$ is a family of r.v.'s independent under $\Prob_{B,M,y'}$ and such that
\begin{equation}
\label{eq: one sided}
\Prob_{B,M,y'}\left( Z_l^{(i)}\geq t \right)\leq \exp( -t^2) ,
\end{equation}
for all $t\geq 0$. To apply Lemma \ref{cor_max} to the random variables in \eqref{concentration_h},
 we take 
$\mathcal A=\tilde{\Gamma}_n \big((0,0),(\sigma(K_{n+1}-1),y') \big)$ with $\mathcal I=\{1,\ldots,2l_n\}$,
and for $a=\tilde{\gamma}\in \mathcal A$, we set 
$$
l\in\{1,\ldots,l_n\},  \qquad {\mathcal Y}_{a,2l-1}=U^\alpha_{B,l}(\sigma,\tilde{\gamma})
\quad \text{and} \quad 
{\mathcal Y}_{a,2l}=V^\alpha_{B,l}(\sigma,\tilde{\gamma}).
$$ 
Thus in \eqref{exp_max} we have $|\mathcal I|=2l_n=2K_{n+1}/K_n$, and (cf. \eqref{concentration_h} and \eqref{total_height})
$$
\sum_{i\in\mathcal I}V_{a,i}=64M^2 (\sigma K_{n+1}+2y')\leq 64M^2 K_{n+1}(\sigma+2y).
$$
To estimate the cardinality $|\mathcal A|$ of the skeletons, we need the following 
\begin{lemma}
\label{count_skeleton}
For every $y'\in\N$, one has
$$
\log | \mathcal A| = 
\log \big |  \tilde{\Gamma}_n  \big( (0,0),(\pm(K_{n+1}-1),y')  \big)  \big |
\leq
2\frac{K_{n+1}}{K_n} \, \left[ 1+\log \left( 1+{K_n y} \right) \right].
$$
\end{lemma}
\begin{proof}
The number of such skeletons satisfies the inequality
\be
\label{skeleton_number}
\sigma\in\{-1,1\}, \qquad 
\Big| \tilde{\Gamma}_n \big( (0,0),(\sigma(K_{n+1}-1),y')  \big) \Big|
\leq {2l_n+y'-1 \choose 2l_n-1}.
\ee
The previous upper  bound follows by noticing that choosing a skeleton amounts to choosing
$2l_n-1$ heights corresponding to the different renewal times to reach the total height $y'$. In fact, 
when $\sigma=1$, some of these heights can be equal if $\tilde y_i =0$ or $\tilde z_i =0$ for some $i 
\leq 2l_n-1$. Thus, the number of ways \textcolor{black}{for} choosing the heights is bounded by the number of ways \textcolor{black}{for} choosing $2l_n-1$
items from a set of $2l_n+y'-1$ items.
Estimate \eqref{skeleton_number} is actually an equality if $\sigma=1$.
Recall the inequality
\be
\label{binomial_bound}
\log \left( \ba{l}
N\\
k
\ea
\right)\leq N h\left(
\frac{k}{N}
\right),
\ee
where $h$ is defined on $[0,1]$ by
\be\label{def_h}
-h(x):=x \log x+(1-x) \log (1-x)
{\color{black}
\quad \text{with} \quad h(0) = h(1) = 0.}
\ee
\textcolor{black}{
Bound \eqref{binomial_bound} follows from Cramer's exact large deviation uppper bound. For completeness,
we give a derivation of \eqref{binomial_bound} at the end of this proof.\\ \\
}
Furthermore,
$$
u h(1/u) \leq 1+ \log u
$$
for $u\geq 1$, and 
$$
\frac{2l_n+y'-1}{2l_n-1}\leq 1+\frac{y'}{l_n}
= 1 + \frac{K_{n+1}}{l_n} y,
$$
(the inequality follows from $l_n\geq 1$). This completes the  proof of Lemma \ref{count_skeleton}.\\ \\
\textcolor{black}{
{\it Proof of \eqref{binomial_bound}.}
The bound follows from
the inequality
\be\label{pre_cramer}
\left(
\ba{l}
N\\
k
\ea
\right)
 \leq  2^N\Prob\left( \sum_{i=1}^N \zeta_i \geq k
\right),
\ee
where $(\zeta_i)_{i=1,\ldots,n}$ is a sum of i.i.d. Bernoulli variables with parameter $1/2$,
and Cramer's exact large deviation upper bound applied to the r.h.s. of \eqref{pre_cramer}. For completeness
we recall the derivation of Cramer's bound:
for $k\geq N/2$ and $\theta\in\R$, by Markov's inequality,
\be\label{cramer}
\Prob\left( \sum_{i=1}^N \zeta_i \geq k
\right)
\leq  e^{-N\theta\frac{k}{N}}
\Exp\left( 
e^{\theta\sum_{i=1}^N \zeta_i}
\right)=
e^{
-N\left(
\theta\frac{k}{N}-L(\theta)
\right)
},
\ee
where
$$
L(\theta):=\log \Exp e^{\theta\zeta_1}=\log \frac{1+e^\theta}{2}.
$$
The Legendre transform of $L$ is $h+\log 2$, where $h$ is given by \eqref{def_h}.
Optimizing the upper bound in \eqref{cramer} over $\theta$ yields
$$
2^N\Prob\left( \sum_{i=1}^N \zeta_i \geq k
\right)\leq 2^N\exp\left\{
-N\sup_{\theta\in\R}\left[
\theta\frac{k}{N}-L(\theta)
\right]\right\}
=e^{-Nh\left(\frac{k}{N}\right)}.
$$
}
\end{proof}
\medskip
Combining \eqref{exp_max} with the entropy estimate of Lemma \ref{count_skeleton}, we obtain 
\begin{align}
\label{eq:fluct:3}
&\Exp_M \left( \max_{\tilde{\gamma}\in\tilde{\Gamma}_n((0,0),(\sigma(K_{n+1}-1),[K_{n+1}y]))}
\Big\{
U^\alpha_{B}(\sigma,\tilde{\gamma})
+V^\alpha_{B}(\sigma,\tilde{\gamma})
\Big\} \right)\\
& \quad \leq  \max_{\tilde{\gamma}\in\tilde{\Gamma}_n((0,0),(\sigma(K_{n+1}-1),[K_{n+1}y]))}
\Big\{
\Exp_M\left(
U^\alpha_{B}(\sigma,\tilde{\gamma})
+V^\alpha_{B}(\sigma,\tilde{\gamma})
\right)
\Big\}
\nonumber\\
%
%
%
& \quad + 8M\sqrt{K_{n+1}}\sqrt{\sigma+2y}
\sqrt{ 2\frac{K_{n+1}}{K_n} }
\left(
\sqrt{\pi}
+\sqrt{A}+
\sqrt{A}
\sqrt{ \left[ 1+\ln\left( 1+{K_n y} \right) \right] }
\right)\nonumber,
%
%
%
\end{align}
where we used that $2\frac{K_{n+1}}{K_n}  \geq \pi$.

\medskip
\noindent
{\bf Step 3.} {\it  Removing the cut-off on $Y$.}
The random variables  $\{ U^\alpha_{B,l}(\sigma,\tilde{\gamma}), \; V^\alpha_{B,l}(\sigma,\tilde{\gamma}) \}_{l,l'}$ are nondecreasing functions of $Y=(Y_{i,j}:\,(i,j)\in\Z\times\N)$ with respect to the product order.
Therefore, their distributions under $\Prob_{B,M,y'}$ are stochastically dominated by their distributions under $\Prob$ and one has 
%
\be
\label{eq: comparaison E et EM}
\Exp \left[U^\alpha_{B,l}(\sigma,\tilde{\gamma}) \right]
\geq 
\Exp_M \left[U^\alpha_{B,l}(\sigma,\tilde{\gamma}) \right],
\qquad 
\Exp \left[V^\alpha_{B,l}(\sigma,\tilde{\gamma})\right]
\geq 
\Exp_M \left[V^\alpha_{B,l}(\sigma,\tilde{\gamma})\right]  .
\ee
On the other hand, 
%
\textcolor{black}{Lemma \ref{lemma_useful}} combined with \eqref{eq: comparaison E et EM} and \eqref{eq:fluct:3} yields 
\begin{align}
\label{eq:fluct:3bis}
& \Exp'_M \left( \max_{\tilde{\gamma}\in\tilde{\Gamma}_n((0,0),(\sigma(K_{n+1}-1),[K_{n+1}y]))}
\Big\{ U^\alpha_{B}(\sigma,\tilde{\gamma}) +V^\alpha_{B}(\sigma,\tilde{\gamma}) \Big\}
\right)\\
& \qquad \leq  \max_{\tilde{\gamma}\in\tilde{\Gamma}_n((0,0),(\sigma(K_{n+1}-1),[K_{n+1}y]))}
\Big\{
\Exp\left[
U^\alpha_{B}(\sigma,\tilde{\gamma})
+V^\alpha_{B}(\sigma,\tilde{\gamma})
\right]  \Big\}+M\nonumber\\
%
%
%
& \qquad +  8M K_{n+1}\sqrt{\sigma+2y} \sqrt{ \frac{2}{K_n} } 
\left(
\sqrt{\pi}
+\sqrt{A}+
\sqrt{A}
\sqrt{
 \left[
1+\log \left(
1+{K_n y}
\right)
\right]
}
\right).\nonumber
\end{align}
%
Recall that $\Exp'_M$ on the left-hand site of \eqref{eq:fluct:3bis} stands for the expectation with respect to $\Prob$ conditioned on the maximum $M_B(y)=M$.
We can now remove this conditioning by integrating both sides of \eqref{eq:fluct:3bis} with respect to the law of $M_B(y)$. We first write
\be\label{def_expmax}
\Exp \big[ M_B(y) \big]=m\left([yK_{n+1}]K_{n+1}\right),
\ee
where the function $t\in[0,+\infty)\mapsto m(t)$ is defined as the expectation of the maximum of $1+[t]$ i.i.d. exponential variables of rate 1. In particular, we have
\be\label{bound_expmax}
m(t)\leq C[1+\log(1+t)],
\ee
for some constant $C>0$.
Thus, after conditioning on $M_B(y)$, we obtain
%
%
%
%
\begin{align}
\label{after_conditioning}
& \frac{1}{K_{n+1}}
\Exp \left( \max_{\tilde{\gamma}\in\tilde{\Gamma}_n((0,0),(\sigma(K_{n+1}-1), [ K_{n+1}y ]))}
\Big\{
U^\alpha_{B}(\sigma,\tilde{\gamma})
+V^\alpha_{B}(\sigma,\tilde{\gamma})
\Big\} \right) \\
& \quad \leq  \frac{1}{K_{n+1}}
\max_{\tilde{\gamma}\in\tilde{\Gamma}_n((0,0),(\sigma(K_{n+1}-1),[K_{n+1}y]))} 
\Big\{ 
\Exp\left( 
U^\alpha_{B}(\sigma,\tilde{\gamma})
+V^\alpha_{B}(\sigma,\tilde{\gamma})
\right) \Big\}
\nonumber\\
%
%
%
& \qquad \qquad 
+m\left([yK_{n+1}]K_{n+1}\right)\Delta_n(y) , \nonumber
\end{align}
where 
\begin{eqnarray*}
\tilde \Delta_n(y) := \frac{1}{K_{n+1}}
+
8 \frac{ \sqrt{\sigma+2y}}{\sqrt{K_{n}}} 
\left(
\sqrt{A}\sqrt{\pi}
+ \sqrt{2\pi}
+  \sqrt{A} \sqrt{ 1+\log \left( 1+{K_n y} \right) }
\right) .
\end{eqnarray*}
A simple computation shows that 
$$
m \big([yK_{n+1}]K_{n+1}\big) \; \tilde \Delta_n(y)
\leq
\delta_n\sqrt{\sigma/2+y} \, \big[1+\log (1+y) \big]^{3/2} ,
$$
%
with $\delta_n$ given by \eqref{eq: delta n}.
Using the notation of \eqref{fluct_part}, we get 
$$
\mathcal F_n (y) \leq
\delta_n\sqrt{\frac{\sigma}{2} +y} \;  [1+\log (1+y)]^{3/2} .
$$ 
This completes the proof of Proposition \ref{prop:fluct}.
\qed

\section{ Completion of proofs of Theorems \ref{th_plateau} and \ref{th_dilute_limit}}
\label{sec:completion}

In this section, we complete the remaining parts in the proof of Theorem \ref{th_plateau}.
In Subsection \ref{unrestricted}, we deduce from Proposition \ref{prop_block} a similar statement for unrestricted 
passage times (that is, when the paths are not restricted to the box defined by the endpoints). Finally, 
Theorem \ref{th_plateau} is completed, in Subsection \ref{conclude}, using the fact that most boxes are good.
The dilute limit (Theorem \ref{th_dilute_limit}) is studied in Section \ref{sec:proof_dilute}.

\subsection{Bounds on unrestricted passage times}
\label{unrestricted}

To obtain Theorem \ref{th_plateau} from Proposition \ref{prop_block}, we first
deduce from Proposition \ref{prop_block} the following result for unrestricted passage times, i.e. passage times obtained by maximizing over paths not bound to stay
in the interval between the two endpoints (see \textcolor{black}{F}igure \ref{fig: unrestricted}).
\begin{figure}[htpb]
\centering
\includegraphics[width=.7\columnwidth]{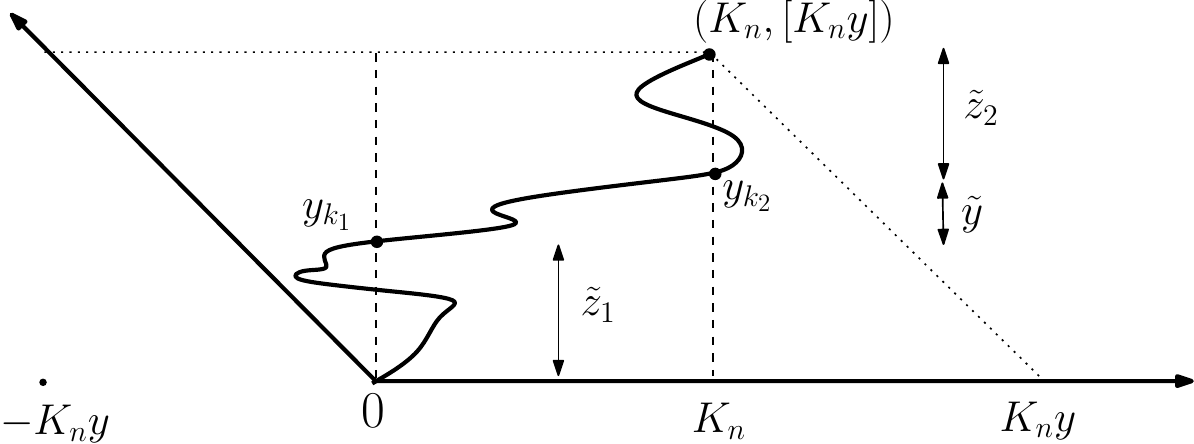}
\caption{\small The optimal path is not restricted to the box $[0,K_n] \times [0,K_n y]$ but can wander around the whole parallelogram marked by the dotted line. The optimal path is split into 3 parts 
$(0,y_{k_1})$, $(y_{k_1},y_{k_2})$ and $(y_{k_2}, (K_n,  \lfloor K_n y\rfloor))$.
 } 
\label{fig: unrestricted} 
\end{figure}

\margincom{\textcolor{black}{c'est un peu genant d'avoir $\sigma$ en indice par intermittence ?}\textcolor{black}{c'etait $\rho_n$ en fait, j'ai corrige la suite}}
Given the sequence {$(\rho_n)_{n \geq 1}$ of Proposition \ref{prop: estimation suites}, we set 
\begin{equation}
\label{eq: rhoc}
\overline{\rho}_c=\overline{\rho}_c(\varepsilon) : = \limsup_{n \to \infty} \rho_n\in[0,1/2).
\end{equation}
{\color{black} In the rest of this section, the asymptotics $\lim_{n\to+\infty}$ means that 
we restrict to a subsequence (fixed once and for all) of $(\rho_n)_{n\geq 1}$ that achieves the $\limsup$ in \eqref{eq: rhoc}.}
\begin{corollary}
\label{cor_block}
For $\sigma = \pm 1$ and $y\geq \sigma^-$, we consider the unrestricted passage time
\be\label{def_unrestricted}
{\tau}_n^\alpha(\sigma,y):=\Exp \left[ \frac{1}{K_n}T^\alpha(\sigma K_n,[K_ny]) \right].
\ee
Then there are functions $e_n(\sigma,y)$  such that, for all $n\in\N^*$ and environments $\alpha$ for which $[0,\sigma(K_n-1)]$ is a good block, the following bound holds:
\be
\label{limit_time}
{\tau}_n^\alpha(\sigma,y)\leq\tau^{\overline{\rho}_c,r/4}(\sigma,y)+e_n(\sigma,y).
\ee
{\color{black} Furthermore $e_n(\sigma,.)$ does not depend on $\alpha$ and
converges locally uniformly to 0 on $[\sigma^-,+\infty)$ as $n\to+\infty$.}
\end{corollary}

\begin{proof}[Proof of Corollary \ref{cor_block}]
The unrestricted passage time $T^\alpha(\sigma K_n,\lfloor K_n y\rfloor)$ 
may use paths that do not stay in $B := [0,\sigma K_n]$. 
To control the contribution outside $B$, we use a decomposition of the path in the same spirit as Section \ref{subsec:outline}. The problem here is simpler because there is no more
renormalization, and there are only three regions to consider for the path according to its $x$-coordinate (recall the simplifying notational convention $[a,b]=[b,a]$), 
namely the interval 
$[0,\sigma (K_n-1)]$ and the two intervals on either side of it, which are also bounded by the fact that the only  possible increments are $(1,0)$ and $(-1,1)$.
If $\sigma=1$, these intervals are $[-\lfloor K_n y\rfloor ,-1]$ and $[K_n, K_n + \lfloor K_n y\rfloor ]$. 
If $\sigma=-1$ then $y \geq 1$ and these intervals are 
$[-  \lfloor K_n y\rfloor  ,-K_n]$ and $[1,-K_n+ \lfloor K_n y\rfloor ]$.
We thus define a simpler path skeleton 
$(\tilde{z}_1,\tilde{y},\tilde{z}_2)=\tilde{\gamma}$ as described below. 

\medskip

Let $\gamma=(x_k,y_k)_{k=0,\ldots,m-1}$ be a path connecting
$(0,0)=(x_0,y_0)$ to $(x_{m-1},y_{m-1})=(\sigma K_n,[K_ny]=y')$. We set
\begin{eqnarray*}
& k_1 :=  1+
\max\{k=0,\ldots,m-1:\,\sigma x_k<0\},\\
& k_2 :=  \min\{k= k_1, \ldots,  m-1:\,x_k=\sigma K_n\},
\end{eqnarray*}
with the convention that the $\max$ is $-1$ if the corresponding set is empty. Since allowed  path increments are 
$(1,0)$ and $(-1,1)$, we have $x_{k_1}=0$. We then define (see \textcolor{black}{F}igure \ref{fig: unrestricted})
\begin{eqnarray*}
\tilde{z}_1 :=  y_{k_1}, \qquad
\tilde{y}  :=  y_{k_2}-y_{k_1}, \qquad 
\tilde{z}_2 :=  y'-y_{k_2}.
\end{eqnarray*}
Let $\tilde{\Gamma}_n$  denote the set of these new ``skeletons'', that is the set of
triples $\tilde{\gamma}=(\tilde{z}_1,\tilde{y},\tilde{z}_2)$ such that 
\be\label{def_new_skeleton}
\tilde{z}_1+\tilde{y}+\tilde{z}_2=\lfloor K_n y\rfloor =:y',\quad
(\tilde{z}_1,\tilde{y},\tilde{z}_2)\in\N\times(\N\cap[\sigma^-K_n,+\infty)) \times \N.
\ee
The path between $k_1$ and $k_2$ corresponds to the restricted part which has already been studied in the previous sections. 
As in \eqref{double_max}, we write
\begin{eqnarray}
T^\alpha \big( (0,0), (K_n,\lfloor K_n y\rfloor ) \big)  = 
\max_{(\tilde{z}_1,\tilde{y},\tilde{z}_2)\in\tilde{\Gamma}_n }\left[
V^\alpha_1 (\sigma,\tilde{\gamma})+U^\alpha_B(\sigma,\tilde{\gamma})+V^\alpha_2(\sigma,\tilde{\gamma})
\right],
\label{cut_in_three}
\end{eqnarray}
where $B:=[0, \sigma K_n-1]$, and
\begin{eqnarray*}
V^\alpha_1(\sigma,\tilde{\gamma}) & := & T^\alpha((0,0),(0,\tilde{z}_1)),\\
U^\alpha_B(\sigma,\tilde{\gamma}) & := & T^\alpha_B\left(\left(\sigma,\tilde{z}_1+\frac{1-\sigma}{2}\right),\left(\sigma(K_n-1),\tilde{z}_1+\tilde{y}-\frac{1-\sigma}{2}\right)\right),\\
V^\alpha_2(\sigma,\tilde{\gamma}) & := & T^\alpha((\sigma K_n,\tilde{z}_1+\tilde{y}),(\sigma K_n,\tilde{z}_1+\tilde{y}+\tilde{z}_2)).
\end{eqnarray*}
Note that the second passage time in \eqref{cut_in_three} is restricted to $B$ by definition of the skeleton.
We then proceed as in Section \ref{subsec:outline} by studying the mean optimization problem (that is the maximum of the expectations of the three terms in \eqref{cut_in_three})
and estimating the error due to  this approximation.

\medskip

\textcolor{black}{Using  \eqref{def_mintime} with $B=\Z$, and \eqref{eq: homogeneous LPP} with $(x,y)=(0,1)$}, 
%
\begin{eqnarray*}
\Exp \Big( V^\alpha_1(\sigma,\tilde{\gamma}) \Big) & \leq & \frac{4}{r}\tilde{z}_1,
\qquad \Exp \Big(  V^\alpha_2(\sigma,\tilde{\gamma}) \Big)  \leq  \frac{4}{r}\tilde{z}_2
\end{eqnarray*}
\textcolor{black}{On the other hand, by definition of restricted passage times $\tau^\alpha_B$},
\begin{eqnarray*}
\Exp \Big(  U^\alpha_B(\sigma,\tilde{\gamma}) \Big)  & \leq  & K_n\tau^\alpha_B(\sigma,K_{n}^{-1}\tilde{y})\leq K_n\tau^{\rho_n,J_n}(\sigma,K_n^{-1}\tilde{y}),
\end{eqnarray*}
where the last inequality follows from Propositions \ref{prop:induction_gen} and \ref{prop:estimation_gen}.
From the definition  \eqref{ref_time} of $\tau^{\rho,J}$, we get for $i = 1,2$
$$
\frac{4}{r}\tilde{z}_i = \tau^{\overline{\rho}_c,r/4}(0,\tilde{z}_i)
\quad \text{and} \quad
\frac{1}{J_n}\tilde{z}_i = \tau^{\rho_n,J_n} (0,\tilde{z}_i).
$$
Thus we deduce that 
\begin{align}
& \frac{1}{K_n}
\left(
\frac{4}{r}\tilde{z}_1+\frac{4}{r}\tilde{z}_2+K_n\tau^{\rho_n,J_n}(\sigma,K_n^{-1}\tilde{y})
\right)
 \leq  \tau^{\overline{\rho}_c,r/4}(\sigma,y)\nonumber\\
& \qquad \qquad  \qquad \qquad  \qquad 
+  \tau^{\rho_n ,J_n}(\sigma,y)-\tau^{\overline{\rho}_c,r/4}(\sigma,y)
 +  \left|
\frac{4}{r} -  \frac{1}{J_n}
\right|y,
\label{upper_bound_three}
\end{align}
where we used that $ \tilde{z}_1 + \tilde{z}_2 \leq K_n y$.
Define $e_n^{(1)}(\sigma,y)$ as the second line
of the r.h.s. of \eqref{upper_bound_three}. We have thus shown that
$$
\frac{1}{K_n} \max_{(\tilde{z}_1,\tilde{y},\tilde{z}_2)\in\tilde{\Gamma}_n }
\Big\{ 
\Exp \big( V^\alpha_1(\sigma,\tilde{\gamma}) \big) +
\Exp  \big(  U^\alpha_B(\sigma,\tilde{\gamma}) \big) +
\Exp  \big( V^\alpha_2(\sigma,\tilde{\gamma})  \big)
\Big\}
\leq \tau^{\overline{\rho}_c,r/4}( \sigma,y)+e_n^{(1)}(\sigma,y),
$$
where by the definition of $\tau^{\rho,J}$ \eqref{ref_time}, of $\overline{\rho}_c$ \eqref{eq: rhoc} and the convergence of $J_n$ to $r/4$, we deduce the  (locally uniform) convergence 
\margincom{{\color{black} mettre $e_n^{(1)}$ dans l'enonce ?}\textcolor{black}{pourquoi ?}}
$$
\lim_{n \to \infty} e_n^{(1)}(\sigma,.)  = 0.
$$
We conclude by controlling the error thanks to Proposition \ref{prop:fluct_2}, in the same spirit as Proposition \ref{prop:fluct}. 
\end{proof}
\begin{proposition}
\label{prop:fluct_2}
For  $\sigma\in\{-1,1\}$, there exist functions $e_n^{(2)}(\sigma,y)$ such that
\begin{align}
&  \Exp \left( \frac{1}{K_n} \; \max_{(\tilde{z}_1,\tilde{y},\tilde{z}_2)\in\tilde{\Gamma}_n }
\Big\{
V^\alpha_1(\sigma,\tilde{\gamma})
+U^\alpha_B(\sigma,\tilde{\gamma})
+V^\alpha_2(\sigma,\tilde{\gamma})
\Big\}
\right)
 \nonumber\\
& \qquad  -  \max_{(\tilde{z}_1,\tilde{y},\tilde{z}_2) \in \tilde{\Gamma}_n }
\Big\{ \frac{1}{K_n}
\Big[
\Exp \big( V^\alpha_1(\sigma,\tilde{\gamma}) \big)+
\Exp  \big(  U^\alpha_B(\sigma,\tilde{\gamma}) \big)+
\Exp  \big(  V^\alpha_2(\sigma,\tilde{\gamma})  \big)
\Big]  \Big\}
\leq  e_n^{(2)}(\sigma,y),
\label{fluct_2}
\end{align} 
and  $e_n^{(2)}(\sigma,.)$ converges locally uniformly to $0$ on $[\sigma^-,+\infty)$ as $n\to+\infty$.
\end{proposition}
\begin{proof}
We proceed in three steps as in the proof of Proposition \ref{prop:fluct}.\\ \\
{\bf Step 1.} {\it Cutoff.}
We again use a truncation procedure for the service times $Y_{i,j}$ as in step one of Subsection \ref{fluct_ent}.
 Here, we define  $\Prob'_{n,M,y'}$ as
 the distribution of the family
$(Y_{i,j}:\,i\in \textcolor{black}{[-  y', \sigma K_n+  y']},\; j\in[0,y'])$ conditioned on their maximum $M_B(y)$ being $rM$, and $\Prob_{n,M,y'}$ as the distribution of the family $(X_{i,j}:\,i\in \textcolor{black}{[-  \sigma y', \sigma K_n+  y']},\;  j\in[0,y'])$,
where $X_{i,j}$ are i.i.d. random variables, and the law of $X_{i,j}$ is the law of $Y_{i,j}$ conditioned on 
$Y_{i,j}\leq rM$. 
For  simplicity, we will only write $\Prob_M$ and $\Prob'_M$ for these distributions. 

\medskip

\noindent
{\bf Step 2.} {\it Fluctuations under cutoff.}
 Applying Lemma \ref{lemma_concentration} under $\Prob_{n,M,y'}$, we have
 \begin{eqnarray}
 V^\alpha_1(\sigma,\tilde{\gamma}) & \leq & 
 \Exp_M \big( V^\alpha_l(\sigma,\tilde{\gamma}) \big)
 +8M\sqrt{2\tilde{z}_1}Z_1  ,  \nonumber\\
 V^\alpha_2(\sigma,\tilde{\gamma})
  & \leq & \Exp_M \big(  V^\alpha_2(\sigma,\tilde{\gamma}) \big) + 8M\sqrt{2\tilde{z}_2}Z_2, 
  \nonumber\\
 U^\alpha_B(\sigma,\tilde{\gamma}) & \leq & \Exp_M \big( U^\alpha_B(\sigma,\tilde{\gamma}) \big)
 +8M\sqrt{ \sigma K_n+2\tilde{y}}Z_0,
 \label{concentration_three}
 \end{eqnarray}
 where $Z_1$, $Z_2$ and $Z_0$ are independent random variables such that
 $$
 \Prob_{n,M,y}(Z_k\geq t)\leq e^{-t^2},
 $$
for $k\in\{0,1,2\}$. 
We now apply Lemma \ref{cor_max} with $\mathcal A=\tilde{\Gamma}_n$, $\mathcal I=\{1,2,B\}$, 
and for $a=\tilde{\gamma}\in\tilde{\Gamma}_n$, ${\mathcal Y}_{a,1}=V^\alpha_1(\sigma,\tilde{\gamma})$,
${\mathcal Y}_{a,2}=V^\alpha_2(\sigma,\tilde{\gamma})$, ${\mathcal Y}_{a,B}=U^\alpha_B(\sigma,\tilde{\gamma})$,
$V_{a,1}=2\tilde{z}_1$, $V_{a,2}=2\tilde{z}_2$, $V_{a,B}=\sigma K_n+2\tilde{y}$.
Since (see \eqref{def_new_skeleton})
\be\label{nb_new_skeletons}
|\tilde{\Gamma}_n|=\left(
\ba{l}
2+ \lfloor K_n y\rfloor-\sigma^- K_n\\
2
\ea
\right)
\leq K_n^2 (1+y)^2,
\ee
%
%
we obtain
\begin{align}
& 
 \Exp_M \left( \frac{1}{K_n} \; \max_{(\tilde{z}_1,\tilde{y},\tilde{z}_2)\in\tilde{\Gamma}_n }
\Big\{
V^\alpha_1(\sigma,\tilde{\gamma})
+U^\alpha_B(\sigma,\tilde{\gamma})
+V^\alpha_2(\sigma,\tilde{\gamma})
\Big\}
\right)
 \nonumber\\
&\qquad \quad \leq \max_{(\tilde{z}_1,\tilde{y},\tilde{z}_2) \in \tilde{\Gamma}_n }
\Big\{ \frac{1}{K_n}\left[
\Exp_M  \big(  V^\alpha_1(\sigma,\tilde{\gamma}) \big) +
\Exp_M \big(  U^\alpha_B(\sigma,\tilde{\gamma}) \big) +
\Exp_M \big(  V^\alpha_2(\sigma,\tilde{\gamma}) \big)
\right]  \Big\}
\label{fluct_3}\\
&  \qquad \qquad \qquad  \nonumber 
+
8M \frac{ \sqrt{\sigma+2y}}{\sqrt{K_n}} 
\left(\sqrt{3\pi}+
\sqrt{\pi}\sqrt{A}+\sqrt{A}\sqrt{\log|\tilde{\Gamma}_n|}
\right)\nonumber\\
&\qquad \quad \nonumber 
\leq \max_{(\tilde{z}_1,\tilde{y},\tilde{z}_2) \in \tilde{\Gamma}_n }
\Big\{ \frac{1}{K_n}\left[
\Exp  \big( V^\alpha_1(\sigma,\tilde{\gamma}) \big) +
\Exp \big( U^\alpha_B(\sigma,\tilde{\gamma}) \big) +
\Exp \big( V^\alpha_2(\sigma,\tilde{\gamma}) \big)
\right]  \Big\} \\
& \qquad \qquad \qquad 
+  8M
\frac{
\sqrt{\sigma+2y}
}{
\sqrt{K_n}
}\left(
\sqrt{3\pi}+
\sqrt{\pi}\sqrt{A}+\sqrt{A}\sqrt{\log|\tilde{\Gamma}_n|}
\right).
\nonumber
\end{align} 
In the last inequality, we have used the fact that the passage times under $\Prob_M$ are stochastically dominated by the passage times under $\Prob$.\\ \\
{\bf Step 3.} {\it Removing the cutoff.}
As in step three of the proof of Lemma \ref{cor_max}, a coupling argument shows that
the distribution under $\Prob'_{M}$ of any passage time $T$ depending only on the previous set of $Y_{i,j}$ is dominated by the distribution under $\Prob_{M}$ of $T+M$. Therefore 
\begin{eqnarray}
& & \Exp'_M \left( \frac{1}{K_n} \; \max_{(\tilde{z}_1,\tilde{y},\tilde{z}_2)\in\tilde{\Gamma}_n }
\Big\{
V^\alpha_1(\sigma,\tilde{\gamma})
+U^\alpha_B(\sigma,\tilde{\gamma})
+V^\alpha_2(\sigma,\tilde{\gamma})
\Big\}
\right)\nonumber\\
& &\qquad  \quad \nonumber 
\leq \max_{(\tilde{z}_1,\tilde{y},\tilde{z}_2) \in \tilde{\Gamma}_n }
\Big\{ \frac{1}{K_n} \Big(
\Exp \big(  V^\alpha_1(\sigma,\tilde{\gamma}) \big) +
\Exp \big( U^\alpha_B(\sigma,\tilde{\gamma}) \big) +
\Exp \big( V^\alpha_2(\sigma,\tilde{\gamma}) \big)
\Big)  \Big\} \\
&  &\qquad \qquad+ \frac{M}{K_n}+ 8M\frac{
\sqrt{\sigma+2y}
}{
\sqrt{K_n}
}\left(
\sqrt{3\pi}+
\sqrt{\pi}\sqrt{A}+\sqrt{A}\sqrt{\log|\tilde{\Gamma}_n|}
\right).
\nonumber
\end{eqnarray} 
Integrating the above inequality with respect to the distribution of $M_B(y)$ yields \eqref{fluct_2}, with
$$
e_n^{(2)}(\sigma,y)  :=  m \big(  \lfloor K_n y\rfloor (\sigma K_n+2 \lfloor K_n y\rfloor) \big) \, E_n(\sigma,y),
$$
where $m(.)$ satisfies the bound \eqref{bound_expmax}, and
$$
E_n(\sigma,y):= \frac{1}{K_n}+ 8\frac{
\sqrt{\sigma+2y}
}{
\sqrt{K_n}
}\left(
\sqrt{3\pi}+
\sqrt{\pi}\sqrt{A}+\sqrt{A}\sqrt{\log|\tilde{\Gamma}_n|}
\right),
$$
from which one can see that $e_n^{(2)}(\sigma,.)$ converges locally uniformly to $0$.
\end{proof}

\subsection{Proof of Theorems \ref{thm: passage upper} and \ref{th_plateau}}
\label{conclude}
Theorem  \ref{th_plateau} is a consequence of Theorem \ref{thm: passage upper} which we prove now.
Given $\sigma\in\{-1,1\}$, by Theorem \ref{th_sep}, we have 
$$
\tau_\varepsilon(\sigma,y)  =  \lim_{n\to\infty} \mathcal E_\varepsilon \times \Exp
\left( \frac{1}{K_n}T(\sigma K_n, \lfloor K_n y\rfloor ) \right)
= \lim_{n\to\infty}\mathcal E_\varepsilon\left(\tau^\alpha_n(\sigma,y)\right), 
$$
where $\mathcal E_\varepsilon$ stands for the expectation with respect to the disorder
$\alpha$.
Note that the above limit does not follow directly from Theorem \ref{th_sep}, which yields an a.s. limit. 
However, the convergence in Theorem \ref{th_sep} holds also in $L^1$. This follows from a quasi-Gaussian tail estimate for the passage time $T(\sigma K_n, \lfloor K_n y\rfloor)$, obtained from Lemma \ref{lemma_concentration} and a cutoff as in step four of Subsection \ref{fluct_ent}.
%
Let ${\bf G}_n(\sigma)$ be the set of environments $\alpha$ for which $[0,\sigma(K_n-1)]$ is a good block. 
The mean passage time 
can be decomposed as
$$
\mathcal E_\varepsilon[\tau^\alpha_n(\sigma,y)]
=
\mathcal E_\varepsilon\left[
\tau^\alpha_n(\sigma,y)\indicator{{\bf G}_n(\sigma)}
\right]+\mathcal E_\varepsilon\left[
\tau^\alpha_n(\sigma,y)\indicator{{\bf A}\backslash{\bf G}_n(\sigma)}
\right],
$$
where ${\bf A}:=[0,1]^\Z$ is the set of environments.
By Corollary \ref{cor_block} (recall that the function \textcolor{black}{$e_n$} in \eqref{limit_time} does not depend on $\alpha$), the $\limsup$ of the first term is bounded above by $\tau^{\overline{\rho}_c,r/4}( \sigma,y)$. On the other hand, the second term is bounded above by
$$
\mathcal E_\varepsilon\left[ \left(\tau^\alpha_n(\sigma,y)\right)^2\right]^{1/2}
\,
\mathcal P_\varepsilon
\left[
{\bf A}
\backslash
{\bf G}_n(\sigma)
\right]^{1/2}.
$$
The $\mathcal P_\varepsilon$-probability vanishes as $n\to\infty$ by Lemma \ref{lemma_good}, while
the expectation of the squared passage time {\color{black} can be bounded by} 
$$
\mathcal E_\varepsilon \left[ \left(\tau^\alpha_n(\sigma,y)\right)^2\right]
\leq\tau_n(\sigma,y)^2,
$$
where $\tau_n(\sigma,y)$ is defined as \eqref{def_unrestricted} for a homogeneous environment $\alpha(x)\equiv r$ (that is for rate $r$ homogeneous TASEP).
The limit $\tau_n(\sigma,y)\to r^{-1} (\sqrt{\sigma+y}+\sqrt{y})^2$ as $n \to \infty$,  follows from the above remark on $L^1$-convergence of rescaled passage times in Theorem \ref{th_sep}. This implies
$\tau_n(\sigma,y)^2\to r^{-2} (\sqrt{\sigma+y}+\sqrt{y})^4$ as $n\to\infty$. 
We finally get
\be
\label{final_bound_passage}
\tau_\varepsilon({\color{black} \sigma},y)\leq\tau^{\overline{\rho}_c,r/4}( {\color{black} \sigma} ,y),
\ee
for every ${\color{black} \sigma} \in\{-1,1\}$ and $y \geq {\color{black} \sigma}^-$. Since $\tau_\varepsilon$ and $\tau^{\overline{\rho}_c,r/4}$ are homogeneous functions, 
\eqref{lower_bounds} follows  for $\rho=\overline{\rho}_c$.

We now show that
$$
\lim_{\varepsilon\to 0}\rho_c(\varepsilon)=\rho_c(0).
$$
Indeed, by \eqref{critical_density_2} and \eqref{final_bound_passage},
we have $\rho_c\leq\overline{\rho}_c$.
Then, by Proposition \ref{prop: estimation suites}, 
$$
\limsup_{\varepsilon\to 0}{\rho}_c(\varepsilon)
\leq
\limsup_{\varepsilon\to 0}\overline{\rho}_c(\varepsilon)\leq\rho_c(0).
$$
{\color{black} The reversed inequality will be proved by contradiction. Suppose that we have}
$$
\liminf_{\varepsilon\to 0}\rho_c(\varepsilon)<\rho_c(0),
$$
then for some $\varepsilon>0$ we would have $\rho_c(\varepsilon)<\rho_c(0)$, hence
$$
\frac{r}{4}=\max f_\varepsilon=f_\varepsilon[\rho_c(\varepsilon)]
\leq f_{\rm TASEP}[\rho_c(\varepsilon)]<f_{\rm TASEP}[\rho_c(0)]=\frac{r}{4},
$$
where $f_{\rm TASEP}$ denotes the flux of the homogeneous rate 1 TASEP, and the last inequality follows from $\rho_c(\varepsilon)<\rho_c(0)\leq 1/2$.

\subsection{Proof of Theorems \ref{th:passage_dilute} and \ref{th_dilute_limit}}
\label{sec:proof_dilute}

{\color{black} 
Using Proposition \ref{prop:passage_dilute}, we are going to derive the limiting passage time of 
Theorem \ref{th:passage_dilute} and then conclude Theorem \ref{th_dilute_limit}.


\medskip

By coupling with a rate $1$ homogenous TASEP, we get 
\begin{equation}
\label{trivial_bound}
\tau_\varepsilon(1,y)\ge g_1(1,y).
\end{equation}
Combining this with Proposition \ref{prop:passage_dilute}, 
we deduce that 
}
$$
\forall y<y_1, \qquad 
\lim_{\varepsilon\to 0}\tau_\varepsilon(1,y)= g_1(1,y) = 
(\sqrt{1+y}+\sqrt{y})^2,
$$
where $\tau_\varepsilon$ denotes the limiting rescaled passage time. 
%
%
Similarly, one can show that 
$$
{\color{black} \forall  y \in [1,y'_1]}, \qquad 
\lim_{\varepsilon\to 0}\tau_\varepsilon(-1,y)=g_1(-1,y)=(\sqrt{-1+y}+\sqrt{{y}})^2.
$$
As the height profile $h_\varepsilon(t,x)=tk_{\textcolor{white}{\varepsilon}}(\frac{x}{t})$ \eqref{corr_4}
is the inverse of $\tau_\varepsilon(x,y)$ wrt $y$, we obtain
\begin{equation}\label{k_tasep}
\lim_{\varepsilon\to 0}k_\varepsilon(v)=\frac{(1-v)^2}{4},\quad\forall v\in [1-2\rho_1(0),1]\cup [-1,2\rho_1(0)-1].
\end{equation}
Next we use
\begin{equation}\label{duality}
f_\varepsilon(\rho)=\inf_{v}[\rho v+k_\varepsilon(v)].
\end{equation}
For $\rho\not\in[\rho_1(0),1-\rho_1(0)]$, it follows from (\ref{k_tasep}) that the minimum in (\ref{duality}) is achieved for
$v=v_\varepsilon\to 1-2\rho$ as $\varepsilon\to 0$, thus
$$
\lim_{\varepsilon\to 0}f_\varepsilon(\rho)=\rho(1-\rho).
$$
Since the above expression takes value $r/4$ for $\rho\in\{\rho_1(0),1-\rho_1(0)\}$, 
and $f_\varepsilon$ is a concave function with maximum value $r/4$, 
we then necessarily have
$$
\lim_{\varepsilon\to 0}f_\varepsilon(\rho)=\frac{r}{4},\quad\forall \rho\in[\rho_1(0),1-\rho_1(0)].
$$

\begin{appendix}
\section{Proof of Proposition \ref{prop_max}.}
\label{appendix_max}
%
%
To show that the maximum value of the flux is at least $r/4$, we use Definition \textcolor{black}{(\ref{flux_current})} and couple the process $(\eta_t^\alpha)_{t\geq 0}$ with generator \eqref{generator}
with a homogeneous rate $r$ TASEP denoted by $(\eta_t^r)_{t\geq 0}$. 
\textcolor{black}{Lemma \ref{lemma_j_alpha}} shows that $J_x^\alpha(t,\eta^\rho)\geq J_x^r(t,\eta^\rho)$, where $J_x^r$ denotes the current in the homogeneous rate $r$ TASEP.
It is known (see e.g. \cite{sep0}) that
$$
\lim_{t\to+\infty} \frac{1}{t} J_x^r(t,\eta^\rho)=r\rho(1-\rho)
$$
and  it is maximum for 
$\rho=1/2$.

\medskip

\margincom{\color{black} On a perdu l'autre definition du flux.\\ On la rajoute ?\textcolor{black}{bien oblige}}
We now prove that $f(\rho)\leq r/4$ for all $\rho\in[0,1]$.
\textcolor{black}
{
In \cite{bgrs4, bb} it is shown that there exists a closed subset $\mathcal R$  of $[0,1]$ containing $0$ and $1$,
and a family $(\nu^\alpha_\rho)_{\rho\in{\mathcal R}}$ of invariant measures for the disordered TASEP, such that, for every $\rho\in\mathcal R$,
\be\label{flux_line}
f(\rho)=\int j_x^\alpha(\eta)d\nu^{\alpha}_{\rho}(\eta),\quad x\in\Z,
\ee
%
%
where $j_x^\alpha(\eta):=\alpha(x)\eta(x)[1-\eta(x+1)]$, and that $f$ is interpolated linearly outside $\mathcal R$. Note (this follows from stationarity) that the integral in \eqref{flux_line} does not depend on $x$.
}
It is thus enough to consider $\rho\in\mathcal R$.
%
Since the random variables $\alpha(x)$ are i.i.d. and the infimum of their support is $r$,
for $\mathcal P$-a.e. environment $\alpha\in{\bf A}$, 
there exist sequences $(x_N)_{N \geq 1}$, $(y_N)_{N  \geq 1}$ and $(\varepsilon_N)_{N  \geq 1}$ such that 
$\lim_{N\to\infty}x_N=+\infty$, $\lim_{N\to\infty}[y_N-x_N]=+\infty$, $\lim_{N\to\infty}\varepsilon_N=0$, and
\be
\label{interval_r}
r\leq\min_{x=x_N,\ldots,y_N}\alpha(x)\leq\max_{x=x_N,\ldots,y_N}\alpha(x)\leq r+\varepsilon_N .
\ee
{\color{black} The intervals $[x_N,y_N]$ with low rates act as bottlenecks for the particle flux.}
Set 
$$
a_N=\frac{2x_N+y_N}{3},\quad
b_N=\frac{x_N+2y_N}{3},
$$
which satisfy $x_N \leq a_N \leq b_N \leq y_N$ and $b_N-a_N\to+\infty$.
By \eqref{flux_line},
\begin{eqnarray}
f(\rho) =  \frac{1}{b_N-a_N+1}
\sum_{x=a_N}^{b_N}\int_{\bf X}j^\alpha_x(\eta)d\textcolor{black}{\nu^\alpha_\rho}(\eta)
\leq [r+\varepsilon_N] \int_{\bf X}\tilde{j}(\eta)d{\mu}_N (\eta),
\label{flux_inequality} 
\end{eqnarray}
where 
$
\tilde{j}(\eta)=\eta(0)[1-\eta(1)]
$, and 
\be
\label{averaged_measure}
{\mu}_N:= \frac{1}{b_N-a_N+1} \sum_{x=a_N}^{b_N}\textcolor{black}{\theta}_x\textcolor{black}{\nu^\alpha_\rho}.
\ee
\textcolor{black}{where $(\theta_x)_{x\in\Z}$ is the group of spatial shifts, whose
action on particle configurations is defined
by $(\theta_x\eta)(y)=\eta(x+y)$ for every $\eta\in\bf X$ and $y\in\Z$. Similarly, $(\theta_x\alpha)(y)=\alpha(x+y)$ for every $\alpha\in\bf A$ and $y\in\Z$.  If $\varphi$ is a cylinder function of $\bf X$, we set
$\theta_x \varphi:=\varphi\circ\theta_x$. If $\mu$ is a probability measure on $\bf X$, 
$\theta_x\mu:=\mu\circ(\theta_x)$. Finally, the action of $\theta_x$ on the generator $L^\alpha$ given by \eqref{generator} is defined by $(\theta_x L^\alpha)\varphi=\theta_x(L^\alpha \varphi)$ for every cylinder function $\varphi$ on $\bf X$.}\\ \\
The sequence $(\mu_N)_{N\in\N^*}$ of probability \textcolor{black}{measures} on the compact space $\bf X$  is tight. Let $\mu^\star$ be one of its limit points. 
It follows from \eqref{averaged_measure} that $\mu^\star$
is shift invariant, i.e. $\textcolor{black}{\theta_x}\mu^\star = \mu^\star$ for all $x\in\Z$. We claim and prove below that $\mu^\star$ is an invariant measure for the homogeneous TASEP, that is the process with generator \eqref{generator} with $\alpha(x)\equiv 1$. By Liggett's characterization result \cite{lig} for shift-invariant stationary measures, $\mu^\star$ is then of the form
$$
\mu^\star  = \int_{[0,1]}\textcolor{black}{\nu_\rho}\gamma(d\rho).
$$
where $\gamma$ is a probability measure on $[0,1]$, and \textcolor{black}{$\nu_\rho$} is the product Bernoulli measure on $\bf X$ with parameter $\rho$. Thus 
$$
\int_{\bf X}\tilde{j}(\eta)d\mu^\star (\eta)=\int_{[0,1]}\rho(1-\rho)d\gamma(\rho)\leq\frac{1}{4}.
$$
Letting $N\to\infty$ in \eqref{flux_inequality} implies $f(\rho)\leq r/4$.\\ \\
We now prove that $\mu^\star$ is an invariant measure for the homogeneous TASEP. Let $g:{\bf X}\to\R$ be a local function that depends on $\eta$ only through sites $x\in\Z$ such that $|x|\leq\Delta$,
where $\Delta\in\N$. Take $N$ large enough so that
$\Delta<(y_N-x_N)/3$. Notice that  the generator $L^\alpha$ defined in \eqref{generator} satisfies the commutation relation
\be\label{commutation}
\textcolor{black}{\theta_x}  L^{\textcolor{black}{\theta_x} \alpha} f=L^{\alpha} (\textcolor{black}{\theta_x} f).
\ee
It follows that
$$
\int_{\bf X} \textcolor{black}{L^{\textcolor{black}{\theta_x}\alpha}} g\,d(\textcolor{black}{\theta_x}\textcolor{black}{\nu_\rho^\alpha})=\int_{\bf X} L^{\alpha}(\textcolor{black}{\theta_x} g)\,d\textcolor{black}{\nu_\rho^\alpha}=0.
$$
The last equality follows from invariance of $\nu_\rho^\alpha$. 
On the other hand, for $x\in[a_N,b_N]$ and $|y|\leq(y_N-x_N)/3$,  $\tau_x\alpha(y)\in[r,r+\varepsilon_N]$. Let $L$ denote the generator of the homogeneous TASEP on $\Z$,
that is the one obtained from \eqref{generator} when $\alpha(x)\equiv 1$. Since
$$
\Big| \textcolor{black}{L^{\textcolor{black}{\theta_x}\alpha}} g(\eta)-r Lg(\eta) \Big|
\leq  
2 ||g||_\infty \sum_{y=-\Delta-1}^{\Delta}|\textcolor{black}{\theta_x}\alpha(y)-r|,
$$
it follows that
$$
\lim_{N\to\infty} \max_{x=a_N,\ldots,b_N}\sup_{\eta\in{\bf X}}\left|
\textcolor{black}{L^{\textcolor{black}{\theta_x}\alpha}}g(\eta)-r Lg(\eta)
\right|=0.
$$
Hence
$$
\int_{\bf X}Lg(\eta) \, d \mu^\star (\eta)=\lim_{N\to\infty}\int_{\bf X}Lg(\eta)d\mu_N(\eta)
=\lim_{N\to+\infty}
%
\textcolor{black}{\frac{1}{b_N-a_N}\sum_{x=a_N}^{b_N}
\int_{\bf X}\frac{1}{r}L^{\alpha}(\textcolor{black}{\theta_x} g)\,d\nu^\alpha_\rho}
=0
$$
holds for every local function $g$.
%


\textcolor{black}{
\section{Proof of Proposition \ref{prop_monotone}}\label{app:std}
The proof uses the following lemma.
\begin{lemma}\label{lemma_j_alpha}
Let $B$ be a nonempty interval of $\Z$, and $B^\#$ as in \eqref{def_bdiese}
Assume $\alpha$ and $\alpha'$ are two environments such that $\alpha(x)\leq\alpha'(x)$ 
for every $x\in B^\#$. Then
\begin{itemize}
\item[(i)] For every $t\geq 0$, $x\in B$ and $\eta\in \{0,1\}^B$,
\be\label{compare_currents_B}
J^{\alpha,B^\#}_x(t,\eta)\leq J^{\alpha',B^\#}_x(t,\eta),
\ee
where $J^{\alpha,B^\#}_x(t,\eta)$ denotes the rightward current across site $x$ up to time $t$ for the process with generator \eqref{generator_open}.
\item[(ii)] If $B$ is finite, then $j_{\infty,B^\#}(\alpha_{B^\#})\leq j_{\infty,B^\#}(\alpha'_{B^\#})$.
\end{itemize}
\end{lemma}
\begin{proof}[Proof of Proposition \ref{prop_monotone}]
Statement (ii) follows from (i) and \eqref{def_critical_density}. We now prove (i).
Let $U=(U_x)_{x\in\Z}$ be a family of i.i.d. $\mathcal U(0,1)$ random variables, and $V=(V_x)_{x\in\Z}$ be a family of i.i.d.  random variables with distribution $Q$, independent of $U$.   
We set
\be
\label{couple_env}
\alpha_\varepsilon(x)=V_x{\bf 1}_{\{ U_x\leq\varepsilon \} }+{\bf 1}_{ \{ U_x>\varepsilon \}},
\ee
Thus
$\alpha_\varepsilon$ has distribution $Q_\varepsilon$
defined in \eqref{law_disorder}, and since $V_x\leq 1$, 
\be
\label{order_alpha}
\alpha_\varepsilon\leq\alpha_{\varepsilon'}\quad\mbox{if } \varepsilon\leq\varepsilon',
\ee
where the inequality is in the sense of product order. 
Let $\rho\in[0,1]$ and $\eta^\rho\in \bf X$ satisfying \eqref{uniform_profile}. By \eqref{flux_current}, for almost every realization of $(U,V)$,   the following limit holds in probability
$$
f_\varepsilon(\rho)=\lim_{t\to+\infty} \frac{1}{t} J_0^{\alpha_\varepsilon}(\eta^\rho,t)=f_\varepsilon(\rho).
$$
The conclusion then follows from \eqref{order_alpha} and Lemma \ref{lemma_j_alpha}. 
\end{proof}
\begin{proof}[Proof of Lemma \ref{lemma_j_alpha}.]
Statement (ii) follows from (i) and \eqref{flux_current}. We now prove (i). Let $B=[x_1,x_2]\cap\Z$, with 
$x_1,x_2\in\overline{\Z}$ such that $x_1\leq x_2$.
We are going to couple the TASEP's $(\eta^{\alpha}_t)_{t\geq 0}$ and $(\eta^{\alpha'}_t)_{t\geq 0} $ in environments $\alpha, \alpha'$ starting from the same configuration $\eta\in\{0,1\}^B$.
If $x_1 \not = -\infty$,   the initial configuration $\eta$ is extended so that site $x_1-1$ has an infinite stack of particle, and there is no particle to the left of $x_1-1$.
Particles of the extended initial configuration are labelled  increasingly towards the left.
We choose the initial labeling so that the lowest label of a particle in the stack is $1$. Thus
the set of labels is
$$
I=\Z\cap[n,+\infty[,\quad\mbox{where }n:=1-\sum_{x\in B}\eta(x).
$$
The coupled evolution is defined using a Harris type construction, as in \cite{har}, of both TASEP's from a common Poisson point measure on $(0,+\infty)\times\Z\times(0,1)$ with intensity $dt dx{\bf 1}_{(0,1)}(u)du$, where $dx$ is the counting measure on $\Z$.
If $(t,x,u)$ is a Poisson point, and $\beta\in\{\alpha,\alpha'\}$, a particle jumps from $x$ to $x+1$ 
in the configuration $\eta^\beta_.$ if 
\begin{itemize}
\item[(i)] 
$x\in\Z\cap[x_1-1,x_2]$, 
\item[(ii)]
$u\leq\beta(x)$, 
\item[(iii)] 
one of the following holds: either $x=x_1-1$ and $\eta^\beta_{t-}(x_1)=0$, or 
$x=x_2$ and $\eta^\beta_{t-}(x_2)=1$, or $x_1<x<x_2$, $\eta^\beta_{t-}(x)=1$ and $\eta^\beta_{t-}(x+1)=0$.
\end{itemize}
If a jump occurs from the stack, the particle that jumps is the stack particle with the lowest label.
Let $\sigma_i(t)$, resp. $\sigma'_t(i)$,  denote the position at time $t$ of particle $i$ in $\eta^\alpha_t$,
resp. $\eta^{\alpha'}_t$.  
{\color{black} 
The initial conditions are identical, i.e. $\sigma_0(i)=\sigma'_0(i)$ for every $i\in I$. As a consequence of the assumption $\alpha\leq\alpha'$ and of the rules (i)--(iii),  
the order between the particle configurations is preserved by the coupling at any time} 
\be
\label{ordered_particles}
\forall i\in\N,
\qquad\sigma_t(i) \leq \sigma'_t(i) .
\ee
{\color{black} By construction, for $\beta\in\{\alpha,\alpha'\}$, $J_x^\beta(t,\eta)$ is equal to} the highest label of a particle in $\eta^\beta_.$ having left site $x$ by time $t$. Thus \eqref{ordered_particles} implies (i).
\end{proof}
}
\section{Proof of Lemma \ref{lemma_reform}}
\label{appendix_lpp}
Before deriving  Lemma  \ref{lemma_reform}, we first explain a mapping between the restricted passage times in a box $B$ and the TASEP restricted to $B$ with reservoirs.

\subsection{last passage times in  a finite domain}
\margincom{
\textcolor{black}{
Plusieurs modifications dans cette sous-section.
}
}
Let $B:=[x_1,x_2]\cap\Z$. The purpose of this subsection is to give an interpretation of the passage times \eqref{restricted_time} restricted to $B$ in terms of an open disordered TASEP
on $B':=[x_1+1,x_2]\cap\Z$ with generator $L^\alpha_{B}$, see \eqref{generator_open} (recall from \eqref{def_bdiese} and \eqref{def_bprime} that
$(B')^\#=B$).
It is convenient to view the dynamics generated by \eqref{generator_open} as follows.
We add an infinite stack of particles (reservoir) at site $x_1$, and a site $x_2+1$ where the number of particles is not restricted.
Particles enter $B'$ from the stack at $x_1$, and when they leave, they stay at $x_2+1$ forever.
%
%
%
We are going to check that  $T^\alpha_B \big((x_1,0),(i,j) \big)$ has the same distribution as the time when particle $j$ reaches site $i+1$ in the process generated by $L^\alpha_B$,
if the initial state is given by
\be
\label{special_initial}
\sigma_0(j)=x_1\indicator{\{j\geq \textcolor{black}{0}\}}+(x_2+1)\indicator{\{j\leq-1\}}.
\ee
where $\sigma_0(j)$ denotes the initial position of  the particle with label $j$, and particles are numbered increasingly from right to left. 
In fact,  we may define passage times associated with more general labeled initial configurations in $B'$. By this we mean that $\sigma_0$, instead of being defined by \eqref{special_initial},
can be any nonincreasing function $\sigma_0$ from $\Z$ to $[x_1,x_2+1]\cap\Z$.
%
Let
\begin{equation}
\label{new_passage_bulk}
\tilde{B}  := \{ (i,j)\in B\times\Z:\,
i\geq\sigma_0(j) \}
\quad \text{and} \quad 
\bar{B} :=  \{ (i,j)\in B\times\Z:\,  i<\sigma_0(j)  \}.
\end{equation}
For $(i,j)\in B\times\Z$, let $T^\alpha_{B,\sigma_0} (i,j)$
%
%
denote the time at which particle $j$ reaches site $i+1$.
These passage times are determined by the boundary condition
\be\label{passage_boundary}
T^\alpha_{B,\sigma_0} (i,j)=0\mbox{ for }(i,j)\in\bar{B}
\ee
together with the following recursions:
%
%
\begin{equation}
\label{recursion_passage}
T^\alpha_{B,\sigma_0} (i,j) =  \frac{Y_{i,j}}{\alpha(i)}+\max[T^\alpha_{B,\sigma_0} (i-1,j),
T^\alpha_{B,\sigma_0} (i+1,j-1)]
\end{equation}
for $(i,j)\in\tilde{B}$ such that  $x_1<i<x_2$,
 \begin{equation}
 \label{recursion_passage_right}
 T^\alpha_{B,\sigma_0} (i,j) =  \frac{Y_{i,j}}{\alpha(i)}+T^\alpha_{B,\sigma_0} (i-1,j),
 \end{equation}
 for  $(i,j)\in\tilde{B}$ such that $i=x_2$,
 \begin{equation}
 \label{recursion_passage_left}
 T^\alpha_{B,\sigma_0} (i,j) =  \frac{Y_{i,j}}{\alpha(i)}+
 T^\alpha_{B,\sigma_0} (i+1,j-1),
 \end{equation}
for  $(i,j)\in\tilde{B}$ such that $i=x_1$.
In the special case \eqref{special_initial},
we have
\be
\label{old_bulk}
\tilde{B}  =  [x_1,x_2]\times\N
\quad \text{and} \quad 
\bar{B} = [x_1,x_2]\times(\Z\setminus\N).
\ee
By plugging \eqref{old_bulk} into \eqref{recursion_passage}, one recovers 
$$
T^\alpha_{B,\sigma_0}(i,j)=T^\alpha_B((x_1,0),(i,j)).
$$
where the r.h.s. was defined in \eqref{restricted_time}.
For notational simplicity, in the \textcolor{black}{sequel of this subsection}, we omit dependence on \textcolor{black}{$\alpha$}, $B$ and $\sigma_0$, and write $T(i,j)$ \textcolor{black}{instead of $T^\alpha_{B,\sigma_0}(i,j)$}.
The position of particle $j$ at time $t$, denoted by $\sigma_t(j)\in[x_1,x_2+1]$, is given by
\be
\label{position_time}
\sigma_t(j)  = \left\{
\ba{lll}
x_1 & \mbox{if} & T(x_1,j)>t\\
x_2+1 & \mbox{if} & T(x_2,j)\leq t\\
i\in[x_1+1,x_2]\cap\Z & \mbox{if} & T(i-1,j)\leq t<T(i,j)
\ea
\right.
\ee
\be
\label{time_position}
T(i,j)  =  \sup\{t\geq 0:\quad \sigma_t(j)\leq i\}.
\ee
The particle process $(\sigma_t)_{t\geq 0}$ is equivalent to the {\color{black} following growing cluster process:}
%
$$
\mathcal C_t:=\{(i,j)\in [x_1,x_2]\times\Z:\,T(i,j)\leq t\}=\{(i,j)\in B\times\Z:\,i<\sigma_t(j)\}
$$
with initial state $\mathcal C_0=\bar{B}$. One can proceed as in \cite{sepcg} to show that both processes are Markovian and that the undistinguishable particle process $(\eta_t)_{t\geq 0}$ 
{\color{black} defined by}
%
\be\label{occupation}
\eta_t(x):=\sum_{j\in\Z}\indicator{\{\sigma_t(j)=x\}}
\ee
is Markov with generator $L^\alpha_{B}$.

\subsection{Proof of Lemma \ref{lemma_reform}}
\  \\

{\color{black}

\noindent
\textbf{Step 0: proof of \eqref{def_mintime}.} 
Let $\Omega$ denote the set of sequences $(Y_{i,j})_{i,j\in\Z\times\N}$ equipped with
the product $\sigma$-algebra and the product $\mathcal E(1)$ probability measure. This measure is invariant for the family of shift operators $(\Theta_{n},\,n\in\N)$ defined on $\Omega$ by 
\be\label{def_bigshift}
(\Theta_{n}Y)_{i,j}:=Y_{i,j+n}.
\ee
We may view $T^\alpha_B((x,y),(x',y'))$ defined by \eqref{restricted_time} as a function on $\Omega$.
For $x_0\in B$, let 
$$
\widetilde{T}^\alpha_{B,x_0}(n):=T^\alpha_{B}(x_0,n).
$$
By \eqref{restricted_time}, the above function on $\Omega$ satisfies the superadditivity property
\begin{eqnarray}\label{superadditivity}
\widetilde{T}^\alpha_{B,x_0}(n+m) & \geq & T^\alpha_B((x_0,0),(x_0,n))+T^\alpha_B((x_0,n),(x_0,n+m))\nonumber\\
& = & \widetilde{T}^\alpha_{B,x_0}(n)+\widetilde{T}^\alpha_{B,x_0}(m)\circ\Theta_n.
\end{eqnarray}
Note that the last term on the third line of \eqref{superadditivity} coincides with the last term on the second line because the disorder is uniform along a column. Therefore a {\it quenched} subadditivity property holds for vertical passage-times, whereas it would not be true (nor would the stationarity be true) along other directions.
This superadditivity combined with Kingman's subadditive ergodic theorem (see \cite{kin}) implies \eqref{def_mintime}.\\ \\
}

\noindent
{\bf Step 1.} We prove that  definition \eqref{def_mintime} does not depend on $x_0$. Let $x_0,x'_0\in B$ with $x_0<x'_0$. Then \eqref{restricted_time} implies
\be\label{compare_times}
T^{\textcolor{black}{\alpha}}_B \big((x_0,0),(x_0,m+x'_0-x_0) \big)
\geq T^{\textcolor{black}{\alpha}}_B \big( (x'_0,0),(x'_0,m) \big) 
\geq T^{\textcolor{black}{\alpha}}_B \big((x_0,x'_0-x_0),(x_0,m) \big).
\ee
Since the sequence $(Y_{i,j}:\,i\in\Z,j\geq 0)$ is stationary with respect to shifts of $j$, the expectation of the last quantity is equal to that of
$T_B \big( (x_0,0),(x_0,m- x'_0 + x_0) \big)$. Thus taking expectations, dividing by $m$ and letting $m\to\infty$ yields the result.

\medskip

\noindent
{\bf Step 2.} {\it Proof of \eqref{max_current_lpp}.}
Given this statement, let us denote by $J^{\alpha,B}_x (t,\eta_0)$ the current up to time $t$ across site $x\in B'$, in the open system on $B'$, when starting from $\eta_0$. Assume $\eta_0$ is the occupation configuration associated with $\sigma_0$ via \eqref{special_initial}--\eqref{occupation}. Then it is clear that
$$
J^{\alpha,B}_x(t,\eta_0)=\min\{j\in\Z:\,T^{\alpha}_B(x,j)>t\}, 
$$
which implies the $\Prob$ a.s. limit
\begin{eqnarray*}
\lim_{t\to\infty}\frac{1}{t}J^{\alpha,B}_x(t,\eta_0) & = & \frac{1}{T_{\infty,B}}
 =  \lim_{t\to\infty}\Exp  \left( \frac{1}{t}J^{\alpha,B}_x(t,\eta_0)  \right)
 =   \lim_{t\to\infty} \Exp  \left( \frac{1}{t}\int_0^t j_x^{\alpha,B}(\eta_s)ds \right)\\
& = & \lim_{t\to\infty}\int j_x^{\alpha,B}(\eta)d\nu_t(\eta) 
=  \int j_x^{\alpha,B}(\eta)d\nu^\alpha_{B}(\eta),
\end{eqnarray*}
where
$$
j_x^{\alpha,B}(\eta)=\left\{
\ba{lll}
\alpha(x)\eta(x)[1-\eta(x+1)] & \mbox{if} & \textcolor{black}{x_1+1<x\leq x_2-1},\\
\alpha(x_2)\eta(x_2) & \mbox{if} & \textcolor{black}{x=x_2},\\
1-\eta(x_1+1) & \mbox{if} & \textcolor{black}{x=x_1+1},
\ea
\right.
$$
and 
$$
\nu_t:=\frac{1}{t}\int_0^t\delta_{\eta_0}e^{s L^\alpha_{B}}ds .
$$
The second equality follows from the fact that the family of random variables
 $(\frac{1}{t} J_x^{\alpha,B}(t,\eta_0))_{t\geq 0}$ is uniformly integrable, because 
 $\big( J_x^{\alpha,B}(t,\eta_0) \big)_{t\geq 0}$ is dominated in distribution by a Poisson random variable with parameter $t$.
The last equality follows from the fact that $\nu_t$ converges to the invariant measure $\nu^\alpha_{B}$ 
as $t$ tends to infinity.  \qed
%
%
\margincom{CH 31/10: enlev\'e l'appendice C preuve Lemme 2.1}
\section{Proof of Lemma \ref{cor_max}}\label{app:proof_cor_max}
\margincom{CH 31/10: ajout\'e preuve Lemme 7.3}
The proof of Lemma \ref{cor_max} relies on the following elementary estimates.
\begin{lemma}\label{lemma_normal}
\  \\
(i)
Let $Y$ be a random variable such that
$\Prob(Y\geq t)\leq Ce^{-t^2/V}$
for all $t\geq 0$, where $C\geq 1$ and $V>0$. Then, we have
$$
Y=\sqrt{V \log C}+\sqrt{V}X,
$$
where $\Prob(X\geq t)\leq e^{-t^2}$.

\medskip
\noindent
(ii) There exists a positive constant $A$ such that the following holds.
Let $(X_k)_{k=1,\ldots,n}$ be \textcolor{black}{independent} random variables such that
$\Prob(X_k\geq t)\leq e^{-t^2}$  for all $t\geq 0$,
and $(V_k)_{k=1,\ldots,n}$ be nonnegative numbers. 
Then 
$$
\sum_{k=1}^n \sqrt{V_k}X_k=\sqrt{\pi}\sum_{k=1}^n \sqrt{V_k}+\left(A\sum_{k=1}^n V_k\right)^{1/2}Z ,
$$
where $Z$ is a r.v. such that $\Prob(Z\geq t)\leq e^{-t^2}$  for all $t\geq 0$.
%
%
\end{lemma}
\textcolor{black}{
\begin{remark}
The random variables $Y$ and $Y_k$ in Lemma \ref{lemma_normal} need not have a definite sign.
\end{remark}
}
%
%
%
\begin{proof}[Proof of Lemma \ref{lemma_normal}]
Assertion (i) follows from an immediate computation. To obtain (ii) we note that, for $\theta\geq 0$, 
\begin{eqnarray*}
\Exp \big( e^{\theta X_k} \big)  \leq  1+\int_0^{+\infty}\theta e^{\theta t}\Prob(X_k\geq t)dt
\leq 1+\theta e^{\theta^2/4}\int_{-\theta/2}^{+\infty}e^{-t^2}dt\leq1+\sqrt{\pi}\theta e^{\theta^2/4} \, .
\end{eqnarray*}
Setting $Y_k=X_k-\sqrt{\pi}$, we have, for $\theta\geq 0$,
$$
\Lambda(\theta) := \log \Exp \big( e^{\theta Y_k} \big) 
\leq  \log \left[
1+\sqrt{\pi}\theta e^{\theta^2/4}
\right]-\sqrt{\pi}\theta.
$$
Thus there exists $A>0$ such that $\Lambda(\theta)\leq A\theta^2/4$ for $\theta\geq 0$. Hence, by independence of the random variables $X_k$, we get
$$
\log \Exp \left[ \exp \left( \theta \left(
\sum_{k=1}^n \sqrt{V_k}X_k-\sqrt{\pi}\sum_{k=1}^n\sqrt{V_k} \right) 
\right) \right]
\leq \frac{A}{4}{\theta^2}\sum_{k=1}^n V_k.
$$
{\color{black} The estimate on the tail of $Z$ follows by an exponential Markov inequality.}
\end{proof}


\begin{proof}[Proof of Lemma \ref{cor_max}]
By (ii) of Lemma \ref{lemma_normal}, for every $a\in\mathcal A$, we have
\be\label{normal_1}
\sum_{i\in\mathcal I} {\mathcal Y}_{a,i}=
\sum_{i\in\mathcal I}\Exp \big( {\mathcal Y}_{a,i} \big)+\sqrt{\pi}\sum_{i\in\mathcal I}\sqrt{V_{a,i}}+\left(
A\sum_{i\in\mathcal I}V_{a,i}
\right)^{1/2}Z_a,
\ee
where
%
$Z_a$ is a random variable satisfying $\Prob(Z_a\geq t)\leq e^{-t^2}$ for all $t\geq 0$. 
On the other hand, by Cauchy-Schwarz inequality,
\be
\label{cs_var}
\sum_{i\in\mathcal I}\sqrt{V_{a,i}}\leq\sqrt{|\mathcal I|}\left(\sum_{i\in\mathcal I}V_{a,i}\right)^{1/2}.
\ee
Thus, for every $a\in\mathcal A$,
\be\label{normal_2}
\sum_{i\in\mathcal I} {\mathcal Y}_{a,i}\leq m+(A\,V)^{1/2}Z^+_a,
\ee
where
\begin{eqnarray*}
V  := \max_{a\in\mathcal A}\sum_{i\in\mathcal I}V_{a,i}
\quad \text{and} \quad
m  :=  \sqrt{\pi}\sqrt{|\mathcal I|} \sqrt{V}
+\max_{a\in\mathcal A} \; \sum_{i\in\mathcal I} \Exp \big( Y_{a,i} \big) .
\end{eqnarray*}
\textcolor{black}{
(note that \eqref{normal_2} may be false with $Z_a$ instead of $Z_a^+$ if $Z_a<0$,
because the last term on the r.h.s. of \eqref{normal_1} is then greater than $(AV)^{1/2}Z_a$).
}
Next, for any $t\geq 0$, we have
\begin{eqnarray*}
\Prob\left(
\max_{a\in\mathcal A}\sum_{i\in\mathcal I} {\mathcal Y}_{a,i}\geq m+t
\right) & \leq & \Prob\left(
\bigcup_{a\in\mathcal A}\left\{
\sum_{i\in\mathcal I} {\mathcal Y}_{a,i}\geq m+t
\right\}
\right)\\
& \leq & 
\sum_{a\in\mathcal A}
\Prob
\left(
(A\,V)^{1/2}Z^+_a\geq t
\right)
\leq |\mathcal A| \, e^{-\frac{t^2}{AV}} .
\end{eqnarray*}
It follows from (i) of Lemma \ref{lemma_normal} that
$$
\max_{a\in\mathcal A}\sum_{i\in\mathcal I} {\mathcal Y}_{a,i}= m+\sqrt{\log |\mathcal A|}(A\,V)^{1/2} +(A\,V)^{1/2}Z,
$$
where $Z$ is a random variable satisfying  $\Prob(Z\geq t)\leq e^{-t^2}$ for all $t\geq 0$. The result then follows from
$$
\Exp(Z)\leq\Exp(Z^+)\leq\int_0^{+\infty}\Prob(Z^+\geq t)dt\leq \int_0^{+\infty}e^{-t^2}dt=\sqrt{\pi}.
$$
\end{proof}
\end{appendix}
\end{document}